\documentclass[11pt,letterpaper]{amsart}
\usepackage[T1]{fontenc}
\usepackage[latin9]{inputenc}
\usepackage{geometry}
\usepackage{wrapfig}
\usepackage{multicol}
\usepackage{enumerate}   
\usepackage{graphicx}
\usepackage{soul}
\usepackage{xcolor}
\usepackage{amssymb}

\newtheorem{theorem}{Theorem}[section]
\newtheorem{lemma}[theorem]{Lemma}
\newtheorem{proposition}[theorem]{Proposition}
\newtheorem{corollary}[theorem]{Corollary}
{ \theoremstyle{definition}
\newtheorem{definition}[theorem]{Definition}}
{ \theoremstyle{remark}
\newtheorem{remark}[theorem]{Remark}}
\usepackage{placeins}
\usepackage{cite}
\usepackage{caption}
\usepackage{enumerate}
\usepackage{afterpage}
\usepackage{enumitem}
\usepackage{bmpsize}
\usepackage{hyperref}
\usepackage{tabu}
\usepackage{enumitem}   
\numberwithin{equation}{section}

\usepackage{tikz}

\usetikzlibrary{matrix,graphs,arrows,positioning,calc,decorations.markings,shapes.symbols}

\definecolor{dullmagenta}{rgb}{0.4,0,0.4} % #660066

\definecolor{darkblue}{rgb}{0,0,0.4}

\title{KMT coupling for random walk bridges}
\date{\today}
\author{Evgeni Dimitrov}
\author{Xuan Wu}

\begin{document}

\maketitle 

\begin{abstract}
In this paper we prove an analogue of the Koml{\'o}s-Major-Tusn{\'a}dy (KMT) embedding theorem for random walk bridges. The random bridges we consider are constructed through random walks with i.i.d jumps that are conditioned on the locations of their endpoints. We prove that such bridges can be strongly coupled to Brownian bridges of appropriate variance when the jumps are either continuous or integer valued under some mild technical assumptions on the jump distributions. Our arguments follow a similar dyadic scheme to KMT's original proof, but they require more refined estimates and stronger assumptions necessitated by the endpoint conditioning. In particular, our result does not follow from the KMT embedding theorem, which we illustrate via a counterexample.
\end{abstract}

\tableofcontents

\section{Introduction and main results}\label{Section1}  Let $X$ be a random variable with $\mathbb{E}[X]=0$ and $\mathbb{E}[X^2] =1$. Suppose that $X_1, X_2,\dots$ is an i.i.d sequence of random variables with the same law as $X$ and let $S_n = X_1+X_2+\cdots+X_n$ for $n \geq 1$. A classical problem in probability theory, called the {\em embedding problem}, asks to construct the process $\{S_{m} \}_{m = 1}^n$ and a standard Brownian motion $(B_t)_{ t \geq 0}$ on the same probability space so that 
\begin{equation}
\Delta_n = \max_{1 \leq k \leq n} | S_k - B_k|
\end{equation}
grows as slowly as possible in $n$. The first major results about the above embedding problem, or {\em strong approximation/coupling problem}, were obtained in the works of Skorokhod \cite{Skorohod1, Skorohod2} and Strassen \cite{Strassen}, who showed that if $\mathbb{E}[X^4] < \infty$ then with high probability
$$\Delta_n = O(n^{1/4} (\log n )^{1/2} (\log \log n)^{1/4}).$$
 In fact, this rate of growth was shown to be optimal under the fourth moment assumption in \cite{Kiefer}. For more than a decade the above rate for strong approximation was the best available result, and the method of obtaining it is now known as the {\em Skorokhod embedding}. For a more detailed account of the history of the Skorokhod embedding and its various applications we refer the reader to the comprehensive survey \cite{JO} and the monograph \cite{CR}. 

Nearly fifteen years after Skorokhod's original work,  Koml{\'o}s, Major and Tusn{\'a}dy showed using completely different techniques that one can achieve $\Delta_n = O(\log n )$ for the rate of strong coupling, provided that $X$ has a finite moment generating function in a neighborhood of zero \cite{KMT1,KMT2}. The construction used to achieve this celebrated result is now referred to as the {\em KMT approximation or coupling}. The results in \cite{Bartfai}, see \cite{Zai02}, show that unless $X$ is normally distributed the $\log n$ rate of approximation is optimal. Since its inception, the KMT coupling has become an invaluable tool in probability theory and statistics, see e.g. \cite{CHall, CR, SW}.

In the last few decades, the KMT approximation has been extended in many different directions. We discuss a few of them here, remarking that the list is very far from complete. A multidimensional version of the KMT coupling was proved in \cite{Einmahl} and the best result was later obtained in \cite{Zai96, Zai98}. See \cite{GZ} for more on the history and references regarding the KMT approximation for $X \in \mathbb{R}^d$. \cite{Sak} generalized and essentially sharpened the KMT results in the case of non-identically distributed independent random variables, see also \cite{Sak11, Shao95} and the references therein. Somewhat more recently, \cite{Cha12} proposed a new proof of the KMT result for the simple random walk via {\em Stein's method}. The main motivation of \cite{Cha12}, as admitted by the author, was to gain a more conceptual understanding of the KMT result so that it could be generalized to cases for sums of dependent random variables. Using different techniques, \cite{BLW} extended the KMT coupling for a large class of dependent stationary processes, successfully breaking away from the independent variables setting.\\

In the present paper we consider a different, albeit related problem to the embedding problem above, which we now describe. Let $\{S_m^{(n,z)}\}_{m = 1}^n$ denote the random process with law equal to that of the random walk $\{S_m\}_{m = 0}^n$ conditioned on $S_n = z$. In order for the latter law to be well-defined we assume one of the following situations.
\begin{itemize}
\item {\bf Continuous jumps.} There are constants $\alpha \in [-\infty, \infty)$ and $\beta \in (\alpha, \infty]$ such that $X$ is a continuous random variable with density $f_X(\cdot)$ such that $f_X(\cdot)$ is positive and continuous on $(\alpha, \beta)$ and zero outside of this interval. Under this assumption the process $\{S_m^{(n,z)}\}_{m= 1}^n$  makes sense for all $n \geq 1$ and $z \in L_n = (n\alpha, n \beta)$.
\item  {\bf Discrete jumps.} There are constants $\alpha \in \mathbb{Z} \cup \{-\infty\}$ and $\beta \in ((\alpha, \infty] \cap \mathbb{Z}) \cup \{\infty \}$ such that $X$ is an integer-valued random variable with probability mass function $p_X(\cdot)$ such that $p_X(\cdot)$ is positive on $(\alpha - 1, \beta + 1) \cap \mathbb{Z}$ and zero for all other values. Under this assumption the process $\{S_i^{(n,z)}\}_{i = 1}^n$  makes sense for all $n \geq 1$ and $z \in L_n = (n \alpha - 1, n \beta + 1) \cap \mathbb{Z}$.
\end{itemize}
In the case of continuous jumps we call the process $\{S_m^{(n,z)}\}_{m = 1}^n$ a {\em continuous random walk bridge} between the points $(0,0)$ and $(n,z)$. Similarly, in the case of discrete jumps we call the process $\{S_m^{(n,z)}\}_{m = 1}^n$ a {\em discrete random walk bridge} between the points $(0,0)$ and $(n,z)$. As a natural extension we define $S_t^{(n,z)}$ for non-integer $t$ by linear interpolation, i.e. if $t \in (m, m+1)$ we have
$$S_t^{(n,z)} = (m+ 1 - t) \cdot S_m^{(n,z)} +  (t - m)\cdot S_{m+1}^{(n,z)}.$$

Our main goal in this paper is to demonstrate that given a reference slope $p \in (\alpha, \beta)$ and $z$, which is close to $np$, we can construct a probability space that supports the process $\{S_t^{(n,z)} \}_{t \in [0,n]}$ and a suitable Brownian bridge $B^{(n,z)}_t$ conditioned on $B^{(n,z)}_0 = 0$ and $B^{(n,z)}_n = z$ such that 
$$ \sup_{0 \leq t \leq n} | S_t^{(n,z)} - B^{(n,z)}_t| = O(\log n)$$ 
with exponentially high probability. In particular, we are interested in establishing the above statement under general conditions on the density $f_X(\cdot)$ and the probability mass function $p_X(\cdot)$ in the continuous and discrete case respectively. 

Somewhat surprisingly, despite its inherent probabilistic interest and its direct connection to the well studied problem of KMT approximations, the problem of finding strong couplings between random walk bridges and Brownian bridges has received very little attention. We believe that the present paper is the first one that considers this problem for general jump distributions. To the authors' knowledge, the only case of the above setup that has been previously considered is when $X$ is a Bernoulli random variable. The latter result can be found in \cite[Theorem 6.3]{LF} and \cite[Theorem 4.1]{Cha12} for the case $p = 1/2$ (in both papers the authors consider the case when $\mathbb{P}(X = 1) = \mathbb{P}(X = -1) = 1/2$ and $p = 0$, but the latter is equivalent to the Bernoulli case and $p = 1/2$ after a simple affine transformation). For arbitrary $p \in (0,1)$ the result was proved in \cite[Theorem 8.1]{CD}. \\

Before we turn to our results, we introduce a bit of notation. If $W_t$ denotes a standard one-dimensional Brownian motion and $\sigma > 0$, then the process
$$B^{\sigma}_t = \sigma(W_t - t W_1), \hspace{5mm} 0 \leq t \leq 1,$$
is called a {\em Brownian bridge (conditioned on $B_0 = 0, B_1 = 0$)  with variance $\sigma^2$.} In the following two statements we present our main results about the random processes $\{S_m^{(n,z)}\}_{1\leq m \leq n}$ when the jump distribution $X$  is continuous and discrete respectively. We forgo stating the results in their full generality as this requires more notation. We refer the reader to Theorems \ref{S2ContKMT} and \ref{S2DiscKMT} in the main body of text for the more general formulations as well as to Section \ref{Section8} for the proofs of the two theorems below.
\begin{theorem}\label{S1ContKMT}
Suppose that $X$ is a continuous random variable with a density function $f_X(\cdot)$. Suppose that the support of $f_X$ is a compact interval $[\alpha, \beta] \subset \mathbb{R}$ and that $f_X$ is continuously differentiable and positive on $(\alpha, \beta)$ with a bounded derivative. Then for every $b > 0$ and $p \in (\alpha, \beta)$, there exist constants $0 < C, a, \alpha' < \infty$ (depending on $b$, $p$ and the function $f_X(\cdot)$) such that the following holds. For every positive integer $n$, there is a probability space on which are defined a Brownian bridge $B^\sigma$ with variance $\sigma^2 = \sigma^2_p$ that explicitly depends on $p$ and $f_X(\cdot)$ and a family of processes $S^{(n,z)} $ for $z\in L_n = (n\alpha, n \beta)$ such that
\begin{equation}\label{S1ContKMTE}
\mathbb{E}\left[ e^{a \Delta(n,z)} \right] \leq C e^{\alpha' (\log n)} e^{b|z- p n|^2/n},
\end{equation}
where $\Delta(n,z) = \Delta(n,z,B^{\sigma}, S^{(n,z)})=  \sup_{0 \leq t \leq n} \left| \sqrt{n} B^\sigma_{t/n} + \frac{t}{n}z - S^{(n,z)}_t \right|.$
\end{theorem}

\begin{theorem}\label{S1DiscKMT}
Suppose that $X$ is an integer valued random variable with probability mass function $p_X(\cdot)$. Suppose that $\alpha, \beta \in \mathbb{Z}$ with $\alpha < \beta$ are such that $\mathbb{P}(X \in [\alpha, \beta ]) = 1$ and $p_X(x) > 0$ for all $x \in \mathbb{Z} \cap [\alpha, \beta]$. Then for every $b > 0$ and $p \in (\alpha, \beta)$, there exist constants $0 < C, a, \alpha' < \infty$ (depending on $b$, $p$ and $p_X(\cdot)$) such that the following holds. For every positive integer $n$, there is a probability space on which are defined a Brownian bridge $B^\sigma$ with variance $\sigma^2 = \sigma^2_p$ that explicitly depends on $p$ and $p_X(\cdot)$ and a family of processes $S^{(n,z)} $ for $z \in L_n = (n \alpha - 1, n \beta + 1) \cap \mathbb{Z}$ such that
\begin{equation}\label{S1DiscKMTE}
\mathbb{E}\left[ e^{a \Delta(n,z)} \right] \leq C e^{\alpha' (\log n)} e^{b|z- p n|^2/n},
\end{equation}
where $\Delta(n,z) = \Delta(n,z,B^{\sigma}, S^{(n,z)})=  \sup_{0 \leq t \leq n} \left| \sqrt{n} B^\sigma_{t/n} + \frac{t}{n}z - S^{(n,z)}_t \right|.$
\end{theorem}
\begin{remark}
From  Theorems \ref{S1ContKMT} and \ref{S1DiscKMT} applied to $b = 1$ and Chebyshev's inequality one readily observes that there are constants $M,K,\lambda > 0$ depending on $a, \alpha'$ and $C$ such that if $z = np$ then
\begin{equation}
\mathbb{P} \left (\Delta(n,z) \geq M \log n + x \right) \leq K e^{-\lambda x}.
\end{equation}
\end{remark}

As mentioned before, Theorems \ref{S1ContKMT} and \ref{S1DiscKMT} are representative of the more general Theorems \ref{S2ContKMT} and \ref{S2DiscKMT} given in Sections \ref{Section2.1} and \ref{Section2.2} respectively. The latter are formulated for random variables $X$ whose density $f_X$ satisfies a certain set of Assumptions C1-C6 in the continuous case, or whose mass function $p_X$ satisfies a certain set of Assumptions D1-D5 in the discrete case. In Section \ref{Section2.3} we give a brief description of the significance of these assumptions. Our approach for proving Theorems \ref{S2ContKMT} and \ref{S2DiscKMT}, developed in Sections \ref{Section5} and \ref{Section6}, is inspired by the proof of \cite[Theorem 6.3]{LF}, which is based on an inductive dyadic construction in the same spirit as KMT's original work \cite{KMT1, KMT2}. The main technical challenges lie in obtaining detailed asymptotic estimates for the distributions of $S_n$ and $S^{(n,z)}_{n/2}$, which are presented in Sections \ref{Section3} and \ref{Section4}. Since we are dealing with generic distributions, the asymptotic statements we need are notably harder to obtain than those in \cite{LF}, which deals with the Bernoulli case. Furthermore, in the process of establishing our results we obtain numerous constants that depend on $f_X$ in the continuous and on $p_X$ in the discrete case. We quantify the dependence of these constants on $f_X$ and $p_X$ through various observables of the latter, which further complicates our arguments. The purpose of this quantification is for example to show that the coupling constants $C,a,\alpha'$ in Theorems \ref{S1ContKMT} and \ref{S1DiscKMT} can be chosen uniformly even if $f_X$ or $p_X$ are allowed to depend on some external parameter or $n$, see also Remark \ref{S2Remark}. Obtaining such a uniformity is important for some of the applications we have in mind and a representative example is given in Section \ref{Section8.3}.

It is worth noting that the random walk bridge is a less well-behaved object than the random walk itself, because of the possibility of conditioning on an atypical endpoint. The latter motivates the introduction of the (rather technical) Assumptions C6 and D5 in Section \ref{Section2}, which are novel to our setting and did not appear in KMT's original work \cite{KMT1,KMT2}. In Section \ref{Section7.1} we discuss some easy to check conditions, under which Assumptions C6 and D5 would follow. Moreover, in Section \ref{Section7.2} we construct an example of a discrete random walk bridge, such that the jump distribution satisfies the conditions of \cite{KMT1,KMT2} but for which our coupling result fails. This example illustrates that one necessarily needs to impose stronger assumptions when dealing with random walk bridges compared to random walks, and in particular shows that our result are not a consequence of \cite{KMT1,KMT2}. It is quite possible that one can relax or remove some of the assumptions we make, but one would need to implement different arguments than the ones we present. We believe that it may be possible to prove the results of the present paper using Stein's method, similarly to the proof of \cite[Theorem 4.1]{Cha12} in the Bernoulli case. The immediate obstacle in generalizing the arguments of that paper, which the author also acknowledges, is the difficulty of finding general smoothening techniques that automatically generate Stein coefficients. Nevertheless, it would be nice to have a less technical derivation of our results using such ideas.\\

We end this section with a brief discussion of the possible applications of our results, specifically to integrable probability, which goes to our initial motivation for considering the present problem. There is a large class of stochastic integrable models that naturally carry the structure of random non-intersecting paths with some {\em Gibbsian} resampling invariance. To give a concrete example, one can consider the case of $a$ random walks with jump size $X$ satisfying $\mathbb{P}(X = 0) = \mathbb{P}(X = 1) = 1/2$. If the walks are started at $j-1$, $1 \leq j \leq a$ and conditioned to not intersect in the time interval $[0, b+c]$, and end at $c - b + j -1$ at time $b+c$ then the trajectories of the walks give rise to $a$ random up-right paths. This model has a natural interpretation as a uniform random lozenge tiling of the $a \times b \times c$ hexagon, see Figure \ref{Fig1}. 
\begin{figure}[h]
\centering
\begin{tikzpicture}[>=stealth,
brxy/.style={fill=yellow!30!white, draw=yellow!30!black},
bryz/.style={fill=blue!18!white, draw=blue!30!black},
brxz/.style={fill=red!60!white, draw=red!30!black},
pth/.style={very thick, draw=black},
intpt/.style={circle, draw=white!100, fill=purple!70!black, very thick, inner sep=1pt, minimum size=2.5mm},
scale=0.7
]
\def\brxy(#1:#2:#3){
\filldraw[brxy]
(${(#1)+0}*(1,0) + {(#2)+0}*({sqrt(2)*cos(deg(pi/4))},{sqrt(2)*sin(deg(pi/4))}) + {(#3)-0.5}*(0,1)$) --
(${(#1)+1}*(1,0) + {(#2)+0}*({sqrt(2)*cos(deg(pi/4))},{sqrt(2)*sin(deg(pi/4))}) + {(#3)-0.5}*(0,1)$) --
(${(#1)+1}*(1,0) + {(#2)+1}*({sqrt(2)*cos(deg(pi/4))},{sqrt(2)*sin(deg(pi/4))}) + {(#3)-0.5}*(0,1)$) --
(${(#1)+0}*(1,0) + {(#2)+1}*({sqrt(2)*cos(deg(pi/4))},{sqrt(2)*sin(deg(pi/4))}) + {(#3)-0.5}*(0,1)$) -- cycle;
}
\def\bryz(#1:#2:#3){
\filldraw[bryz]
(${(#1)+0}*(1,0) + {(#2)+0}*({sqrt(2)*cos(deg(pi/4))},{sqrt(2)*sin(deg(pi/4))}) + {(#3)-0.5}*(0,1)$) --
(${(#1)+0}*(1,0) + {(#2)+1}*({sqrt(2)*cos(deg(pi/4))},{sqrt(2)*sin(deg(pi/4))}) + {(#3)-0.5}*(0,1)$) --
(${(#1)+0}*(1,0) + {(#2)+1}*({sqrt(2)*cos(deg(pi/4))},{sqrt(2)*sin(deg(pi/4))}) + {(#3)+0.5}*(0,1)$) --
(${(#1)+0}*(1,0) + {(#2)+0}*({sqrt(2)*cos(deg(pi/4))},{sqrt(2)*sin(deg(pi/4))}) + {(#3)+0.5}*(0,1)$) -- cycle;
\draw[pth]
(${(#1)+0}*(1,0) + {(#2)+0}*({sqrt(2)*cos(deg(pi/4))},{sqrt(2)*sin(deg(pi/4))}) + {(#3)}*(0,1)$) --
(${(#1)+0}*(1,0) + {(#2)+1}*({sqrt(2)*cos(deg(pi/4))},{sqrt(2)*sin(deg(pi/4))}) + {(#3)}*(0,1)$);
}
\def\brxz(#1:#2:#3){
\filldraw[brxz]
(${(#1)+0}*(1,0) + {(#2)+0}*({sqrt(2)*cos(deg(pi/4))},{sqrt(2)*sin(deg(pi/4))}) + {(#3)-0.5}*(0,1)$) --
(${(#1)+1}*(1,0) + {(#2)+0}*({sqrt(2)*cos(deg(pi/4))},{sqrt(2)*sin(deg(pi/4))}) + {(#3)-0.5}*(0,1)$) --
(${(#1)+1}*(1,0) + {(#2)+0}*({sqrt(2)*cos(deg(pi/4))},{sqrt(2)*sin(deg(pi/4))}) + {(#3)+0.5}*(0,1)$) --
(${(#1)+0}*(1,0) + {(#2)+0}*({sqrt(2)*cos(deg(pi/4))},{sqrt(2)*sin(deg(pi/4))}) + {(#3)+0.5}*(0,1)$) -- cycle;
\draw[pth]
(${(#1)+0}*(1,0) + {(#2)+0}*({sqrt(2)*cos(deg(pi/4))},{sqrt(2)*sin(deg(pi/4))}) + {(#3)}*(0,1)$) --
(${(#1)+1}*(1,0) + {(#2)+0}*({sqrt(2)*cos(deg(pi/4))},{sqrt(2)*sin(deg(pi/4))}) + {(#3)}*(0,1)$);
}
\brxy(0:1:3);\brxy(0:2:3);\brxy(1:2:3);\brxz(2:3:2);\brxz(3:3:2);
\bryz(0:0:2);\brxz(0:1:2);\bryz(1:1:2);\brxy(1:1:2);\brxz(1:2:2);\brxy(2:2:2);\bryz(2:2:2);
\bryz(0:0:1);\brxz(0:1:1);\brxz(1:1:1);\bryz(2:1:1);\brxz(2:2:1);\bryz(3:2:1);\brxz(3:3:1);
\brxy(0:0:1);\brxy(1:0:1);\brxy(2:0:1);\brxy(2:1:1);\brxy(3:0:1);\brxy(3:1:1);\brxy(3:2:1);
\brxz(0:0:0);\brxz(1:0:0);\brxz(2:0:0);\brxz(3:0:0);\bryz(4:0:0);\bryz(4:1:0);\bryz(4:2:0);
% Create the axes
\draw[->] (-0.5,0) -- (8,0) node[right] {$t$};
\draw[->] (0,-0.5) -- (0,6.5) node[above] {$x$};
% Create the time slice
\draw[-,line width=0pt] (3,6) -- (3,-1) node[below] {$t=3$};
\node[intpt] at (3,0) {}; \node[intpt] at (3,2) {}; \node[intpt] at (3,4) {};
\end{tikzpicture}
\caption{Lozenge tiling of the hexagon and corresponding up-right path configuration. The dots represent the location of the random walks at time $t = 3$.} \label{Fig1}
\end{figure}
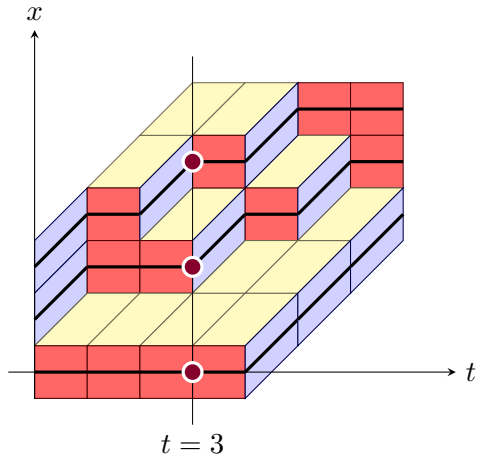 

Let us number the random paths from top to bottom by $L_1, L_2, \dots, L_a$, and denote the position of the $k$-th random walk at time $t$ by $L_k(t)$. Then law of $\{L_m\}_{m = 1}^a$ enjoys the following Gibbs property. Suppose that we sample $\{L_m\}_{m = 1}^a$ and fix two times $0 \leq s < t \leq b+c$ and an index $k \in \{1, \dots, a\}$. We can erase the part of the path $L_k$ between the points $(s, L_k(s))$ and $(t, L_k(t))$ and sample independently a new up right path between these two points uniformly from the set of all such paths that do not intersect the lines $L_{k-1}$ and $L_{k + 1}$ with the convention that $L_{0} = \infty$ and $L_{a+1} = -\infty$. In this way we obtain a new random collection of paths $\{L'_m\}_{m = 1}^a$ whose law is readily seen to be the same as that of $\{L_m\}_{m = 1}^a$. 

The above is a simple example of a discrete Gibbsian line ensemble. A (notably more complex) continuous Gibbsian line ensemble is given by the {\em Airy line ensemble}, introduced in \cite{CH}. The Airy line ensemble is a certain collection of countably many random continuous curves $\{\mathcal{L}_m \}_{m = 1}^\infty$, such that each $\mathcal{L}_i$ is a random continuous function on $\mathbb{R}$ and for each $i \geq 1$ and $x \in \mathbb{R}$ one has $\mathcal{L}_i(x) \geq \mathcal{L}_{i+1}(x)$. The top curve $\mathcal{L}_1$ is the Airy$_2$ process and the ensemble satisfies the following {\em Brownian Gibbs property}. Suppose we sample $\{\mathcal{L}_m \}_{m = 1}^\infty$ and fix two times $s, t \in \mathbb{R}$ with $ s < t$ and an index $k \in \mathbb{N}$. We can erase the part of the path $\mathcal{L}_k$ between the points $(s, \mathcal{L}_k(s))$ and $(t, \mathcal{L}_k(t))$ and sample independently a Brownian bridge between these two points, which is conditioned on not crossing $\mathcal{L}_{k-1}$ and $\mathcal{L}_{k + 1}$ with the convention that $\mathcal{L}_{0} = \infty$. In this way we obtain a new random collection of paths $\{\mathcal{L}'_m\}_{m = 1}^\infty$ and the essense of the Brownian Gibbs property is that this new random line ensemble has the same law as $\{\mathcal{L}_m \}_{m = 1}^\infty$. 

In \cite{CH} the authors heavily rely on the Brownian Gibbs property to construct and establish various properties of the Airy line ensemble. In a remarkable series of recent papers \cite{Ham1, Ham2, Ham3, Ham4} one of the authors of \cite{CH} significantly strengthened the arguments from that paper to obtain a multitude of results about the Airy line ensemble and {\em Brownian last passage percolation} (this is a different random line ensemble that enjoys the same Brownian Gibbs property we described above). These results are more qualitative in nature, e.g. estimating the modulus of continuity of the models, establishing refined regularity properties for them and finding critical exponents; however, a marked advantage of the arguments in \cite{Ham1, Ham2, Ham3, Ham4} is that they depend mostly on tools from analysis and geometry. The latter is important, as it makes the arguments (for the most part) free of exact computations and hence more easily extendable to other settings.

One of the directions we are interested in exploring is bringing some of the ideas from the continuous Gibbsian line ensemble setting to the discrete one. A particularly successful instance of the latter is \cite{CD}, where the authors investigated a discrete Gibbsian line ensemble related to the {\em ascending Hall-Littlewood process} (a special case of the {\em Macdonald processes} \cite{BorCor}). By developing discrete analogues of the arguments in \cite{CH}, \cite{CD} were successful in establishing the long-predicted $2/3$ critical exponent for the {\em asymmetric simple exclusion process (ASEP)}. A critical component of the argument in that paper is the strong coupling of Bernoulli random walk bridges to Brownian bridges, which enabled the translation of ideas from the continuous to the discrete line ensemble setting. We believe that the same could be done for other discrete models in integrable probability, whose line ensemble structure is linked to random walks with jumps that are not Bernoulli.  To give a few examples, through various versions of the {\em Robinson-Schensted-Knuth (RSK) correspondence}, one can link {\em geometric last passage percolation (LPP)} to random walk bridges with geometric jumps, {\em exponential LPP} to  random walk bridges with exponential jumps (see \cite{KJ}) and the {\em log-gamma polymer model} to  random walk bridges with log-gamma jumps  \cite{COSZ}. We remark that while the correspondence of the latter integrable models to discrete Gibbsian line ensembles is known to experts in the field, to our best knowledge the exact formulation does not appear in the literature.

We hope that many of the ideas in \cite{CH} and \cite{Ham1, Ham2, Ham3, Ham4} can be adapted to all of the examples we listed above and more. Achieving this would require strong couplings of the underlying random walk bridges in these models to Brownian bridges, and we hope that the results in the present paper will be a valuable tool for obtaining such couplings. We have attempted to make the statements in this paper as generic as possible with this goal in mind.

%-------------------------------------------------------------------------------------------------------------------------------------------------------------------------------------------------
% Section 1.3
%
%-------------------------------------------------------------------------------------------------------------------------------------------------------------------------------------------------
\subsection*{Acknowledgments} The authors are deeply grateful to Ivan Corwin for many useful suggestions and comments as well as Julien Dubedat, Alisa Knizel and Konstantin Matetski for numerous fruitful discussions. The first author is partially supported by the Minerva Foundation Fellowship. For the second author partial financial support was available through the NSF grants DMS:1811143, DMS:1664650 and the Minerva Foundation Summer Fellowship program.

\section{General setup}\label{Section2}
In this section we describe the general setting of a random walk bridge that we consider and the specific assumptions we make about it. Our discussion naturally splits into two parts, depending on whether the jump of the random walk is continuous or discrete. In each case we formulate a precise list of assumptions and present the statements we can prove for the corresponding random walk bridges that satisfy them. In the last part of this section we give a brief explanation of the significance of our assumptions.

%-------------------------------------------------------------------------------------------------------------------------------------------------------------------------------------------------
% Section 2.1
%
%-------------------------------------------------------------------------------------------------------------------------------------------------------------------------------------------------
\subsection{Continuous random walk bridges}\label{Section2.1}

We start by fixing some notation. Suppose that $X$ is a continuous random variable with density $f_X(\cdot)$ and $X_i$ are i.i.d. random variables with density $f_X$. For $n \in \mathbb{N}$ we define $S_n := X_1 + \cdots + X_n$ and also let $f_n(x)$ be the density of $S_n$.  

For any random variable $X$ and $t \in \mathbb{R}$ we define
\begin{equation}\label{S2S1E1}
M_X(t) := \mathbb{E} \left[ e^{t X} \right], \hspace{3mm}  \phi_X(t) := \mathbb{E} \left[ e^{itX} \right], \hspace{3mm} \Lambda(t) :=\log M_X(t), \hspace{3mm} \Lambda^*(t) := \sup_{x \in \mathbb{R}} \{ t x - \Lambda(x ) \}.
\end{equation}
Let $\mathcal{D}_\Lambda := \{ x: \Lambda(x) < \infty\} $ and $\mathcal{D}_{\Lambda^*} := \{x: \Lambda^*(x) < \infty\}$. 

We make the following assumptions on the function $f_X(x)$. \\

{\raggedleft \bf Assumption C1.} We assume that there are $\alpha \in [-\infty ,\infty)$ and $\beta \in (\alpha, \infty]$ and that $f_X(x)$ is positive and continuous on $(\alpha, \beta)$ and zero outside this interval. In addition, we assume that $f_X(x)$ has a continuous extension to $\alpha$ if $\alpha > -\infty$ and to $\beta$ if $\beta < \infty$.\\

{\raggedleft \bf Assumption C2.} We assume that there is a $\lambda > 0$ such that $\mathbb{E} \left[e^{\lambda |X|} \right] < \infty.$\\

For each $n \geq 1$ we set $L_n = (n \alpha, n \beta)$, where $\alpha, \beta$ are as in Assumption C1. For $z \in L_n$ we let $S^{(n,z)} = \{S_m^{(n,z)} \}_{m = 0}^n$ denote the process with the law of $\{S_m\}_{m = 0}^n$ conditioned so that $S_n = z$. We call this process a {\em continuous random walk bridge} between the points $(0,0)$ and $(n,z)$. Notice that this law is well-defined by Assumption C1.  As a natural extension of this definition we define $S_t^{(n,z)}$ for non-integer $t$ by linear interpolation.  In addition, we will denote the density of $S_m^{(n,z)}$ by $f_{m, n -m}(\cdot | z)$.

If $f_X$ satisfies Assumption C2 then $\mathcal{D}_\Lambda$  contains a neighborhood of $0$. In addition, it is easy to see that $\mathcal{D}_{\Lambda}$ is a connected set and hence an interval. 
We denote $(A_{\Lambda}, B_{\Lambda})$ the interior of $\mathcal{D}_\Lambda$ where $A_{\Lambda} \in [-\infty, - \lambda]$ and $B_{\Lambda} \in [\lambda, \infty]$. We isolate some properties for the functions in (\ref{S2S1E1}) under the above assumptions in the following lemma.

\begin{lemma}\label{S2S1L1}
Suppose that $X$ is a random variable with density $f_X$, which satisfies Assumptions C1 and C2. Then $M_X(u)$ has an analytic continuation to the vertical strip $D:= \{z:  A_{\Lambda} < Re(z) < B_{\Lambda} \}$. Moreover, $\Lambda(\cdot)$ is a smooth function on $(A_{\Lambda}, B_{\Lambda})$ and $\Lambda''(x) > 0$ for all $x \in (A_{\Lambda}, B_{\Lambda})$.  
\end{lemma}
\begin{proof}
Let $[a_n, b_n]$ be such that $a_n$ strictly decreases to $\alpha$ and $b_n$ strictly increases to $\beta$. For each $z \in D$ and $x \in (\alpha,\beta)$ we define $F(z,x) = e^{xz} f_X(x)$ and note that $F(z,x)$ is holomorphic in $z$ for each $x$ and continuous on $D \times [a_n, b_n]$. It follows from \cite[Theorem 2.5.4]{Stein} that the function
$$g_n(z) = \int_{a_n}^{b_n} F(z,x) dx$$
is holomorphic on $D$. If $K$ is a compact subset of $D$, and $z \in K$ we note that
$$g(z):= \int_{\alpha}^{\beta} e^{xz} f_X(x) dx$$
is well defined because 
$$\int_{\alpha}^{\beta} \left|e^{xz}\right| f_X(x) dx = \int_{\alpha}^{\beta} e^{x Re(z)} f_X(x)dx = M_X(Re(z)) < \infty.$$
 which is true as $Re(z) \in (A_{\Lambda}, B_{\Lambda})$. 

Note that there is $[c,d] \subset (A_{\Lambda}, B_{\Lambda})$ such that if $z \in K$ then $Re(z) \in [c,d]$. In particular, we see that $e^{xRe(z) } \leq e^{cx} + e^{dx}$ and so by the dominated convergence theorem with dominating function $f_X(x) \cdot [e^{cx} + e^{dx}]$ we get that
$$\lim_{n \rightarrow \infty} g_n(z) = g(z),$$
where the convergence is uniform over $K$. It follows from \cite[Theorem 2.5.2]{Stein} that $g(z)$ is holomorphic in $D$. Clearly, $g(z) = M_X(z)$ when $z \in (A_{\Lambda}, B_{\Lambda})$, which proves the first part of the lemma.

One can use further applications of the dominated convergence theorem to show that the derivatives of $g(z)$ are given by
$$g^{(n)}(z) = \int_{\alpha}^{\beta} \left[ \frac{d^n}{dz^n} e^{xz} \right] f_X(x)dx = \int_{\alpha}^{\beta} x^n e^{xz} f_X(x) dx,$$
and the latter integral is absolutely convergent for $Re(z) \in (A_{\Lambda}, B_{\Lambda})$. For example, see \cite{Mattner}. We next observe that for $x \in (A_{\Lambda}, B_{\Lambda})$, using the continuity and positivity of $f_X$, we know that $g(x) > 0$ and so $\Lambda(x) = \log [g(x)]$ is a smooth function on $(A_{\Lambda}, B_{\Lambda})$. From the Chain rule, we see that
$$\Lambda''(y) = \frac{g''(y) g(y) - [g'(y)]^2}{g^2(y)} = \frac{1}{2g^2(y)} \int_{\alpha}^{\beta} \int_{\alpha}^{\beta}  e^{(x_1 + x_2)y}\left[ x_1^2 + x_2^2 - 2x_1x_2 \right]f_X(x_1) f_X(x_2)      dx_1 dx_2,$$
which is clearly positive. This suffices for the proof.
\end{proof}

If $f_X$ satisfies Assumptions C1 and C2 then in view of Lemma \ref{S2S1L1} we know that $\Lambda'(x)$ is a strictly increasing function on $(A_\Lambda, B_{\Lambda})$. We let $(A^*,B^*)$ denote the image of $(A_\Lambda, B_\Lambda)$ under the map $\Lambda'(\cdot)$. In addition, we write $M_X(u)$ for all $u \in D = \{z \in \mathbb{C}: A_\Lambda< Re(z) < B_\Lambda \}$ to mean the (unique) analytic extension of $M_X(x)$ to $D$ afforded by Lemma \ref{S2S1L1}.\\

{\raggedleft \bf Assumption C3.} We assume that the function $\Lambda(\cdot)$ is lower semi-continuous on $\mathbb{R}$.\\

\begin{lemma}\label{S2S1L2} Suppose that $X$ is a random variable with density $f_X$, which satisfies Assumptions C1-C3. Then $(\alpha, \beta) \subset (A^*, B^*) \subset \mathcal{D}_{\Lambda^*}$ and for all $y \in (A^*,B^*)$ we have $\Lambda^*(y) = \eta y - \Lambda(\eta)$, where $\eta = (\Lambda')^{-1}(y)$.
\end{lemma}
\begin{proof}
By Lemma \ref{S2S1L1} we know that $\Lambda'(\cdot)$ is a strictly increasing smooth function from $(A_\Lambda, B_\Lambda)$ to $(A^*, B^*)$, which implies that $(\Lambda')^{-1}(\cdot)$ is also a smooth increasing function from $(A^*, B^*) $ to $(A_\Lambda, B_\Lambda)$. The statements $(A^*, B^*) \subset \mathcal{D}_{\Lambda^*}$ and $\Lambda^*(y) = \eta y - \Lambda(\eta)$ for all $y \in (A^*,B^*)$ follow from \cite[Lemma 2.2.5]{DZ}. In the remainder we show that $(\alpha, \beta) \subset (A^*, B^*)$.

Let $z \in (\alpha, \beta)$ and suppose that $\epsilon > 0$ is such that $(z- \epsilon, z+ \epsilon) \subset (\alpha, \beta)$. Suppose first that $A_\Lambda > - \infty$. Then by Assumption C3, we know that $\liminf_{x_n \rightarrow A_\Lambda} \Lambda(x_n) = \infty$. This implies that 
$$\lim_{x_n \rightarrow A_{\Lambda}} z x_n - \Lambda(x_n) = -\infty.$$
Conversely, if $A_{\Lambda} = -\infty$ and $x_n \rightarrow  A_\Lambda$ then 
$$\limsup_{x_n \rightarrow A_{\Lambda}} z x_n - \Lambda(x_n) = \limsup_{x_n \rightarrow A_{\Lambda}} z x_n - \log \left[ \mathbb{E} \left[ e^{x_n X} \right] \right]\leq $$
$$  \limsup_{x_n \rightarrow A_{\Lambda}} z x_n - \log \left[ e^{x_n (z - \epsilon/2)} \cdot \mathbb{P}(X \in \left[z- \epsilon, z- \epsilon/2 \right] ) \right]\leq \frac{x_n \epsilon}{2} - \log( \mathbb{P}(X \in \left[z- \epsilon, z- \epsilon/2 \right] ) ) = -\infty.$$
Similar considerations show that $\lim_{x_n \rightarrow B_{\Lambda}} zx_n - \Lambda(x_n) = -\infty$. 

By Lemma \ref{S2S1L1} $zx - \Lambda(x)$ is smooth in $(A_\Lambda, B_\Lambda)$ and from the above we conclude that its maximum is achieved at a point $x_z \in (A_\Lambda, B_\Lambda)$ with $0 = \frac{d}{dx} [ zx - \Lambda(x)] = z - \Lambda'(x_z)$. This shows that $z \in (A^*, B^*)$.

\end{proof}

{\raggedleft \bf Assumption C4.} We assume that for every $B_\Lambda > t> s > A_\Lambda$ there exist positive constants $K_1(s,t)$ and $p(s,t) > 0$ such that $\left|M_X(z)\right| \leq \frac{K_1(s,t)}{(1 + |Im(z)|)^{p(s,t)}}$, provided $s\leq Re(z) \leq t$.\\

{\bf \raggedleft Assumption C5.} We suppose that there are constants $L, D, d > 0$ such that $f_X(x) \leq L$ for all $x \in \mathbb{R}$ and at least one of the following statements holds
\begin{equation}\label{S2S1E2}
\mbox{1. }f_X(x) \leq De^{-dx^2} \mbox{ for all $x \geq 0$  or 2. } f_X(x) \leq De^{-dx^2} \mbox{ for all $x \leq 0$}.
\end{equation}

{\raggedleft \bf Assumption C6.} We assume that there are functions $\hat{C}: \mathbb{R}_{> 0} \rightarrow \mathbb{R}_{> 0} $ and $\hat{a}: \mathbb{R}_{> 0} \rightarrow \mathbb{R}_{> 0} $ such that the following holds. For all $n \geq 1$, $z \in L_n$ and $\hat{b} > 0$ we have
\begin{equation}\label{S2S1E3}
\mathbb{E} \left[ \exp \left(\hat{a}(\hat{b}) \max_{1 \leq k \leq n} |S_k| \right) \Big{|} S_n = z\right]\leq \hat{C}(\hat{b}) \exp \left( \hat{b} (n + z^2/n)\right).
\end{equation}

In the sequel we denote $u_z = (\Lambda')^{-1}(z)$, $\sigma_z^2 = \Lambda''(u_z)$ -- these are well defined for densities $f_X$ that satisfy Assumptions C1-C3 as follows from Lemmas \ref{S2S1L1}  and \ref{S2S1L2}. Using this notation we can formulate the main theorem we prove for continuous random walk bridges.
\begin{theorem}\label{S2ContKMT}
Suppose that $X$ is a random variable whose density function $f_X$ satisfies Assumptions C1-C6 and fix $p \in (\alpha ,\beta)$. For every $b > 0$, there exist constants $0 < C, a, \alpha' < \infty$ (depending on $b$, $p$ and the function $f_X(\cdot)$) such that for every positive integer $n$, there is a probability space on which are defined a Brownian bridge $B^\sigma$ with variance $\sigma^2 = \sigma^2_p$ and the family of processes $S^{(n,z)} $ for $z \in L_n$ such that
\begin{equation}\label{S2ContKMTE}
\mathbb{E}\left[ e^{a \Delta(n,z)} \right] \leq C e^{\alpha' (\log n)} e^{b|z- p n|^2/n},
\end{equation}
where $\Delta(n,z) = \Delta(n,z,B^{\sigma}, S^{(n,z)})=  \sup_{0 \leq t \leq n} \left| \sqrt{n} B^\sigma_{t/n} + \frac{t}{n}z - S^{(n,z)}_t \right|.$
\end{theorem}
In Section \ref{Section2.3} we provide some explanation of the significance of Assumptions C1-C6.

%-------------------------------------------------------------------------------------------------------------------------------------------------------------------------------------------------
% Section 2.2
%
%-------------------------------------------------------------------------------------------------------------------------------------------------------------------------------------------------
\subsection{Discrete random walk bridges}\label{Section2.2}

We start by fixing some notation. Suppose that $X$ is a random variable such that $\mathbb{P}(X \in \mathbb{Z}) = 1$ and let $p_X(n) = \mathbb{P}(X = n)$ for $n \in \mathbb{Z}$ denote its probability mass function. We let $X_i$ be an i.i.d. sequence of random variables with distribution function $p_X$. For $n \in \mathbb{N}$ we define $S_n = X_1 + \cdots + X_n$ and also let $p_n(\cdot)$ be the probability mass function of $S_n$.

Similarly to Section \ref{Section2.1} we define
\begin{equation}\label{S2S2E1}
M_X(t) := \mathbb{E} \left[ e^{t X} \right], \hspace{3mm}  \phi_X(t) := \mathbb{E} \left[ e^{itX} \right], \hspace{3mm} \Lambda(t) :=\log M_X(t) \hspace{3mm} \Lambda^*(t) := \sup_{x \in \mathbb{R}} \{ t x - \Lambda(x) \}.
\end{equation}
Let $\mathcal{D}_\Lambda := \{ x: \Lambda(x) < \infty\} $ and $\mathcal{D}_{\Lambda^*} := \{x: \Lambda^*(x) < \infty\}$. 

We make the following assumptions on the function $p_X(x)$. \\

{\raggedleft \bf Assumption D1.} We assume that $p_X(x)$ has a single interval of support, i.e. $I = \{ x \in \mathbb{Z} : p_X(x) > 0 \} = (\alpha- 1,\beta + 1) \cap \mathbb{Z}$ for some $\alpha \in \mathbb{Z} \cup \{-\infty\}$ and $\beta \in \left((\alpha, \infty] \cap \mathbb{Z} \right)\cup \{ \infty\}$.\\

{\raggedleft \bf Assumption D2.} We assume that there is a $\lambda > 0$ such that $\mathbb{E} \left[e^{\lambda |X|} \right] < \infty.$\\

For each $n \geq 1$ we set $L_n =  (n \alpha - 1, n \beta + 1) \cap \mathbb{Z}$, where $\alpha, \beta$ are as in Assumption D1. For $z \in L_n$ we let $S^{(n,z)} = \{S_m^{(n,z)} \}_{m = 0}^n$ denote the process with the law of $\{S_m\}_{m = 0}^n$ conditioned so that $S_n = z$. We call this process a {\em discrete random walk bridge} between the points $(0,0)$ and $(n,z)$. Notice that this law is well-defined by Assumption D1.  As a natural extension of this definition we define $S_t^{(n,z)}$ for non-integer $t$ by linear interpolation.  In addition, we will denote the distribution function of $S_m^{(n,z)}$ by $p_{m, n -m}(\cdot | z)$.

If $p_X$ satisfies Assumption D2 then $\mathcal{D}_\Lambda$  contains a neighborhood of $0$. In addition, it is easy to see that $\mathcal{D}_{\Lambda}$ is a connected set and hence an interval. We denote $(A_{\Lambda}, B_{\Lambda})$ the interior of $\mathcal{D}_\Lambda$ where $A_{\Lambda} \in [-\infty, - \lambda]$ and $B_{\Lambda} \in [\lambda, \infty]$. We isolate some properties for the functions in (\ref{S2S2E1}) under the above assumptions in the following lemma.

\begin{lemma}\label{S2S2L1}
Suppose that $X$ is a random variable whose distribution function $p_X$ satisfies Assumptions D1 and D2. Then $M_X(u)$ has an analytic continuation to the vertical strip $D:= \{z:  A_{\Lambda} < Re(z) < B_{\Lambda} \}$. Moreover, $\Lambda(\cdot)$ is a smooth function on $(A_{\Lambda}, B_{\Lambda})$ and $\Lambda''(x) > 0$ for all $x \in (A_{\Lambda}, B_{\Lambda})$.  
\end{lemma}
\begin{proof}
The proof is analogous to that of Lemma \ref{S2S1L1}.
\end{proof}

If $p_X$ satisfies Assumptions D1 and D2 then in view of Lemma \ref{S2S2L1} we know that $\Lambda'(x)$ is a strictly increasing function on $(A_\Lambda, B_{\Lambda})$. We let $(A^*,B^*)$ denote the image of $(A_\Lambda, B_\Lambda)$ under the map $\Lambda'(\cdot)$. In addition, we write $M_X(u)$ for all $u \in D = \{z \in \mathbb{C}: A_\Lambda< Re(z) < B_\Lambda \}$ to mean the (unique) analytic extension of $M_X(x)$ to $D$ afforded by Lemma \ref{S2S2L1}.\\

{\raggedleft \bf Assumption D3.} We assume that the function $\Lambda(\cdot)$ is lower semi-continuous on $\mathbb{R}$.\\

\begin{lemma}\label{S2S2L2} Suppose that $X$ is a random variable whose distribution function $p_X$ satisfies satisfies Assumptions D1-D3. Then $(\alpha,\beta) \subset (A^*, B^*) \subset \mathcal{D}_{\Lambda^*}$ and for all $y \in (A^*,B^*)$ we have $\Lambda^*(y) = \eta y - \Lambda(\eta)$, where $\eta = (\Lambda')^{-1}(y)$. Furthermore, $\Lambda^*(x)$ is lower semi-continuous. If $\alpha > -\infty$ then $\alpha \in \mathcal{D}_{\Lambda^*}$ and $\Lambda^*(\alpha) = - \log p_X( \alpha)$. Similarly, if $\beta < \infty$ then $\beta \in \mathcal{D}_{\Lambda^*}$ and $\Lambda^*(\beta) = - \log p_X( \beta)$. 
\end{lemma}
\begin{proof}
By Lemma \ref{S2S2L1} we know that $\Lambda'(\cdot)$ is a strictly increasing smooth function from $(A_\Lambda, B_\Lambda)$ to $(A^*, B^*)$, which implies that $(\Lambda')^{-1}(\cdot)$ is also a smooth increasing function from $(A^*, B^*) $ to $(A_\Lambda, B_\Lambda)$. The statements $(A^*, B^*) \subset \mathcal{D}_{\Lambda^*}$, $\Lambda^*(y) = \eta y - \Lambda(\eta)$ for all $y \in (A^*,B^*)$ and the lower semi-continuity of $\Lambda^*$ follow from \cite[Lemma 2.2.5]{DZ}. We next show that $(\alpha, \beta) \subset (A^*, B^*)$.

Let $z \in (\alpha, \beta)$ and fix $k,m \in \mathbb{Z}$ such that $ \alpha \leq k < z$ and $z > m \geq \beta$. Suppose first that $A_\Lambda > - \infty$. Then by Assumption D3, we know that $\liminf_{x_n \rightarrow A_\Lambda} \Lambda(x_n) = \infty$. This implies that 
$$\lim_{x_n \rightarrow A_{\Lambda}} z x_n - \Lambda(x_n) = -\infty.$$
Conversely, if $A_{\Lambda} = -\infty$ and $x_n \rightarrow  A_\Lambda$ then 
$$\limsup_{x_n \rightarrow A_{\Lambda}} z x_n - \Lambda(x_n) = \limsup_{x_n \rightarrow A_{\Lambda}} z x_n - \log \left[ \mathbb{E} \left[ e^{x_n X} \right] \right]\leq $$
$$  \limsup_{x_n \rightarrow A_{\Lambda}} z x_n - \log \left[ e^{x_n k} \cdot \mathbb{P}(X = k ) \right]\leq x_n(z - k) - \log( p_X(k)) = -\infty.$$
Similar considerations show that $\lim_{x_n \rightarrow B_{\Lambda}} zx_n - \Lambda(x_n) = -\infty$. 

By Lemma \ref{S2S2L1} $zx - \Lambda(x)$ is smooth in $(A_\Lambda, B_\Lambda)$ and from the above we conclude that its maximum is achieved at a point $x_z \in (A_\Lambda, B_\Lambda)$ with $0 = \frac{d}{dx} [ zx - \Lambda(x)] = z - \Lambda'(x_z)$. This shows that $z \in (A^*, B^*)$.\\

Next suppose that $\alpha > -\infty$. Then we have $A_\Lambda = -\infty$. We have for any $x \in \mathbb{R}$ that
$$\alpha x - \Lambda(x) \leq \alpha x - \log \left[ \mathbb{E} \left[ e^{x X} \right]\right] \leq \alpha x - \log \left[e^{\alpha x} p_X( \alpha)  \right] \leq - \log p_X(\alpha).$$
Furthermore, we have
$$\liminf_{x_n \rightarrow -\infty} \alpha x_n - \Lambda(x_n)  \geq  \liminf_{x_n \rightarrow -\infty}  \alpha x_n - \log \left[ e^{\alpha x_n} p_X(\alpha) + e^{(\alpha +1)x_n }\cdot (1 - \mathbb{P}(X = \alpha)) \right] = $$
$$ \liminf_{x_n \rightarrow -\infty}  - \log \left[ p_X(\alpha) + e^{x_n }\cdot (1 - p_X(\alpha)) \right]  = - \log p_X(\alpha).$$
Thus $\alpha \in \mathcal{D}_{\Lambda^*}$ and $\Lambda^*(\alpha) = - \log p_X(\alpha)$. Analogous arguments prove the statement for $\beta < \infty$. 
\end{proof}

{\bf \raggedleft Assumption D4.} We suppose that there are constants $D, d > 0$ such that at least one of the following statements holds
\begin{equation}\label{S2S2E2}
\mbox{1. } p_X(x) \leq De^{-dx^2} \mbox{ for all $x \geq 0$  or 2. } p_X(x) \leq De^{-dx^2} \mbox{ for all $x \leq 0$}.
\end{equation}

{\raggedleft \bf Assumption D5.} We assume that there are functions $\hat{C}: \mathbb{R}_{> 0} \rightarrow \mathbb{R}_{> 0} $ and $\hat{a}: \mathbb{R}_{> 0} \rightarrow \mathbb{R}_{> 0} $ such that the following holds. For all $n \geq 1$, $z \in L_n$ and $\hat{b} > 0$ we have
\begin{equation}\label{S2S2E3}
\mathbb{E} \left[ \exp \left(\hat{a}(\hat{b}) \max_{1 \leq k \leq n} |S_k| \right) \Big{|} S_n = z\right]\leq \hat{C}(\hat{b}) \exp \left( \hat{b} (n + z^2/n)\right).
\end{equation}

In the sequel we denote $u_z = (\Lambda')^{-1}(z)$, $\sigma_z^2 = \Lambda''(u_z)$ -- these are well defined for distribution functions $p_X$ that satisfy Assumptions D1-D3 as follows from Lemmas \ref{S2S2L1}  and \ref{S2S2L2}. Using this notation we can formulate the main theorem we prove for discrete random walk bridges.
\begin{theorem}\label{S2DiscKMT}
Suppose that $X$ is a random variable whose probability distribution function $p_X$ satisfies Assumptions D1-D5 and fix $p \in (\alpha ,\beta)$. For every $b > 0$, there exist constants $0 < C, a, \alpha' < \infty$ (depending on $b$, $p$ and the function $p_X(\cdot)$) such that for every positive integer $n$, there is a probability space on which are defined a Brownian bridge $B^\sigma$ with variance $\sigma^2 = \sigma^2_p$ and the family of processes $S^{(n,z)} $ for $z \in L_n$ such that
\begin{equation}\label{S2DiscKMTE}
\mathbb{E}\left[ e^{a \Delta(n,z)} \right] \leq C e^{\alpha' (\log n)} e^{b|z- p n|^2/n},
\end{equation}
where $\Delta(n,z) = \Delta(n,z,B^{\sigma}, S^{(n,z)})=  \sup_{0 \leq t \leq n} \left| \sqrt{n} B^\sigma_{t/n} + \frac{t}{n}z - S^{(n,z)}_t \right|.$
\end{theorem}
In Section \ref{Section2.3} we provide some explanation of the significance of Assumptions D1-D5.

%-------------------------------------------------------------------------------------------------------------------------------------------------------------------------------------------------
% Section 2.3
%
%-------------------------------------------------------------------------------------------------------------------------------------------------------------------------------------------------
\subsection{Significance of assumptions}\label{Section2.3}
Let us explain the role of the different Assumptions C1-C6 and D1-D5 that we made in the previous sections.  Assumption C1 (resp. D1) ensures that the law of the random walk bridge $S^{(n,z)}$ is well defined. Without the assumption that the support of $f_X(\cdot)$ (resp. $p_X(\cdot)$) is a single interval one runs into the possibility of conditioning on events of zero probability (in the density sense for the continuous bridges). It is possible to relax this condition, by requiring that sufficiently many convolutions of $f_X(\cdot)$ (resp. $p_X$) with itself satisfy this assumption, but we will assume that $f_X(\cdot)$ (resp. $p_X$) satisfies it instead, as this somewhat simplifies our discussion.

Assumptions C2 and C4 (resp. D2) are essentially the same as those used in KMT's original work \cite{KMT1, KMT2}. Since our results are analogues of \cite[Theorem 1]{KMT1} it is natural to have these assumptions.

In the process of proving Theorem \ref{S2ContKMT} (resp. Theorem \ref{S2DiscKMT}) we will require detailed estimates on the conditional distributions $f_{m,n}(\cdot| z)$ (resp. $p_{m,n}(\cdot |z)$) for $m,n \geq 1$, which in turn would require estimates on $f_{n + m}(z)$ (resp. $p_{n+m}(z)$). Consequently, we will require large deviation estimates for the latter densities, which involve the rate function $\Lambda$. For this reason, it will be convenient for us to assume that $\Lambda$ is lower semi-continuous, which is Assumption C3 (resp. D3). 

Assumptions C5 and C6 (resp. D4 and D5) are more technical and more directly tied to the particular approach we take to proving Theorem \ref{S2ContKMT} (resp. Theorem \ref{S2DiscKMT}). It is possible that one can relax (or entirely remove) some of these assumptions, but one would need to implement different ideas than the ones we use. Our argument goes through a comparison of the distribution $f_{n,n}(\cdot|z)$ (resp. $p_{n,n}(\cdot |z)$) with a suitable Gaussian density, for which it is useful to know that $f_{n,n}(\cdot| z)$ (resp. $p_{n,n}(\cdot |z)$) has Gaussian tails -- this is the essence of Assumption C5 (resp. D5). Our proof of Theorems \ref{S2ContKMT} and \ref{S2DiscKMT} relies on an inductive argument on $n$. When we go from $n/2$ to $n$, Assumptions C1-C5 (resp. D1-D4) are enough to complete the induction step, provided $z$ is close to the reference slope $pn$, but for points that are macroscopically away from this point, we require the estimates in Assumption C6 (resp. D5). Later in Section \ref{Section7} we provide several easy to check conditions that imply Assumption C6 (resp. D5).

We want to emphasize that it is not enough to assume Assumptions C1-C5 (resp D1-D4), and obtain Theorem \ref{S2ContKMT} (resp. Theorem \ref{S2DiscKMT}) as we demonstrate in Section \ref{Section7.2}, by providing a counterexample. The counterexample is for the discrete setting of our problem but can be naturally adapted to the continuous one. This indicates that one should make additional assumptions on $f_X(\cdot)$ (resp. $p_X(\cdot)$) and our choice of Assumption C6 (resp. D5) is made because it is somewhat natural and satisfied by the distributions in the particular applications that we have in mind.

We end this section with the following remark.
\begin{remark}\label{S2Remark}
In the process of establishing the results necessary for the proofs of Theorems \ref{S2ContKMT} and \ref{S2DiscKMT} we will obtain numerous constants that depend on the jump distribution $f_X$ in the continuous and $p_X$ in the discrete case. Some of the applications we have in mind are to situations when the jump distribution depends on a parameter that is allowed to vary in some (possibly infinite) interval. Consequently, we are interested in quantifying the dependence of our coupling constants on the functions $f_X$ and $p_X$, through various observables of these distributions. In words, we are interested in showing that the coupling constants $a, C$ and $\alpha'$ in Theorems \ref{S2ContKMT} and \ref{S2DiscKMT} can be taken uniformly even if $f_X$ or $p_X$ depend on some parameter so long as one has uniform control of several observables for $f_X$ or $p_X$ that will be made explicit in later sections. These more quantified versions of  Theorems \ref{S2ContKMT} and \ref{S2DiscKMT} can be found in Section \ref{Section6} as  Theorems \ref{ContKMTA} and \ref{KMTA} respectively. We provide an example of the situation described in this remark in Section \ref{Section8.3}.
\end{remark}

\section{Midpoint distribution: Continuous case}\label{Section3}

We continue with the same notation as in Section \ref{Section2.1}. To ease the notation a bit we will write $M, \phi$ and $\Lambda$ instead of $M_X, \phi_X$ and $\Lambda_X$. Let $f_{n,m}(x|y)$ be the density of $S_m$ conditioned on $S_{n+m} = y$. Our goal in this section is to obtain several asymptotic statements about the distribution $f_{m,n}(\cdot | (m+n) z)$ and we start by analyzing $f_N(Nz)$. 

%-------------------------------------------------------------------------------------------------------------------------------------------------------------------------------------------------
% Section 3.1
%
%-------------------------------------------------------------------------------------------------------------------------------------------------------------------------------------------------

\subsection{Asymptotics of $f_N(Nz)$}\label{Section3.1}
In this section we assume that $f_X(\cdot)$ satisfies Assumptions C1-C4. For a fixed $z \in (A^*, B^*)$ we define 
\begin{equation}\label{S3S1E1}
G_z(u) = \Lambda(u) - z \cdot u, \mbox{ for $u \in (A_{\Lambda}, B_{\Lambda})$.}
\end{equation}

\begin{definition}\label{DefDelta} Suppose that we are given $s,t \in \mathbb{R}$ such that $\alpha < s < t < \beta$, where $\alpha, \beta$ are as in Assumption C1. In addition, we denote $S = (\Lambda')^{-1}(s)$ and $T = (\Lambda')^{-1}(t)$ -- these quantities are well-defined in view of Lemma \ref{S2S1L2}. By Lemma \ref{S2S1L1} there exist $\infty > M_{s,t} \geq m_{s,t} > 0$ such that $M_{s,t} \geq \Lambda''(y) \geq m_{s,t}$ for all $y \in [S,T]$. We can pick $\delta_{s,t} > 0 $ sufficiently small (depending on $s, t$ and $f_X(\cdot)$) so that  
\begin{enumerate}
\item If $D_{\delta_{s,t}}(S,T) : = \{ z \in \mathbb{C}: d(z, [S,T]) < \delta_{s,t}\}$ then $\overline{D}_{\delta_{s,t}}(S,T) \subset \{z \in \mathbb{C}: A_{\Lambda} < Re(z) < B_{\Lambda} \}$;
\item $Re [M_X(u)] > 0$ for all $u \in  \overline{D}_{\delta_{s,t}}(S,T) $;
\item $\delta_{s,t} < 1/2$;
\item $8\delta_{s,t} \cdot |\log (M_X(u))| < m_{s,t}$ for all  $u \in \overline{D}_{\delta_{s,t}}(S,T)$.
\end{enumerate}
\end{definition}

\begin{definition}\label{DefK} Suppose that we are given $s,t \in \mathbb{R}$ such that $\alpha < s < t < \beta$, where $\alpha, \beta$ are as in Assumption C1. In view of Assumption C4 there exists a constant $K_{s,t} \geq 1$ sufficiently large (depending on $s,t$ and $f_X(\cdot)$) and $p_{s,t} > 0$ so that for every $u_z \in [\min(u_s, 0),\max(u_t,0)]$ we have
$$\left| M(u_z + iy) \cdot e^{-z(u_z+iy)} e^{-G_z(u_z)} \right| \leq \frac{K_{s,t}}{(1 + |y|)^{p_{s,t}}}.$$
\end{definition}

\begin{definition}\label{DefQ} Suppose that we are given $s,t \in \mathbb{R}$ such that $\alpha < s < t < \beta$, where $\alpha, \beta$ are as in Assumption C1. Suppose that $\delta_{s,t}$ and $K_{s,t}, p_{s,t}$ satisfy the conditions in Definitions \ref{DefDelta} and \ref{DefK}. Denote $\epsilon_{s,t} = \delta_{s,t}^4$ and $R_{s,t} =  [4 K_{s,t}]^{2/ p_{s,t}}$. Then we can find $q_{s,t} \in (0,1)$ (depending on $s,t, \delta_{s,t}, K_{s,t}, p_{s,t}$ and $f_X(\cdot)$)such that for every $z \in [s,t]$ and $y \in [\epsilon_{s,t}, R_{s,t}]$ we have
$$\left| \mathbb{E} \left[e^{(u_z + iy)X}\right] \right|  e^{-zu_z} e^{-G_z(u_z)} \leq q_{s,t}. $$
To see why the above is true, notice that
$$\left| \mathbb{E} \left[e^{(u_z + iy)X}\right] \right|  e^{-zu_z} e^{-G_z(u_z)} < \mathbb{E} \left[ \left|e^{(u_z + iy)X} \right|\right] e^{-zu_z} e^{-G_z(u_z)} = 1,$$
where the above inequality is strict for any $y \neq 0$ as the contrary would imply $X \in 2\pi y^{-1} \cdot \mathbb{Z}$ almost surely, which is not true. This combined with the continuity of $ \mathbb{E} \left[e^{(u_z + iy)X}\right]$ in $y$ and $z$ ensures the existence of $q_{s,t}$ with the desired properties.
\end{definition}

We are interested in proving the following statement.
\begin{proposition}\label{S3S1P1} Suppose that $f_X$ satisfies Assumptions C1-C4. Fix $\beta > t > s > \alpha$ and $z \in [s,t]$. Then there exists $N_0\in \mathbb{N}$ such that if $N \geq N_0$ one has
\begin{equation}\label{S3S1E3}
f_N(Nz) = \frac{1}{\sqrt{2\pi N }\sigma_z} \cdot \exp \left( N G_z(u_z) + \delta_1(z,N) \right), \mbox{ where $\delta_1(z,N) = O(N^{-1/2})$}.
\end{equation}
The number $N_0$ and the constant in the big $O$ notation depend on $f_X, s$ and $t$ only through the constants in Definitions \ref{DefDelta}, \ref{DefK} and \ref{DefQ}. 
\end{proposition}
\begin{proof}
From Definition \ref{DefK} and \cite[Theorem 3.3.5]{Durrett} we know that for $N$ sufficiently large (specifically it suffices to take $N > p^{-1}_{s , t}$) then
$$f_N(Nz) = \frac{1}{2\pi} \int_{\mathbb{R}} e^{-iy Nz} \left( \phi(y) \right)^N dy.$$
Performing the change of variables $u = iy$ we see that 
\begin{equation}\label{S3S1P1E1}
f_N(Nz) = \frac{1}{2\pi i} \int_{-i \infty}^{ i \infty} M^N(u) e^{ - u N z} du.
\end{equation}

Let us shift the $u$ contour in (\ref{S3S1P1E1}) to the vertical contour passing through $u_z$. In view of Lemma \ref{S2S1L1}, we do not pass any poles in the process of deformation and so by Cauchy's theorem the value of the integral remains unchanged. The decay necessary to deform the contours near $\pm i \infty$ comes from Definition \ref{DefK} and our assumption that $N$ is sufficiently large. The result is 
\begin{equation}\label{S3S1P1E3}
f_N(Nz) = \frac{e^{NG_z(u_z)}}{2\pi i} \int_{u_z -i \infty}^{u_z + i \infty} M(u)^N e^{-uNz}e^{-NG_z(u_z)} du.
\end{equation}

For the given $s,t$ as in the statement of the proposition we define $\delta_{s,t}, m_{s,t},  K_{s,t}, \epsilon_{s,t}, R_{s,t}, p_{s,t}$ and $q_{s,t}$ as in Definitions \ref{DefDelta}, \ref{DefK} and \ref{DefQ}. To ease notation we will drop $s,t$ from the notation of these quantities. We will also denote by $C_{s,t}$ the supremum of $|\log(M(u))| $ as $u$ varies over $\overline{D_\delta}$. Notice that by construction we have
$$\epsilon < \delta/2 \mbox{ and } \epsilon \cdot 8C_{s,t} \cdot \delta^{-3} < m.$$

From (\ref{S3S1P1E3}) we have $f_N(Nz) = (I) + (II), \mbox{ where }$
\begin{equation}\label{S3S1P1E4}
\begin{split}
&(I) = \frac{e^{NG_z(u_z)}}{2\pi i} \int_{u_z -i \epsilon}^{u_z + i \epsilon} e^{N[ G_z(u) - G_z(u_z)]} du, (II) = \frac{e^{NG_z(u_z)}}{2\pi i} \int_{u_z -i \infty}^{u_z - i \epsilon}\left[M(u)e^{-uz}e^{-G_z(u_z)} \right]^N \hspace{-2mm}du \\
& + \frac{e^{NG_z(u_z)}}{2\pi i} \int^{u_z  + i \infty}_{u_z + i \epsilon}\left[M(u)e^{-uz}e^{-G_z(u_z)} \right]^N  \hspace{-2mm} du.
\end{split}
\end{equation}

We will first obtain estimates on (I), which will require analyzing the power series expansion of $G_z(u_z + ir) - G_z(u_z)$ around the point $u_z$. Note that by definition 
$$G_z(u_z + ir) - G_z(u_z) = - \frac{r^2\sigma_z^2}{2} + \sum_{n = 3}^\infty\frac{\Lambda^{(n)}(u_z)}{n!} (ir)^n.$$
From the Cauchy inequalities \cite[Corollary 2.4.3]{Stein} and our choice of $\epsilon$ we conclude that for $|r| \leq \epsilon$
\begin{equation}\label{S3S1P1E5}
\left|  G_z(u_z + ir) - G_z(u_z) + \frac{r^2 \sigma_z^2}{2} \right| \leq C_{s,t}|r|^3 \sum_{n = 3}^\infty \frac{|\epsilon|^{n-3}}{\delta^n}  \leq 2\delta^{-3} C_{s,t} |r|^3 =:C(s,t,\delta)|r|^3.
\end{equation}
Changing variables in (\ref{S3S1P1E4}) and using (\ref{S3S1P1E5}) we obtain
\begin{equation*}
 \frac{e^{NG_z(u_z)}}{2\pi \sqrt{N}} \int_{-\epsilon N^{1/2}}^{\epsilon N^{1/2}} \exp \left[ - \frac{x^2\sigma_z^2}{2} - \frac{C(s,t,\delta)}{\sqrt{N}} |x|^3  \right] dx \leq (I) \leq \frac{e^{NG_z(u_z)}}{2\pi \sqrt{N}} \int_{-\epsilon N^{1/2}}^{\epsilon N^{1/2}} \exp \left[ - \frac{x^2\sigma_z^2}{2}  + \frac{C(s,t,\delta)}{\sqrt{N}}  |x|^{3} \right] dx.
\end{equation*}

Using the inequality $|e^{A} - 1| \leq |A|e^{|A|} \mbox{ for all $A \in \mathbb{R}$},$ we obtain
$$\left| (I) -  \frac{e^{NG_z(u_z)}}{2\pi \sqrt{N}} \int_{-\epsilon N^{1/2}}^{\epsilon N^{1/2}} \exp \left[ - \frac{\sigma_z^2 x^2}{2}   \right] dx\right| \leq   \frac{e^{NG_z(u_z)}}{2\pi \sqrt{N}} \int_{-\epsilon N^{1/2}}^{\epsilon N^{1/2}} \frac{C(s,t,\delta) |x|^3}{\sqrt{N}}  \exp \left[ - \frac{\sigma_z^2 x^2}{2}   + \frac{C(s,t,\delta)}{\sqrt{N}}|x|^3\right] dx.$$
Notice that by our choice of $\epsilon$ we have for $|x| \leq \epsilon N^{1/2}$ that
$$ - \frac{\sigma_z^2 x^2}{2}   + C(s,t,\delta)|x|^3N^{-1/2} \leq - \frac{\sigma_z^2 x^2}{4},$$
which implies from above that
\begin{equation}\label{S3S1P1E6}
\left| (I) -  \frac{e^{NG_z(u_z)}}{2\pi \sigma_z \sqrt{N}} \cdot \left( 1 - 2 \bar{\Phi}\left(\epsilon \sqrt{N} \right)\right)  \right| \leq \frac{e^{NG_z(u_z)}}{2\pi \sqrt{N}} \int_{\mathbb{R}} \frac{C(s,t,\delta)}{\sqrt{N}} |x|^3  \exp \left[  - \frac{\sigma_z^2x^2}{4}  \right]dx,
\end{equation}
where $\bar{\Phi}(x) = \mathbb{P}( Z > x)$ with $Z$ being a Gaussian variable with mean zero and variance $1$.

Using a simple change of variables we have
$$\int_{\mathbb{R}}  |x|^3  \exp \left[  - \frac{\sigma_z^2}{4} x^2 \right]dx = \frac{4}{\sigma_z}  \int_0^\infty y^3e^{-y^2} dy = \frac{2}{\sigma_z}.$$
Combining the latter with the inequality $\bar{\Phi}(x) \leq 2 e^{-x^2/2}$ for all $x \geq 0$ and (\ref{S3S1P1E6}) we get
\begin{equation}\label{P1E6v2}
\left|(I) - \frac{e^{NG_z(u_z)}}{2\pi \sigma_z \sqrt{N}} \right| \leq \frac{e^{NG_z(u_z)}}{2\pi \sigma_z \sqrt{N}}  \cdot \left( \frac{2C(s,t,\delta)}{\sqrt{N}} +  4 \exp \left( -\epsilon N/2 \right) \right) 
\end{equation}
We can now make $N_0$ sufficiently large so that for all $z \in[s,t]$ and $N \geq N_0$
\begin{equation}\label{S3S1P1E7}
(I) =\frac{e^{NG_z(u_z)}}{2\pi \sigma_z \sqrt{N}} \left[1 + O \left( \frac{1}{\sqrt{N}} \right) \right].
\end{equation}

We next forcus on estimating (II). We first note by construction we have
$$\left| M(u_z + iy) \cdot e^{-z(u_z+iy)} e^{-G_z(u_z)} \right| \leq \frac{K}{(1 + |y|)^{p}}.$$
 The latter implies that if $N \geq N_0 > 2/p$ we have
\begin{equation}\label{S3S1P1E8v0}
 \int_{|y| > R} \left| M(u_z + iy) \cdot e^{-z(u_z+iy)} e^{-G_z(u_z)} \right|^N dy \leq  2  K^N \frac{R^{1 - pN}}{pN - 1} \leq 2 K^N R^{ - pN/2} =  2\cdot 4^{-N}.
\end{equation}

Suppose next that $y \in [\epsilon, R]$. Then by definition we have 
\begin{equation}\label{S3S1P1E8v1}
 \int_{\epsilon \leq |y| \leq R} \left| M(u_z + iy) \cdot e^{-z(u_z+iy)} e^{-G_z(u_z)} \right|^N dy \leq  2R q^N.
\end{equation}
Combining (\ref{S3S1P1E8v0}) and (\ref{S3S1P1E8v1}) we get
\begin{equation}\label{S3S1P1E8}
|(II)| \leq \frac{e^{NG_z(u_z)}}{2\pi }\cdot [ 2R q^{N} + 2\cdot 4^{-N}] \leq \frac{e^{NG_z(u_z)}}{2\pi \sigma_z N } ,
\end{equation}
where the last inequality holds provided $N_0$ is sufficiently large and $N \geq N_0$. Combining (\ref{S3S1P1E7}) and (\ref{S3S1P1E8}) yields (\ref{S3S1E3}).
\end{proof}

%-------------------------------------------------------------------------------------------------------------------------------------------------------------------------------------------------
% Section 3.2
%
%-------------------------------------------------------------------------------------------------------------------------------------------------------------------------------------------------
\subsection{Asymptotics of $f_{n,m}(\cdot|(m+n)z)$}\label{Section3.2}

We start with a useful definition.
\begin{definition}\label{DefM34}
Suppose that $f_X(\cdot)$ satisfies Assumptions C1-C4 and that $\beta > t > s > \alpha$ are given. Then in view of Lemmas \ref{S2S1L1} and \ref{S2S1L2} we know that $F(z):= G_z(u_z)$ is smooth on $(\alpha, \beta)$ and so for each $k \geq 0$ exists $M^{(k)}_{s,t} > 0$ such that $| F^{(k)}(z)| \leq M^{(k)}_{s,t}$ for all $z \in [s,t]$. 
\end{definition}

 We have the following asymptotic estimate for $f_{n,m}(\cdot|(m+n)z)$. 
\begin{proposition}\label{S3S1P2} Suppose that $f_X$ satisfies Assumptions C1-C4. Fix $s,t$ such that $\beta > t > s > \alpha$ and let $N_0$ be as in the statement of Proposition \ref{S3S1P1}. Then there exists $M > 0$ such that the following holds. Suppose that $m ,n \geq N_0$ are such that $|m-n| \leq 1$ and denote $N = n + m$. In addition, let $z, x$ be such that $xN/n, (z - x)N/m, z \in [s, t]$. Then we have 
\begin{equation}\label{S3S1E11}
f_{n,m}(Nx|Nz) = \frac{2}{\sqrt{2\pi N} \sigma_z} \cdot \exp \left(- N\cdot \frac{4}{2\sigma_z^2}\left[x-\frac{z}{2}\right]^2 + \delta_2(N,x,z) \right),
\end{equation}
where
\begin{equation}
\left| \delta_2(N,x,z) \right| \leq M \cdot \left(\frac{1}{\sqrt{N}} + N\left|x-\frac{z}{2}\right|^3 \right).
\end{equation}
The constant $M$ depends on $s,t$ and also on $f_X(\cdot)$, where the dependence on the latter is only through the constants in Definitions \ref{DefDelta}, \ref{DefK} and \ref{DefQ} as well as $M^{(3)}_{s,t}, M^{(4)}_{s,t}$ in Definition \ref{DefM34}.
\end{proposition}
\begin{proof}
Set $\phi = \frac{m}{n}$ and $\psi = \frac{n}{m}$. From Proposition \ref{S3S1P1} we know that for $m,n \geq N_0$ we have
\begin{equation}\label{S3S1P2E1}
\begin{split}
&f_{n,m}(Nx|Nz) = \frac{f_n(Nx) f_m(N(z-x))}{f_{N}(Nz)} = \frac{f_n(n[xN/n]) f_m(m[(z-x)N/m])}{f_{N}(Nz)} =\\
& e^{N \left( \frac{F[x(1+\phi)] }{1 + \phi}+ \frac{F[ (z-x)(1 + \psi)]}{1 + \psi} -  F(z)\right)} \cdot \frac{2 \sigma_z }{\sqrt{2 \pi N} \sigma_{x(1 + \phi)} \cdot \sigma_{(z-x)(1 + \psi)}} \cdot \exp \left[ O \left(\frac{1}{\sqrt{N}}\right) \right],
\end{split}
\end{equation}
where the constant in the big $O$ notation depends on $s,t$ and the constants in the statement of Proposition \ref{S3S1P1}.

Notice that $ F'(z) = \partial_z [\Lambda(z) - z u_z] = - u_z$, where the last equality used that $\Lambda'(u_z) = z$. In addition, differentiating the last expression shows that $\partial_z u_z = \frac{1}{\Lambda''(u_z ) } = \frac{1}{\sigma_z^2}$. This means that $F''(z) = - \frac{1}{\sigma_z^2}$ and $F'(z) = - u_z$. This shows that $F$ is a strictly concave function in $z$ and its second derivative is bounded from above by $-1/M_{s,t}$ as in Definition \ref{DefDelta}. 

Let us write $x = \frac{z}{1+ \phi} + r$ and denote 
$$h(r) := \frac{F(z + (1+\phi)r) }{1 + \phi}+ \frac{F( z- r(1 + \psi))}{1 + \psi} -  F(z).$$
Then $h(0) = h'(0) = 0$ and 
$$h''(r) = (1+\phi) F''(z + (1 + \phi)r) + (1+\psi)F''(z +(1+\psi)r), \mbox{ hence } h''(0)=  - \frac{2 + \phi + \psi}{\sigma_z^2} .$$ 
Next we have
$$h'''(r) = (1 + \phi)^2 F'''(z + (1+\phi) r) + (1 + \psi)^2 F'''(z +(1+\psi)r),$$
In view of Definition \ref{DefM34} there exists a constant $K$ depending only on $M^{(3)}_{s,t}$ such that $\left|h^{(3)}(r) \right| \leq K,$ provided $z + (1 + \phi)r, z +(1+\psi)r \in [s,t]$.
Then we see that 
\begin{equation}\label{S3S1P2E2}
\begin{split}
&e^{N \left( \frac{F[x(1+\phi)] }{1 + \phi}+ \frac{F[ (z-x)(1 + \psi)]}{1 + \psi} -  F(z)\right)} = \exp \left(N h \left( x - \frac{z}{1 +\phi} \right)\right) =  \\
&\exp \left(  -N \frac{2 + \phi + \psi}{2\sigma_z^2}\left[x-\frac{z}{1+ \phi}\right]^2 + O \left( N \left|x-\frac{z}{1+ \phi}\right|^3 \right) \right),
\end{split}
\end{equation}
where the constant in the big $O$ notation is just $K$.

We claim that
\begin{equation}\label{S3S1P2E3}
 \frac{\sigma_z^2}{ \sigma_{x(1 + \phi)} \cdot \sigma_{(z-x)(1 + \psi)}}=   \exp \left[O\left( \frac{1}{\sqrt{N}} + N \left| x - \frac{z}{1+\phi} \right|^3 \right) \right] .
\end{equation}
Combining (\ref{S3S1P2E1}), (\ref{S3S1P2E2}) and (\ref{S3S1P2E3}) gives (\ref{S3S1E11}). In the remainder we establish (\ref{S3S1P2E3}).\\

Squaring the left side of (\ref{S3S1P2E3}) and taking logarithm gives
$$ \log[-F''(x(1 + \phi))] + \log[-F''((z-x)(1 + \psi))] - 2\log[ -F''(z)].$$
Let us set $x = \frac{z}{1+\phi} + r$ and denote 
$$g(r) = \log[-F''(z + r(1 + \phi))] + \log[-F''(z - r(1 + \psi))] - 2\log[ -F''(z)].$$
Then $g(0) = 0$ and 
$$g'(r) = -\frac{(1 + \phi)F'''(z + r(1 + \phi))}{F''(z + r(1 + \phi))} +  \frac{(1 + \psi)F'''(z - r(1 + \psi))}{F''(z + r(1 + \psi))}.$$
This implies that 
$$g'(0)= (\psi - \phi)\frac{F^{'''}(z )}{F''(z )}.$$
As discussed before $|F''(z)|\geq 1/M_{s,t}$ for all $z\in \left[ s,t\right]$ and so we conclude that $|g'(0)| \leq \frac{K_2}{N}$ for some constant $K_2$ that depends on $s,t, M_{s,t}$ and $M^{(3)}_{s,t}$. On the other hand, it is easy to see that $|g''(r)|\leq K_3$ for some constant that depends on $s,t, M_{s,t}, M^{(3)}_{s,t}$ and $M^{(4)}_{s,t}$. This implies 
$$|g(r)|  \leq r \cdot\frac{K_2}{N} + r^2 K_3 ,$$
which implies that 
$$ \frac{\sigma_z^2}{ \sigma_{x(1 + \phi)} \cdot \sigma_{(z-x)(1 + \psi)}} =    \exp \left[O\left(  \frac{1}{N} \left| x - \frac{z}{1+\phi} \right| +\left| x - \frac{z}{1+\phi} \right|^2 \right) \right] .$$
The latter inequality implies (\ref{S3S1P2E3}) and concludes the proof of the proposition.
\end{proof}

%-------------------------------------------------------------------------------------------------------------------------------------------------------------------------------------------------
% Section 3.3
%
%-------------------------------------------------------------------------------------------------------------------------------------------------------------------------------------------------
\subsection{Tails of $f_{n,m}(\cdot|(m+n)z)$}\label{Section3.3}
In this section we will further assume that $f_X(\cdot)$ satisfies Assumption C5 and use that to deduce tail estimates for $f_{n,m}(\cdot|(m+n)z)$. We start with a couple of lemmas.

\begin{lemma}\label{S3S1L1} Suppose that $f_X$ satisfies Assumption C5. Then for all  $N \geq 1$
\begin{equation}\label{DensDec}
f_N(x) \leq \begin{cases} W^N e^{-dN^{-1} x^2} \mbox{ for all $x \geq 0$ if C5.1 holds and } \\
                                       W^N  e^{-d N^{-1} x^2} \mbox{ fro all $x \leq 0$ if C5.2 holds,}
                       \end{cases}
\end{equation}
where $W = D\frac{\sqrt{\pi}}{\sqrt{d}} + 1 + D $.
\end{lemma}
\begin{proof}By symmetry it is clearly enough to consider the case when C5.A.1 holds. Suppose that $ C_1, C_2, c_1, c_2 > 0$ and $h_1, h_2$ are probability density functions such that
$$h_i(x) \leq C_i e^{-c_ix^2} \mbox{ for all $x \geq 0$  and $i = 1,2$}.$$
In addition, set $g(y) = \int_{\mathbb{R}} h_1(y-x) h_2(x) dx$ and $h_i^1(x) = h_i(x) \cdot {\bf 1}_{x \geq 0}$ and $h_i^2 = h_i(x) \cdot {\bf 1}_{x < 0}$ for $i = 1,2$. We thus obtain for $y \geq 0$  
$$g(y) = \int_0^\infty h^1_1(y-x) h_2^1(x)dx +  \int_0^\infty h^2_1(y-x) h_2^1(x)dx +  \int_0^\infty h^1_1(x) h_2^2(y-x)dx \leq $$
$$ C_1 C_2 \int_0^y e^{-c_1 (y-x)^2} e^{-c_2x^2}dx + C_2\int_y^\infty h_1^2(y-x)e^{-c_2x^2}dx + C_1\int_y^\infty e^{-c_1 x^2} h_2^2(y-x) dx.$$ 
Using that $h_i$ are probability density functions we get
$$\int_y^\infty e^{-c_ix^2}h_j^2(y-x) dx \leq e^{-c_i y^2} .$$
Using that the convolution of two Gaussian densities is again a Gaussian density we get
\begin{equation}\label{GSNconv}
\int_0^y e^{-c_1 (y-x)^2} e^{-c_2x^2}dx \leq \int_{\mathbb{R}}e^{-c_1 (y-x)^2} e^{-c_2x^2}dx =  \frac{\sqrt{\pi}}{\sqrt{c_1+c_2}} \exp \left( - \frac{y^2 c_1 c_2}{c_1 + c_2} \right).
\end{equation}
Combining all of the above we get
\begin{equation}\label{DensIneq}
g(y) \leq C_1C_2 \frac{\sqrt{\pi}}{\sqrt{c_1+c_2}} \exp \left( - \frac{y^2 c_1 c_2}{c_1 + c_2} \right) + C_2 e^{-c_2 y^2} + C_1 e^{-c_1y^2}.
\end{equation}

We now proceed to prove (\ref{DensDec}) by induction on $N$ with base case $N = 1$, being true by assumption. Suppose the result holds true for $N$. Setting $h_1(x) = f_X(x)$ and $h_2(x) = f_N(x)$ and applying (\ref{DensIneq}) we obtain for any $y \geq 0$ that
$$f_{N+1}(y)\leq \frac{D W^{N}  \sqrt{\pi}}{\sqrt{d+d/N}} \exp \left( - \frac{y^2 d (d/N)}{d + d/N} \right) + W^{N} e^{-(d/N) y^2} + D e^{-dy^2}\leq $$
$$\leq W^N\left[D\frac{\sqrt{\pi}}{\sqrt{d}} + 1 + D \right]e^{-d(N+1)^{-1} y^2} = W^{N+1} e^{-d(N+1)^{-1} y^2}.$$
This proves (\ref{DensDec}) for the case $N+1$ and the general result proceeds by induction on $N$.
\end{proof}

\begin{lemma}\label{S3S1L2} Suppose that $f_X$ satisfies Assumption C5 and $\alpha > -\infty$ or $\beta < \infty$ or both. Then for all  $N \geq 1$
$$f_N(x) \leq \begin{cases} \frac{L^N}{(N-1)!} \left(x-N\alpha  \right)^{N-1}  \mbox{ for all $x > N \alpha$ if $\alpha > -\infty$ } \\
                                              \frac{L^N}{(N-1)!} \left(N\beta - x  \right)^{N-1}  \mbox{ for all $x < N \beta$ if $\beta < \infty$ }.
                       \end{cases}$$
\end{lemma}
\begin{proof}By symmetry it is clearly enough to consider the case $\alpha > -\infty$ and prove the first statement of the lemma. By shifting $X$ by $-\alpha$ we may assume that $\alpha = 0$. We proceed by induction on $N$ with base case $N = 1$ being true by assumption. We now suppose that the result holds true for $N$ and let $y > 0$. Then
$$f_{N+1}(y) = \int_0^y f_N(x) f_1(y-x) \leq \int_0^y \frac{L^N}{(N-1)!}  x^{N-1} \cdot L dx =  \frac{L^{N+1}}{N!} y^{N} .$$
This proves the induction step and the general result follows by induction.
\end{proof}

We next summarize a couple of parameter choices for future use.

\begin{definition}\label{DefN2} Suppose that $f_X(\cdot)$ satisfies Assumptions C1-C5. Fix $t,s$ such that $\beta > t > s > \alpha$. Then in view of Proposition \ref{S3S1P1} we can find $C_1 > 1$ sufficiently large depending on the constants  in Definitions \ref{DefDelta}, \ref{DefK} and \ref{DefQ} and $M^{(0)}_{s,t}$ in Definition \ref{DefM34} so that
$$C_1^{-N} \leq f_N(Nz)$$
for all $z \in [s,t]$ and $N \geq N_0$ (where $N_0$ is as in the statement of Proposition \ref{S3S1P1}). 

We can also find $\epsilon_1 > 0$ sufficiently small so that $48 C_1^2 L \cdot \epsilon_1 \leq 1$, $s \geq \alpha + 3\epsilon_1$ and $t \leq \beta - 3\epsilon_1$, where $L$ is as in Assumption C5. 

We can also find $R_1 > 1$ sufficiently large so that 
$$[s,t] \subset [ - R_1, R_1 ] \mbox{ and } W C_1 e^{-d R_1^2/2} \leq 1,$$
where $W = D \frac{\sqrt{\pi}}{\sqrt{d}} + 1 + D$ with $D,d$ as in Assumption C5.

Finally, given the above choice of $\epsilon_1$ and $R_1$ we can define the variables $\hat{s}, \hat{t}$ as follows:
\begin{itemize}
\item $\hat{s} = \alpha + \epsilon_1$ and $\hat{t} = \beta - \epsilon_1$ if $\alpha > -\infty$ and $\beta < \infty$;
\item $\hat{s} = \alpha + \epsilon_1$ and $\hat{t} = 3\max(t,0) - \alpha - \epsilon_1$ if $\alpha > -\infty$ and $\beta = \infty$; 
\item $\hat{s} = 3 \min(0, s) - \beta+ \epsilon_1$ and $\hat{t} = \beta - \epsilon_1$ if $\alpha = -\infty$ and $\beta < \infty$; 
\item $\hat{s} = -6R_1$ and $\hat{t} = 6R_1$ if $\alpha = -\infty$ and $\beta =\infty$.
\end{itemize}  
\end{definition}

\begin{definition}\label{DefParC}  Suppose that $f_X(\cdot)$ satisfies Assumptions C1-C5. Fix $t,s$ such that $\beta > t > s > \alpha$ and let $C_1, \epsilon_1, R_1, \hat{s}$ and $\hat{t}$ be as in Definition \ref{DefN2}. For future reference we summarize the following list of constants:
\begin{enumerate}
\item the constants in Assumptions C1 and C5;
\item $C_1, \epsilon_1, R_1, \hat{t}, \hat{s}$ as in Definition \ref{DefN2};
\item $M_{\hat{s},\hat{t}}, m_{\hat{s},\hat{t}}$, $\delta_{\hat{s}, \hat{t}}$ as in Definition \ref{DefDelta};
\item $K_{\hat{s}, \hat{t}}, p_{\hat{s}, \hat{t}}$ as in Definition \ref{DefK};
\item $q_{\hat{s},\hat{t}} $ as in Definition \ref{DefQ};
\item $M_{\hat{s}, \hat{t}}^{(0)}, M_{\hat{s}, \hat{t}}^{(1)}, M_{\hat{s}, \hat{t}}^{(2)}, M_{\hat{s}, \hat{t}}^{(3)}, M^{(4)}_{\hat{s}, \hat{t}}$ from Definition \ref{DefM34}.
\end{enumerate}
\end{definition}

We can now prove the following complement to Proposition \ref{S3S1P2}, which establishes tail estimates for the midpoint density of a continuous random walk bridge. 
\begin{proposition}\label{S3S1P3}
Suppose that $f_X(\cdot)$ satisfies Assumptions C1-C5. Fix $s,t$ such that $\beta > t > s > \alpha$. There exist constants $A,a > 0$ and $N_1 \in \mathbb{N}$, such that the following holds. Suppose that $m ,n \geq  N_1$ are such that $|m-n| \leq 1$ and denote $N = n + m$. In addition, let $z \in [s,t]$. Then we have for any $x \in \mathbb{R}$
\begin{equation}\label{S3S1E12}
f_{n,m}(Nx|Nz) \leq A \cdot \exp \left(-  aN \left[x-\frac{z}{2}\right]^2  \right).
\end{equation}
The constants $a,A$ and $N_1$ depend on the values $s,t$ and the function $f_X(\cdot)$, where the dependence on the latter is through the constants in Definition \ref{DefParC}.
\end{proposition}
\begin{proof}
Denote $\phi = \frac{m}{n}$ and $\psi = \frac{n}{m}$. For clarity we will split the proof into several cases.

{\bf \raggedleft Case 1.} Suppose first that $\alpha > -\infty$. From the first line of (\ref{S3S1P2E1}) we know that  
\begin{equation}\label{S3S1P3E1}
f_{n,m}(Nx|Nz) =  \frac{f_n(Nx) \cdot f_m(N(z-x))}{f_N(Nz)},
\end{equation}
and the latter expression is zero unless $Nx \geq n\alpha$ and $N(z-x) \geq m\alpha$. We will assume that $x$ satisfies these inequalities as otherwise (\ref{S3S1E12}) trivially holds for any $A, a > 0$. From Definition \ref{DefN2} we know that for all $N \geq N_0$ we have
\begin{equation}\label{S3S1P3E2}
f_{n,m}(Nx|Nz) \leq  C_1^N f_n(Nx) \cdot f_m(N(z-x)).
\end{equation}
In particular, since $f_n$ and $f_m$ are uniformly bounded by a constant (namely $L$), we see that we can make (\ref{S3S1E12}) true for all small $N \geq N_0$ by choosing $A$ sufficiently large and $a \leq 1$. We will thus focus on showing (\ref{S3S1E12}) for sufficiently large $N \geq N_0$. \\

Suppose that $Nx \leq  n\alpha + n\epsilon_1,$ where $\epsilon_1$ is as in Definition \ref{DefN2}. From Lemma \ref{S3S1L2} and the inequality
\begin{equation}\label{ineqGamma}
 \frac{1}{(N-1)!} = \frac{1}{\Gamma(N)} \leq \frac{e^{N-1}}{N^{N-1}},
\end{equation}
which can be found in \cite{LC} we conclude that 
$$f_{n}(Nx) \leq L \left(\frac{L e n\epsilon_1}{n}\right)^{n-1} \leq L \left (L \epsilon_1e\right)^{n-1} .$$
The above, combined with the definition of $\epsilon_1$ and (\ref{S3S1P3E2}) imply
$$f_{n,m}(Nx|Nz) \leq  C_1^N \cdot L \left(L \epsilon_1 e \right)^{n-1} \cdot L \leq 16 C_1^4 L^2 2^{-N},$$
while for $N \geq N_1$ with $N_1$ sufficiently large depending on $\alpha$ we have
$$A \cdot \exp \left(-  aN \left[x-\frac{z}{2}\right]^2  \right) \geq  A \cdot  \exp \left( - a N \frac{\epsilon_1^2}{4} \right).$$
It follows from the above inequalities that (\ref{S3S1E12}) holds provided we take $A \geq  16 C_1^4 L^2$, $a$ sufficiently small and $Nx \in [n\alpha, n\alpha + n\epsilon_1]$. Analogous arguments applied to $z-x$ in place of $x$ show that for the same $A$ and $a$ we have (\ref{S3S1E12}) provided that $N(z-x) \in [m\alpha, m\alpha + m\epsilon_1]$. We may thus assume that $Nx\geq n\alpha + n\epsilon_1$ and $N(z-x) \geq m\alpha + m\epsilon_1$.\\

We next consider the cases $\beta = \infty$ and $\beta < \infty$ separately starting with the former. 

{\bf \raggedleft Case 1.A.} If $\beta = \infty$ then we let $N_1$ be sufficiently large so that 
 $N_1 \geq N_0$, where $N_0$ is as in the statement of Proposition \ref{S3S1P1} for the values $\hat{s} = \alpha + \epsilon_1$ and $\hat{t} = 3\max(t,0) - \alpha - \epsilon_1$.

Then from Proposition \ref{S3S1P1} (see also equation (\ref{S3S1P2E1})) we know that we have for $m,n \geq N_1$ and $Nx\geq n\alpha + n\epsilon_1$ and $N(z-x) \geq m\alpha + m\epsilon_1$ that
\begin{equation}\label{S3S1P3E3}
\begin{split}
&f_{n,m}(Nx|Nz) \leq C_2 \exp\left[N \left( \frac{F(x(1+\phi)) }{1 + \phi}+ \frac{F ((z-x)(1 + \psi))}{1 + \psi} -  F(z)\right) \right],
\end{split}
\end{equation}
where the constant $C_2$ depends on $m_{\hat{s}, \hat{t}}$ and $M_{\hat{s}, \hat{t}}$ as in Definition \ref{DefDelta} for the values $\hat{s} = \alpha + \epsilon_1$ and $\hat{t} = 3\max(t,0) - \alpha - \epsilon_1$. As in the proof of Proposition \ref{S3S1P2} we write $x = \frac{z}{1+ \phi} + r$ and denote 
$$h(r) = \frac{F[z + (1+\phi)r] }{1 + \phi}+ \frac{F[ z- r(1 + \psi)]}{1 + \psi} -  F(z).$$
Then $h(0) = h'(0) = 0$ and 
$$h''(r) = (1+\phi) F''(z + (1 + \phi)r) + (1+\psi)F''(z +(1+\psi)r) = - \left[ \frac{1 + \phi}{\sigma^2_{z + (1 +\phi)r}} +  \frac{1 + \psi}{\sigma^2_{z + (1 +\psi)r}} \right] .$$ 
The above shows that $h(r)$ is strictly concave and its second derivative is less than $-d_2$ for some $d_2 > 0$ (depending on $M_{\hat{s}, \hat{t}}$ alone) on the interval $z + (1 +\phi)r, z + (1 +\psi)r \in [\hat{s}, \hat{t}]$. Putting this in (\ref{S3S1P3E3}) we conclude
$$f_{n,m}(Nx|Nz) \leq C_2 \exp \left( - \frac{d_2}{2} \cdot N \cdot \left[x - \frac{z}{1 + \phi} \right]^2 \right),$$
which implies (\ref{S3S1E12}) in this case.\\

{\bf \raggedleft Case 1.B.} We suppose that $\beta < \infty$. As before we know that (\ref{S3S1E12}) holds for any $A, a > 0$ if $Nx > n\beta$ or $N(z-x) > m \beta$ and so we may assume that $Nx \leq n \beta$ and $N(z-x) \leq m\beta$.

 Suppose that $Nx \geq n\beta - n\epsilon_1$. Then from Lemma \ref{S3S1L2}, (\ref{S3S1P3E2}) and (\ref{ineqGamma}) we know that 
$$f_{n,m}(Nx|Nz) \leq  C_1^N \cdot L \left(L e \epsilon_2 \right)^{n-1} \cdot L \leq 16 C^4 L^2 2^{-N},$$
while for $N \geq N_1$ with $N_1$ sufficiently large depending on $\beta$ we have
$$A \cdot \exp \left(-  aN \left[x-\frac{z}{2}\right]^2  \right) \geq  A \exp \left( - a N \frac{\epsilon_1^2}{4} \right).$$
It follows from the above inequalities that (\ref{S3S1E12}) holds provided we take $A \geq  16 C^4 L^2$, $a$ sufficiently small and $Nx \in [ n\beta - n\epsilon_1, n\beta]$. Analogous arguments applied to $z-x$ in place of $x$ show that for the same $A$ and $a$ we have (\ref{S3S1E12}) provided that $N(z-x) \in [ m\beta - m\epsilon_1, m\beta]$. We may thus assume that $Nx \in [n\alpha + n \epsilon_1, n\beta - n \epsilon_1]$ and $N(z-x) \in [ m\alpha + m \epsilon_1, m\beta - m\epsilon_1].$

We let $N_1$ be sufficiently large so that $N_1 \geq N_0$, where $N_0$ is as in the statement of Proposition \ref{S3S1P1} for the values $\hat{s} = \alpha + \epsilon_1$ and $\hat{t} = \beta - \epsilon_1$.

 Then from Proposition \ref{S3S1P1} (see also equation (\ref{S3S1P2E1}) ) we know that for $m,n \geq N_1$ and  $Nx \in [n\alpha + n \epsilon_1, n\beta - n \epsilon_1]$ and $N(z-x) \in [ m\alpha + m \epsilon_1, m\beta - m\epsilon_1]$ that
\begin{equation}\label{S3S1P3E4}
\begin{split}
&f_{n,m}(Nx|Nz) \leq  C_2 \exp\left[N \left( \frac{F(x(1+\phi)) }{1 + \phi}+ \frac{F ((z-x)(1 + \psi))}{1 + \psi} -  F(z)\right) \right],
\end{split}
\end{equation}
where the constant $C_2$ depends on $m_{\hat{s}, \hat{t}}$ and $M_{\hat{s}, \hat{t}}$ as in Definition \ref{DefDelta} for the values $\hat{s} = \alpha + \epsilon_1$ and $\hat{t} = \beta - \epsilon_1$.

Repeating the same arguments that follow (\ref{S3S1P3E3}) and using the strict negativity of $F''(z)$ for $z \in [\hat{s}, \hat{t}]$ we conclude that 
$$f_{n,m}(Nx|Nz) \leq C_2 \exp \left( - \frac{d_2}{2}\cdot N \cdot \left[x - \frac{z}{1 + \phi} \right]^2 \right),$$
which implies (\ref{S3S1E12}) in this case. Overall, we conclude (\ref{S3S1E12}) under the condition that $\alpha > -\infty$. 

{\bf \raggedleft Case 2.} Suppose now that $\alpha = -\infty$. 

{\bf \raggedleft Case 2.A.} If $\beta < \infty$ then we can conclude (\ref{S3S1E12}) by the same arguments as those in Case 1.A.

{\bf \raggedleft Case 2.B.} Suppose that $\beta = \infty$. By symmetry it suffices to consider the case when Assumption C5.1 holds. Let $R_1$ be as in Definition \ref{DefN2}. Then from Lemma \ref{S3S1L1} and (\ref{S3S1P3E2}) we know that for $x \geq R_1$  and $N \geq N_0$ 
$$f_{n,m}(Nx|Nz) \leq  C_1^N \cdot W^n e^{-d  N x^2}\cdot L \leq L \cdot e^{-dNx^2/2},$$
while 
$$A \cdot \exp \left(-  aN \left[x-\frac{z}{2}\right]^2  \right) \geq  A \exp \left( - a N [ x + R_1/2 ]^2 \right).$$
It follows from the above inequalities that (\ref{S3S1E12}) holds provided we take $A \geq  L$, $a$ sufficiently small (say $a \leq d/8$) and $x \geq R_1$. Analogous arguments applied to $z-x$ in place of $x$ show that for the same $A$ and $a$ we have (\ref{S3S1E12}) provided that $z-x \geq R_1$. We may thus assume that $x, z-x \in [ -2R_1, 2R_1]$. 

We let $N_1$ be sufficiently large so that  $N_1 \geq N_0$, where $N_0$ is as in the statement of Proposition \ref{S3S1P1} for the values $\hat{s} = -6R_1$ and $\hat{t} = 6R_1$.
Then from Proposition \ref{S3S1P1} (see also equation (\ref{S3S1P2E1}))  we know that for $m,n \geq N_1$ and $x \in[-2R_1, 2R_1]$
\begin{equation}\label{S3S1P3E5}
\begin{split}
&f_{n,m}(Nx|Nz) \leq  C_2 \exp\left[N \left( \frac{F(x(1+\phi)) }{1 + \phi}+ \frac{F ((z-x)(1 + \psi))}{1 + \psi} -  F(z)\right) \right],
\end{split}
\end{equation}
where the constant $C_2$ depends on $m_{\hat{s}, \hat{t}}$ and $M_{\hat{s}, \hat{t}}$ as in Definition \ref{DefDelta} for the values $\hat{s} = -6R_1$ and $\hat{t} = 6R_1$. Repeating the same arguments that follow (\ref{S3S1P3E3}) and using the strict negativity of $F''(z)$ for $z \in [\hat{s}, \hat{t}]$ we conclude that 
$$f_{n,m}(Nx|Nz) \leq C \exp \left( - \frac{d_2}{2}\cdot N\cdot  \left[x - \frac{z}{1 + \phi} \right]^2 \right),$$
which implies (\ref{S3S1E12}) in this case. Overall, we conclude (\ref{S3S1E12}) when $\alpha = -\infty$ and $\beta = \infty$.
\end{proof}

\section{Midpoint distribution: Discrete case}\label{Section4}

We continue with the same notation as in Section \ref{Section2.2}. To ease the notation a bit we will write $M, \phi$ and $\Lambda$ instead of $M_X, \phi_X$ and $\Lambda_X$. Let $p_{n,m}(\cdot|l)$ be the distribution of $S_m$ conditioned on $S_{n+m} = l$. Our goal in this section is to obtain several asymptotic statements about the distribution $p_{m,n}(\cdot | l )$ and we start by analyzing $p_N(l)$.

%-------------------------------------------------------------------------------------------------------------------------------------------------------------------------------------------------
% Section 4.1
%
%-------------------------------------------------------------------------------------------------------------------------------------------------------------------------------------------------
\subsection{Asymptotics of $p_N(l)$}\label{Section4.1}
In this section we assume that $p_X(\cdot)$ satisfies Assumptions D1-D3. For a fixed $z \in (A^*, B^*)$ we define 
\begin{equation}\label{S4S1E1}
G_z(u) = \Lambda(u) - z \cdot u, \mbox{ for $u \in (A_{\Lambda}, B_{\Lambda})$.}
\end{equation}

\begin{definition}\label{DefDeltaDisc} Suppose that we are given $s,t \in \mathbb{R}$ such that $\alpha < s < t < \beta$, where $\alpha, \beta$ are as in Assumption D1. In addition, we denote $S = (\Lambda')^{-1}(s)$ and $T = (\Lambda')^{-1}(t)$ -- these quantities are well-defined in view of Lemma \ref{S2S2L2}. By Lemma \ref{S2S2L1} there exist $\infty > M_{s,t} \geq m_{s,t} > 0$ such that $M_{s,t} \geq \Lambda''(y) \geq m_{s,t}$ for all $y \in [S,T]$. We can pick $\delta_{s,t} > 0 $ sufficiently small (depending on $s, t$ and $p_X(\cdot)$) so that  
\begin{enumerate}
\item If $D_{\delta_{s,t}}(S,T) : = \{ z \in \mathbb{C}: d(z, [S,T]) < \delta_{s,t}\}$ then $\overline{D}_{\delta_{s,t}}(S,T) \subset \{z \in \mathbb{C}: A_{\Lambda} < Re(z) < B_{\Lambda} \}$;
\item $Re [M_X(u)] > 0$ for all $u \in  \overline{D}_{\delta_{s,t}}(S,T) $;
\item $\delta_{s,t} < 1/2$;
\item $8\delta_{s,t} \cdot |\log (M_X(u))| \leq m_{s,t}$ for all  $u \in \overline{D}_{\delta_{s,t}}(S,T)$.
\end{enumerate}
\end{definition}

\begin{definition}\label{DefQDisc} Suppose that we are given $s,t \in \mathbb{R}$ such that $\alpha < s < t < \beta$, where $\alpha, \beta$ are as in Assumption D1. Suppose that $\delta_{s,t}$ satisfies the conditions in Definitions \ref{DefDeltaDisc} and let $\epsilon_{s,t} = \delta_{s,t}^4$. Then we can find $q_{s,t} \in (0,1)$ (depending on $s,t, \delta_{s,t}$ and $f_X(\cdot)$)such that for every $z \in [s,t]$ and $y \in [\epsilon_{s,t}, \pi]$ we have
$$\left| \mathbb{E} \left[e^{(u_z + iy)X}\right] \right|  e^{-zu_z} e^{G_z(u_z)} \leq q_{s,t}. $$
To see why the above is true, notice that
$$\left| \mathbb{E} \left[e^{(u_z + iy)X}\right] \right|  e^{-zu_z} e^{G_z(u_z)} < \mathbb{E} \left[ \left|e^{(u_z + iy)X} \right|\right] e^{-zu_z} e^{G_z(u_z)} = 1,$$
where the above inequality is strict for any $y \neq 0$ as the contrary would imply $X \in 2\pi y^{-1} \cdot \mathbb{Z}$ almost surely, which is not true. This combined with the continuity of $ \mathbb{E} \left[e^{(u_z + iy)X}\right]$ in $y$ and $z$ ensures the existence of $q_{s,t}$ with the desired properties.
\end{definition}

We are interested in proving the following statement.
\begin{proposition}\label{S4S1P1} Suppose that $p_X$ satisfies Assumptions D1-D3. Fix $\beta > t > s > \alpha$. Then there exists $N_0$ such that if $N \geq N_0$, $l \in \mathbb{Z}$ and $z = l/N \in [s,t]$ one has
\begin{equation}\label{S4S1E3}
p_N(l) = \frac{1}{\sqrt{2\pi N }\sigma_z} \cdot \exp \left( N G_z(u_z) + \delta_1(z,N) \right), \mbox{ where $\delta_1(z,N) = O(N^{-1/2})$}.
\end{equation}
The number $N_0$ and the constant in the big $O$ notation depend on $f_X, s$ and $t$ only through the constants $\delta_{s,t}, m_{s,t}$ and $q_{s,t}$ as in Definitions \ref{DefDeltaDisc} and \ref{DefQDisc}.
\end{proposition}
\begin{proof}
To simplify the notation, we drop the dependence on $X$. For any $l \in \mathbb{Z}$ and $N \geq 1$ we have
$$p_N (l)  =  \frac{1}{2\pi} \int_{-\pi}^{\pi} e^{-it l} \cdot \left( \phi(t) \right)^N dt.$$
Performing the change of variables $u = it$ we see that 
\begin{equation}\label{S4P1E1}
p_N(l)= \frac{1}{2\pi i} \int_{-i \pi}^{ i \pi} M^N(u) e^{ - ul} du.
\end{equation}

Consider the rectangular contour $R$ consisting of straight segments connecting $-i\pi$ to $u_z - i\pi$, to $u_z + i\pi$, to $i\pi$ back to $-i\pi$ with a positive orientation. It follows by Lemma \ref{S2S2L1} that $ M^N(u) e^{ - ul}$ is analytic in a neighborhood enclosing that rectangle and so by Cauchy's theorem the integral over $R$ vanishes. In addition, the integral over the top segment and the bottom segment are equal and hence their sum vanishes (as they have opposite orientation). The conclusion is
\begin{equation}\label{S4P1E2}
p_N(l) = \frac{e^{NG_z(u_z)}}{2\pi i} \int_{u_z-i \pi}^{u_z + i \pi}  M(u)^N e^{-uNz}e^{-NG_z(u_z)} du.
\end{equation}

For the given $s,t$ as in the statement of the proposition we define $\delta_{s,t},  m_{s,t}, \epsilon_{s,t}$ and $q_{s,t}$ as in Definitions \ref{DefDeltaDisc} and \ref{DefQDisc}. To ease notation we will drop $s,t$ from the notation for these quantities. We will also denote by $C_{s,t}$ the supremum of $|\log(M(u))| $ as $u$ varies over $\overline{D_\delta}$. Notice that by construction we have
$$\epsilon < \delta/2 \mbox{ and } \epsilon \cdot 8C_{s,t} \cdot \delta^{-3} < m.$$

From (\ref{S4P1E2}) we have $p_N(l) = (I) + (II), \mbox{ where }$
\begin{equation}\label{S4P1E3}
\begin{split}
&(I) = \frac{e^{NG_z(u_z)}}{2\pi i} \int_{u_z -i \epsilon}^{u_z + i \epsilon} e^{N[ G_z(u) - G_z(u_z)]} du, (II) = \frac{e^{NG_z(u_z)}}{2\pi i} \int_{u_z -i \pi}^{u_z - i \epsilon}\left[M(u)e^{-uz}e^{-G_z(u_z)} \right]^N \hspace{-2mm}du  \\
&+ \frac{e^{NG_z(u_z)}}{2\pi i} \int^{u_z  + i \pi}_{u_z + i \epsilon}\left[M(u)e^{-uz}e^{-G_z(u_z)} \right]^N  \hspace{-2mm} du.
\end{split}
\end{equation}

Arguing as in the proof of Proposition \ref{S3S1P1}, we have for $N_0$ sufficiently large and $N \geq N_0$ 
\begin{equation}\label{S4P1E7}
(I) =\frac{e^{NG_z(u_z)}}{2\pi \sigma_z \sqrt{N}} \left[1 + O \left( \frac{1}{\sqrt{N}} \right) \right],
\end{equation}
where the constant in the big O notation depends on the constants in this proposition.

We next forcus on estimating (II). Suppose that $\pm y \in [\epsilon, \pi]$. Then by definition we have 
$$\left| M(u_z + iy)  e^{-z(u_z+iy)} e^{G_z(u_z)} \right| \leq q.$$
The above implies that
\begin{equation}\label{S4P1E8}
|(II)| \leq \frac{e^{NG_z(u_z)}}{2\pi }\cdot 2\pi q^{N}\leq \frac{e^{NG_z(u_z)}}{2\pi \sigma_z N },
\end{equation}
where the last inequality holds provided $N_0$ is sufficiently large and $N \geq N_0$. Combining (\ref{S4P1E7}) and (\ref{S4P1E8}) yields (\ref{S4S1E3}).
\end{proof}

%-------------------------------------------------------------------------------------------------------------------------------------------------------------------------------------------------
% Section 4.2
%
%-------------------------------------------------------------------------------------------------------------------------------------------------------------------------------------------------

\subsection{Asymptotics of $p_{n,m}(\cdot|l)$}\label{Section4.2}

We start with a useful definition.
\begin{definition}\label{DefM34Disc}
Suppose that $p_X(\cdot)$ satisfies Assumptions D1-D3 and that $\beta > t > s > \alpha$ are given. Then in view of Lemmas \ref{S2S2L1} and \ref{S2S2L2} we know that $F(z):= G_z(u_z)$ is smooth on $(\alpha, \beta)$ and so for each $k \geq 0$ exists $M^{(k)}_{s,t} > 0$ such that $| F^{(k)}(z)| \leq M^{(k)}_{s,t}$ for all $z \in [s,t]$. 
\end{definition}

 We have the following asymptotic estimate for $p_{n,m}(\cdot|l)$. 
\begin{proposition}\label{S4S1P2} Suppose that $p_X$ satisfies Assumptions D1-D3. Fix $s,t$ such that $\beta > t > s > \alpha$ and let $N_0$ be as in the statement of Proposition \ref{S4S1P1}. Then there exists $M > 0$ such that the following holds. Suppose that $m ,n \geq N_0$ are such that $|m-n| \leq 1$ and denote $N = n + m$. In addition, let $k, l \in \mathbb{Z}$ be such that if $z:= l/N$ and $x := k/N$, then $z, xN/n$ and $(z-x)N/m \in [s, t]$. Then
\begin{equation}\label{S4S1E11}
p_{n,m}(k|l) = \frac{2}{\sqrt{2\pi N} \sigma_z} \cdot \exp \left(- N\cdot \frac{4}{2\sigma_z^2}\left[x-\frac{z}{2}\right]^2 + \delta_2(N,x,z) \right),
\end{equation}
where
\begin{equation}
\left| \delta_2(N,x,z) \right| \leq M \cdot \left(\frac{1}{\sqrt{N}} + N\left|x-\frac{z}{2}\right|^3 \right).
\end{equation}
The constant $M$ depends on $s,t$ and also on $p_X(\cdot)$, where the dependence on the latter is only through the constants in the statement of Proposition \ref{S4S1P1} and $M^{(3)}_{s,t}, M^{(4)}_{s,t}$ in Definition \ref{DefM34Disc}.
\end{proposition}
\begin{proof}
Set $\phi = \frac{m}{n}$ and $\psi = \frac{n}{m}$. From Proposition \ref{S4S1P1} we know that for $m,n \geq N_0$ we have
\begin{equation}\label{S4S1P2E1}
\begin{split}
&p_{n,m}(k|l) = \frac{p_n(k)p_m(l - k)}{p_N ( l)} = \\
& e^{N \left( \frac{F[x(1+\phi)] }{1 + \phi}+ \frac{F[ (z-x)(1 + \psi)]}{1 + \psi} -  F(z)\right)} \cdot \frac{2 \sigma_z }{\sqrt{2 \pi N} \sigma_{x(1 + \phi)} \cdot \sigma_{(z-x)(1 + \psi)}} \cdot \exp \left[ O \left(\frac{1}{\sqrt{N}}\right) \right],
\end{split}
\end{equation}
where the constant in the big $O$ notation depends on $s,t$ and the constants in the statement of Proposition \ref{S4S1P1}. From here the proof of the proposition follows the same arguments as in the proof of Proposition \ref{S3S1P2}.
\end{proof}

%-------------------------------------------------------------------------------------------------------------------------------------------------------------------------------------------------
% Section 4.3
%
%-------------------------------------------------------------------------------------------------------------------------------------------------------------------------------------------------
\subsection{Tails of $p_{n,m}(\cdot|l)$}\label{Section4.3}
In this section we will further assume that $p_X(\cdot)$ satisfies Assumption D4 and use that to deduce tail estimates for $p_{n,m}(\cdot|l)$. We start with a couple of lemmas.

\begin{lemma}\label{S4S1L1} Suppose that $p_X$ satisfies Assumption D4. Then for all $N \geq 1$ and $x \in \mathbb{Z}$
$$p_N(x) \leq \begin{cases} W^N e^{-dN^{-1} x^2} \mbox{ for all $x \geq 0$ if D4.1 holds and } \\
                                       W^N  e^{-d N^{-1} x^2} \mbox{ fro all $x \leq 0$ if D4.2 holds,}
                       \end{cases}$$
where $W = D\frac{\sqrt{\pi}}{\sqrt{d}} + 1 + 2D $.
\end{lemma}
\begin{proof}By symmetry it is clearly enough to consider the case when Assumption D4.1 holds. We proceed by induction on $N$ with base case $N = 1$ being true by assumption. Suppose the result holds true for $N$ and let $y \geq 0$. Then we have
$$p_{N+1}( y) = \sum_{x = 0}^y p_N( x) p_1(y - x) + \sum_{x = y}^\infty p_N(x)p_1(  y-x) +   \sum_{x = y}^\infty p_N(y - x) p_1( x). $$ 
By induction hypothesis and Assumption D4.1 we have
$$ \sum_{x = y}^\infty p_N(x) p_1(y-x) \leq W^{N} e^{-dN^{-1} y^2} \mbox{ and }\sum_{x = y}^\infty p_N(y - x) p_1( x) \leq D e^{-d y^2}.$$
Denote $f(x) =  e^{-d N^{-1}x^2} e^{-d(x-y)^2}$ and note that the function has a unique maximum on $[0, y]$, given by $x_{\max} = \frac{Ny}{N+1}$, and $f(x_{\max}) = e^{-d (N+1)^{-1} y^2}$. We thus have
$$\sum_{x = 0}^{y} f(x) \leq \int_0^y f(u) du + e^{-d (N+1)^{-1} y^2} \leq \left( \frac{\sqrt{\pi}}{\sqrt{d + d/N}} + 1\right) \cdot e^{-d (N+1)^{-1} y^2} ,$$
where in the last inequality we used (\ref{GSNconv}). The latter implies that 
$$ \sum_{x = 0}^y p_N(x) p_1( y - x)  \leq W^N D \left( \frac{\sqrt{\pi}}{\sqrt{d + d/N}} + 1\right) \cdot e^{-d (N+1)^{-1} y^2}.$$
Combining all of the above we see that 
$$p_{N+1}( y) \leq W^N e^{-d (N+1)^{-1} y^2} \left[ 1 + DW^{-N} + D\left( \frac{\sqrt{\pi}}{\sqrt{d + d/N}} + 1\right)  \right] \leq W^{N+1}  e^{-d (N+1)^{-1} y^2}.$$
This proves the case $N+1$ and the general result follows by induction.
\end{proof}

We next summarize a couple of parameter choices for future use.

\begin{definition}\label{DefN2Disc} Suppose that $p_X(\cdot)$ satisfies Assumptions D1-D4. Fix $t,s$ such that $\beta > t > s > \alpha$. Then in view of Proposition \ref{S4S1P1} we can find $C_1 > 1$ sufficiently large depending on the constants in that proposition and $M^{(0)}_{s,t}$ in Definition \ref{DefM34Disc} so that
$$C_1^{-N} \leq p_N(z)$$
for all $z \in [s,t] \cap \mathbb{Z}$ and $N \geq N_0$ (where $N_0$ is as in the statement of Proposition \ref{S4S1P1}). 

We can also find $\epsilon_1 > 0$ sufficiently small so that $s \geq \alpha + 3\epsilon_1$ and $t \leq \beta - 3\epsilon_1$.

We can also find $R_1 > 1$ sufficiently large so that 
$$[s,t] \subset [ - R_1, R_1 ] \mbox{ and } W C_1 e^{-d R_1^2/2} \leq 1,$$
where $W = D \frac{\sqrt{\pi}}{\sqrt{d}} + 1 + 2D$ with $D,d$ as in Assumption D4.

Finally, given the above choice of $\epsilon_1$ and $R_1$ we can define the variables $\hat{s}, \hat{t}$ as follows:
\begin{itemize}
\item $\hat{s} = \alpha + \epsilon_1$ and $\hat{t} = \beta - \epsilon_1$ if $\alpha > -\infty$ and $\beta < \infty$;
\item $\hat{s} = \alpha + \epsilon_1$ and $\hat{t} = 3\max(t,0) - \alpha - \epsilon_1$ if $\alpha > -\infty$ and $\beta = \infty$; 
\item $\hat{s} = 3 \min(0, s) - \beta+ \epsilon_1$ and $\hat{t} = \beta - \epsilon_1$ if $\alpha = -\infty$ and $\beta < \infty$; 
\item $\hat{s} = -6R_1$ and $\hat{t} = 6R_1$ if $\alpha = -\infty$ and $\beta =\infty$.
\end{itemize}  
\end{definition}

\begin{definition}\label{DefParDisc}  Suppose that $p_X(\cdot)$ satisfies Assumptions D1-D4. Fix $t,s$ such that $\beta > t > s > \alpha$ and let $C_1, \epsilon_1, R_1, \hat{s}$ and $\hat{t}$ be as in Definition \ref{DefN2Disc}. For future reference we summarize the following list of constants:
\begin{enumerate}
\item the constants in Assumptions D1 and D4;
\item $C_1, \epsilon_1, R_1, \hat{t}, \hat{s}$ as in Definition \ref{DefN2Disc};
\item $M_{\hat{s},\hat{t}}, m_{\hat{s},\hat{t}}$, $\delta_{\hat{s}, \hat{t}}$ as in Definition \ref{DefDeltaDisc};
\item $q_{\hat{s},\hat{t}}$ as in Definition \ref{DefQDisc};
\item $M_{\hat{s}, \hat{t}}^{(0)}, M_{\hat{s}, \hat{t}}^{(1)}, M_{\hat{s}, \hat{t}}^{(2)}, M_{\hat{s}, \hat{t}}^{(3)}, M^{(4)}_{\hat{s}, \hat{t}}$ from Definition \ref{DefM34Disc}.
\end{enumerate}
\end{definition}

We can now prove the following complement to Proposition \ref{S4S1P2}, which establishes tail estimates for the midpoint density of a discrete random walk bridge. 
\begin{proposition}\label{S4S1P3}
Suppose that $p_X$ satisfies Assumptions D1-D4. Fix $s,t$ such that $\beta > t > s > \alpha$. There exist constants $A,a$ and $N_1 \in \mathbb{N}$ such that the following holds. Suppose that $m ,n \geq 1$ are such that $|m-n| \leq 1$ and denote $N = n + m,$. In addition, let $l \in \mathbb{Z}$ be such that $z:= l/N \in [s,t]$. Then for any $k \in \mathbb{Z}$ and $x = k/N$ we have
\begin{equation}\label{S4S1E12}
p_{n,m}( k| l) \leq A \cdot \exp \left(-  aN \left[x-\frac{z}{2}\right]^2  \right).
\end{equation}
The constants $a,A$ and $N_1$ depend on the values $s,t$ and the function $p_X(\cdot)$, where the dependence on the latter is through the constants in Definition \ref{DefParDisc}.
\end{proposition}
\begin{proof}
Denote $\phi = \frac{m}{n}$ and $\psi = \frac{n}{m}$. For clarity we split the proof into several cases.

{\bf \raggedleft Case 1.} Suppose first that $\alpha > -\infty$. From the first line of (\ref{S4S1P2E1}) we know that 
\begin{equation}\label{S4S1P3E1}
p_{n,m}(k|l) =  \frac{p_n(k) \cdot p_m(l - k)}{p_N(l)},
\end{equation}
and the latter expression is zero unless $k \geq n\alpha$ and $l- k \geq m\alpha$. We will assume that $k$ satisfies these inequalities as otherwise (\ref{S4S1E12}) trivially holds for any $A, a > 0$. From Definition \ref{DefN2Disc} we know that for all $N \geq N_0$ we have
\begin{equation}\label{S4P3E2}
p_{n,m}(k|l) \leq  C_1^N p_n(k) \cdot p_m(l) \leq C_1^N.
\end{equation}
The latter implies that (\ref{S4S1E12}) is true for all small $N \geq N_0$ by choosing $A$ sufficiently large and $a \leq 1$. We will thus focus on showing (\ref{S4S1E12}) for sufficiently large $N \geq N_0$. \\

Recall that $F(z) = G_z(u_z) = - \Lambda_X^*(z) $ is defined for $z \in (\alpha, \beta)$ but by Lemma \ref{S2S2L2} we can continuously extend it to $\alpha$ (and to $\beta$ provided $\beta < \infty$) by setting $F(\alpha) = \log p_X(\alpha)$ (and $F(\beta) = \log p_X(\beta)$ if $ \beta < \infty$). We next observe that for any $m, n \geq 1$, $ n\beta \geq k \geq n\alpha$ and $m \beta \geq l - k \geq m\alpha $
\begin{equation}\label{S4P3E3}
\begin{split}
p_n(k) \leq e^{nF(k/n)} \mbox{ and } p_m(l - k)\leq e^{m F((k-l)/m)}.
\end{split}
\end{equation}
Indeed, focusing on the first inequality, the statement is true for $k \neq \alpha n$ and $k \neq \beta n$ from (\ref{S4P1E1}) and the fact that the integrand in that equation is bounded in absolute value by $1$ as shown in Definition \ref{DefQDisc}. The statement is also true for $k = \alpha n$ and $k = \beta n$ by our extension of $F$ above.

Suppose that $Nx \leq  n\alpha + n\epsilon_1,$ where $\epsilon_1$ is as in Definition \ref{DefN2Disc}. From (\ref{S4P3E3}) and Proposition \ref{S4S1P1} we know that there is a $C > 0$, depending on $m_{\hat{s}, \hat{t}}$, such that for $m,n \geq N_0$
\begin{equation}\label{S4P3E4}
\begin{split}
& p_{n,m}(k |l) \leq C \sqrt{N} \cdot \exp\left[N \left( \frac{F[x(1+\phi)] }{1 + \phi}+ \frac{F[ (z-x)(1 + \psi)]}{1 + \psi} -  F(z)\right) \right].
\end{split}
\end{equation}
Similarly to the proof of Proposition \ref{S3S1P2} we write $x = \frac{z}{1+ \phi} + r_x$ and denote 
$$h(r) = \frac{F[z + (1+\phi)r] }{1 + \phi}+ \frac{F[ z- r(1 + \psi)]}{1 + \psi} -  F(z).$$
Notice that since $k \leq n \alpha + n \epsilon_1$ we have that $r_x \geq \frac{2 \epsilon_1}{1 + \phi} \geq \frac{2\epsilon_1}{3}$. 
In addition, we have
$$h''(r) = (1+\phi) F''(z + (1 + \phi)r) + (1+\psi)F''(z +(1+\psi)r) \leq 0$$
for all $r \in [0, r_x]$ and so by the continuity of $F$ and its smoothness on $(\alpha, \beta)$ we conclude
$$\frac{F[x(1+\phi)] }{1 + \phi}+ \frac{F[ (z-x)(1 + \psi)]}{1 + \psi} -  F(z) = \int_0^{r_x} \int_0^{y} h''(r) dr dy \leq \int_0^{\epsilon_1/3}\int_0^y h''(r) dr dy $$
$$\leq\int_0^{\epsilon_1/3}\int_0^y \left[-\frac{2}{M_{\hat{s}, \hat{t}}}\right]dr dy = -\frac{\epsilon_1^2}{9 M_{\hat{s},\hat{t}}}.$$
Applying the above in (\ref{S4P3E4}) we conclude
\begin{equation}\label{S4P3E5}
\begin{split}
&p_{n,m}(k|l)  \leq C \sqrt{N} \cdot \exp\left(-\frac{\epsilon_1^2N}{9 M_{\hat{s},\hat{t}}} \right).
\end{split}
\end{equation}
On the other hand, for $N_1$ sufficiently large  depending on $\alpha$ and $N \geq N_1$ we have
$$A \cdot \exp \left(-  aN \left[x-\frac{z}{2}\right]^2  \right) \geq  A \cdot  \exp \left( - a N \frac{\epsilon_1^2}{4} \right).$$
It follows from the above inequalities that (\ref{S4S1E12}) holds provided we take $A = 1$, $a$ sufficiently small, $N_1$ sufficiently large and $Nx \in [n\alpha, n\alpha + n\epsilon_1]$ for $m,n \geq N_1$. Analogous arguments applied to $z-x$ in place of $x$ show that for the same $A$ and $a$ we have (\ref{S4S1E12}) provided that $N(z-x) \in [m\alpha, m\alpha + m\epsilon_1]$. We may thus assume that $Nx\geq n\alpha + n\epsilon_1$ and $N(z-x) \geq m\alpha + m\epsilon_1$.\\

We next consider the cases $\beta = \infty$ and $\beta < \infty$ separately starting with the former. 

{\bf \raggedleft Case 1.A.} If $\beta = \infty$ then we let $N_1$ be sufficiently large so that 
 $N_1 \geq N_0$, where $N_0$ is as in the statement of Proposition \ref{S4S1P1} for the values $\hat{s} = \alpha + \epsilon_1$ and $\hat{t} = 3\max(t,0) - \alpha - \epsilon_1$.

Then from Proposition \ref{S4S1P1} (see also equation (\ref{S4S1P2E1})) we know that we have for $m,n \geq N_1$ and $Nx\geq n\alpha + n\epsilon_1$ and $N(z-x) \geq m\alpha + m\epsilon_1$ that
\begin{equation}\label{S4P3E6}
\begin{split}
&p_{n,m}(k|l) \leq C_2 \exp\left[N \left( \frac{F(x(1+\phi)) }{1 + \phi}+ \frac{F ((z-x)(1 + \psi))}{1 + \psi} -  F(z)\right) \right],
\end{split}
\end{equation}
where the constant $C_2$ depends on $m_{\hat{s}, \hat{t}}$ and $M_{\hat{s}, \hat{t}}$ as in Definition \ref{DefDeltaDisc} for the values $\hat{s} = \alpha + \epsilon_1$ and $\hat{t} = 3\max(t,0) - \alpha - \epsilon_1$. From here the proof continues as that of Case 1.A. in Proposition \ref{S3S1P3}.

{\bf \raggedleft Case 1.B.} We suppose that $\beta < \infty$. As before we know that (\ref{S4S1E12}) holds for any $A, a > 0$ if $Nx > n\beta$ or $N(z-x) > m \beta$ and so we may assume that $Nx \leq n \beta$ and $N(z-x) \leq m\beta$.

 Suppose $Nx \geq n\beta - n\epsilon_1$. We can repeat  our arguments from before and see that (\ref{S4P3E5}) holds in this case as well.
On the other hand, for $N \geq N_1$ with $N_1$ sufficiently large depending on $\beta$ we have
$$A \cdot \exp \left(-  aN \left[x-\frac{z}{2}\right]^2  \right) \geq  A \exp \left( - a N \frac{\epsilon_1^2}{4} \right).$$
It follows from the above inequalities that (\ref{S4S1E12}) holds provided we take $A = 1$, $a$ sufficiently small, $N_1$ sufficiently large and $Nx \in [m\beta, m\beta - m\epsilon_1]$ for $m,n \geq N_1$. Analogous arguments applied to $z-x$ in place of $x$ show that for the same $A$ and $a$ we have (\ref{S4S1E12}) provided that $N(z-x) \in [ m\beta - m\epsilon_1, m\beta]$. We may thus assume that $Nx \in [n\alpha + n \epsilon_1, n\beta - n \epsilon_1]$ and $N(z-x) \in [ m\alpha + m \epsilon_1, m\beta - m\epsilon_1].$

We let $N_1$ be sufficiently large so that $N_1 \geq N_0$, where $N_0$ is as in the statement of Proposition \ref{S4S1P1} for the values $\hat{s} = \alpha + \epsilon_1$ and $\hat{t} = \beta - \epsilon_1$.

 Then from Proposition \ref{S4S1P1} (see also equation (\ref{S4S1P2E1}) ) we know that for $m,n \geq N_1$ and  $Nx \in [n\alpha + n \epsilon_1, n\beta - n \epsilon_1]$ and $N(z-x) \in [ m\alpha + m \epsilon_1, m\beta - m\epsilon_1]$ that
\begin{equation}\label{S4S1P3E4}
\begin{split}
&p_{n,m}(k|l) \leq  C_2 \exp\left[N \left( \frac{F(x(1+\phi)) }{1 + \phi}+ \frac{F ((z-x)(1 + \psi))}{1 + \psi} -  F(z)\right) \right],
\end{split}
\end{equation}
where the constant $C_2$ depends on $m_{\hat{s}, \hat{t}}$ and $M_{\hat{s}, \hat{t}}$ as in Definition \ref{DefDeltaDisc} for the values $\hat{s} = \alpha + \epsilon_1$ and $\hat{t} = \beta - \epsilon_1$. From here the proof continues as that of Case 1.B. in Proposition \ref{S3S1P3}. Overall, we conclude (\ref{S4S1E12}) under the condition that $\alpha > -\infty$. 

{\bf \raggedleft Case 2.} Suppose now that $\alpha = -\infty$. 

{\bf \raggedleft Case 2.A.} If $\beta < \infty$ then we can conclude (\ref{S4S1E12}) by the same arguments as those in Case 1.A.

{\bf \raggedleft Case 2.B.} Suppose that $\beta = \infty$. By symmetry it suffices to consider the case when Assumption D4.1 holds. Let $R_1$ be as in Definition \ref{DefN2Disc}. Then from Lemma \ref{S4S1L1} and (\ref{S4P3E2}) we know that for $x \geq R_1$  and $N \geq N_0$ 
$$p_{n,m}(k|l) \leq  C_1^N \cdot W^n e^{-d  N x^2} \leq  e^{-dNx^2/2},$$
while 
$$A \cdot \exp \left(-  aN \left[x-\frac{z}{2}\right]^2  \right) \geq  A \exp \left( - a N [ x + R_1/2 ]^2 \right).$$
It follows from the above inequalities that (\ref{S4S1E12}) holds provided we take $A = 1$, $a$ sufficiently small (say $a \leq d/8$) and $x \geq R_1$. Analogous arguments applied to $z-x$ in place of $x$ show that for the same $A$ and $a$ we have (\ref{S4S1E12}) provided that $z-x \geq R_1$. We may thus assume that $x, z-x \in [ -2R_1, 2R_1]$. 

We let $N_1$ be sufficiently large so that  $N_1 \geq N_0$, where $N_0$ is as in the statement of Proposition \ref{S4S1P1} for the values $\hat{s} = -6R_1$ and $\hat{t} = 6R_1$.
Then from Proposition \ref{S4S1P1} (see also equation (\ref{S4S1P2E1}))  we know that for $m,n \geq N_1$ and $x \in[-2R_1, 2R_1]$
\begin{equation*}
\begin{split}
&p_{n,m}(k|l) \leq  C_2 \exp\left[N \left( \frac{F(x(1+\phi)) }{1 + \phi}+ \frac{F ((z-x)(1 + \psi))}{1 + \psi} -  F(z)\right) \right],
\end{split}
\end{equation*}
where the constant $C_2$ depends on $m_{\hat{s}, \hat{t}}$ and $M_{\hat{s}, \hat{t}}$ as in Definition \ref{DefDeltaDisc} for the values $\hat{s} = -6R_1$ and $\hat{t} = 6R_1$. From here the proof proceeds as that of Case 2.B. in Proposition \ref{S3S1P3}.

\end{proof}

\section{Gaussian coupling}\label{Section5}

In this section we isolate some results about the quantile coupling of random variables with certain estimates on their probabilities to Gaussian random variables. We start by isolating some results about Gaussian random variables. We denote by $\Phi(x)$ and $\phi(x)$ the cumulative distribution function and density of a standard normal random variable. The following two lemmas can be found in \cite[Section 4.2]{MZ}. 
\begin{lemma}\label{LemmaI1}
There is a constant $c > 1$ such that for all $x \geq 0$ we have
\begin{equation}\label{LI2}
 \frac{1}{c(1+x)} \leq \frac{1 - \Phi(x)}{\phi(x)} \leq \frac{c}{1 +x},
\end{equation}
\end{lemma}
\begin{lemma}\label{LemmaI3}
For all $A > 0$, $n \geq 64 A^2$ and $0 \leq x \leq \frac{1}{8A}$ we have 
\begin{equation}\label{LI5}
\log \left( \frac{\Phi(-\sqrt{n} x + u)}{\Phi( - \sqrt{n}x)}\right) = \log\left(\frac{1 - \Phi(\sqrt{n} x - u)}{1 - \Phi(\sqrt{n}x)} \right) \geq A(nx^3 + n^{-1/2})
\end{equation}
and
\begin{equation}\label{LI6}
\log \left( \frac{\Phi(-\sqrt{n} x - u)}{\Phi( - \sqrt{n}x)}\right) = \log\left(\frac{1 - \Phi(\sqrt{n} x + u)}{1 - \Phi(\sqrt{n}x)} \right) \leq -A(nx^3 + n^{-1/2}),
\end{equation}
where $u = 2A (\sqrt{n}x^2 + n^{-1/2})$.
\end{lemma}

From Rolle's theorem one deduces the following simple result.
\begin{lemma}\label{LemmaI2}
Let $R > 0$ be given. There exists a positive constant $c_1$ such that for $x, y \in [-R, R]$
\begin{equation}\label{LI4}
\left|\Phi(x) - \Phi(y)\right| \leq c_1|x-y|.
\end{equation}
\end{lemma}

The following is an analogue of \cite[Lemma 6.9]{LF}. We include it here for the sake of completeness.
\begin{lemma}\label{L2}
Let $M_0 > 0$, $\epsilon_0 > 0$, $\tilde{c}  \in (0,1)$, $b' > 0$ and $c' > 0$ be given. Then we can find constants $c_2, \epsilon_2 > 0$, $N_2 \in \mathbb{N}$ such that the following holds for every positive integer $n \geq N_2$ and every $\sigma^2 \in[ \tilde{c}, \tilde{c}^{-1}]$. Suppose that $X$ is an integer random variable and for all $x \in \{y: y \in \mathbb{Z}, |y| \leq n \epsilon_0\}$
\begin{equation}\label{S2E1}
\mathbb{P}(X = x) = \frac{1}{\sqrt{2\pi \sigma^2 n}} \exp \left( - \frac{x^2}{2n \sigma^2} + \delta(x) \right),
\end{equation}
where 
\begin{equation}\label{S2E2}
|\delta(x)| \leq  M_0 \left[ \frac{1}{\sqrt{n}} + \frac{|x|^3}{n^2} \right].
\end{equation}
Assume additionally that for any $m \in \mathbb{Z}$ 
\begin{equation}\label{S2E2v2}
\mathbb{P}(X = m) \leq c' e^{-b' m^2 / n}.
\end{equation}
Then for any $|x| \leq \epsilon_2 n$ we have
\begin{equation}\label{S2E3}
F \left(x - c_2 \left(1 + \frac{x^2}{n} \right) \right) \leq \mathbb{P}(X \leq x -1) \leq \mathbb{P}(X \leq x+ 1) \leq F \left(x + c_2 \left(1 + \frac{x^2}{n} \right) \right),
\end{equation} 
where $F(x)$ is the cumulative distribution function of a $N(0, \sigma^2 n)$ random variable.
\end{lemma}
\begin{proof}
For convenience we denote $G(x) = \mathbb{P}(X \leq x)$, $\bar{F}(x) = 1 - F(x)$, $f(x) = F'(x) = \frac{e^{-x^2/(2\sigma^2 n)}}{\sqrt{2\pi n \sigma^2}}$ and $\bar{G}(x) = 1 - G(x)$. Throughout $C,c$ will stand for generic constants that depend on $M_0, \tilde{c}, \epsilon_0, b', c'$ unless otherwise specified.

 By symmetry we can assume $x \geq 0$. It suffices to prove (\ref{S2E3}) only for integer values of $x$ and for $n$ sufficiently large. In particular, we assume that $N_2$ is sufficiently large so that $ \epsilon_0 n \geq n^{5/8} \geq  \sqrt{3 \tilde{c} } \cdot n^{1/2} \geq 1$ for all $n \geq N_2$. We prove (\ref{S2E3}) in three cases depending on the size of $|x|$. \\

We first consider the case $x \leq \sqrt{3 \tilde{c} } \cdot n^{1/2}$. We then have
\begin{equation}\label{L2E1}
 \bar{F}(x) =  \sum_{j > x} f(j) + \sum_{j > x} \left[ \mathbb{P}(X = j) - f(j) \right] = \int_x^\infty f(x) dx +\sum_{j > x} \left[ \mathbb{P}(X = j) - f(j) \right] + O\left( \frac{1}{\sqrt{n}} \right),
\end{equation}
where in the last equality we used that $f(x)$ is decreasing for $x \geq 0$ and its integral over any unit interval is at most $\frac{1}{\sqrt{2\pi n \sigma^2}}$. Using that $f(x)$ is decreasing for all $x \geq 0$ we get
\begin{equation}\label{L2E2}
\sum_{j > x} \left| \mathbb{P}(X = j) - f(j) \right|  \leq \sum_{j = x+1}^{\lfloor n^{2/3} \rfloor} \frac{e^{-\frac{j^2}{2 n \sigma^2}}}{\sqrt{2\pi \sigma^2 n}} |e^{\delta(j)} - 1| + \mathbb{P}(X \geq n^{2/3}) + \int_{n^{2/3} - 1}^\infty f(x) dx
\end{equation}
We next increase $N_2$ so that $N_2^{-1/3}M_0 \leq \frac{\tilde{c}}{4} \leq  \frac{1}{4\sigma^2}$ and use the inequality $|e^x - 1| \leq |x| e^{|x|}$ to estimate
\begin{equation}\label{L2E3}
\sum_{j = x+1}^{\lfloor n^{2/3} \rfloor} \frac{e^{-\frac{j^2}{2 n \sigma^2}}}{\sqrt{2\pi \sigma^2 n}} |e^{\delta(j)} - 1| \leq \hspace{-1.5mm} \sum_{j = x+1}^{\lfloor n^{2/3} \rfloor} \frac{e^{-\frac{j^2}{2 n \sigma^2}}}{\sqrt{2\pi \sigma^2 n}}|\delta(j)|e^{|\delta(j)|} \leq \hspace{-1.5mm}  \sum_{j = x+1}^{\lfloor n^{2/3} \rfloor} \frac{e^{-\frac{j^2}{4 n \sigma^2} + \frac{M_0}{\sqrt{n}}}}{\sqrt{2\pi \sigma^2 n}}\left[ \frac{M_0}{\sqrt{n}} + \frac{M_0 |j|^3}{n^2} \right].
\end{equation}
Since $f(x)$ is decreasing for all $x \geq 0$
\begin{equation}\label{L2E4}
\begin{split}
 &\sum_{j = x+1}^{\lfloor n^{2/3} \rfloor} \frac{e^{-\frac{j^2}{4 n \sigma^2} }}{\sqrt{2\pi \sigma^2 n}}\left[\frac{M_0}{\sqrt{n}} \right] \leq \frac{\sqrt{2}M_0}{\sqrt{n}} \cdot \int_{0}^\infty \frac{e^{-\frac{u^2}{4 n \sigma^2} }}{\sqrt{4\pi \sigma^2 n}}du =  \frac{M_0}{\sqrt{2n}}.
\end{split}
\end{equation}
Analogously, by using that $x^3e^{-x^2/2}$ is decreasing for all $x \geq \sqrt{3}$ we have
\begin{equation}\label{L2E5}
\begin{split}
&\sum_{j = x+1}^{\lfloor n^{2/3} \rfloor} \frac{e^{-\frac{j^2}{4 n \sigma^2} }}{\sqrt{2\pi \sigma^2 n}}\left[\frac{M_0|j|^3}{n^2} \right] = \sum_{j = x+1}^{\lfloor \sqrt{3 \tilde{c}n} \epsilon  + 1\rfloor} \frac{e^{-\frac{j^2}{4 n \sigma^2} }}{\sqrt{2\pi \sigma^2 n}}\left[\frac{M_0|j|^3}{n^2} \right] + \sum_{j = \lfloor \sqrt{3 \tilde{c}n} \rfloor +2}^{\lfloor n^{2/3} \rfloor} \frac{e^{-\frac{j^2}{4 n \sigma^2} }}{\sqrt{2\pi \sigma^2 n}}\left[\frac{M_0|j|^3}{n^2} \right] \\
&\leq \frac{M_0 [2 \sqrt{3 \tilde{c}} \epsilon]^3}{\sqrt{2n}} + \frac{M_0}{n^2} \int_{0}^\infty  \frac{u^3e^{-\frac{u^2}{4 n \sigma^2} }}{\sqrt{2\pi \sigma^2 n}}du =  \frac{M_0 [2 \sqrt{3 \tilde{c}} \epsilon]^3}{\sqrt{2n}} + \frac{2(\sqrt{2\sigma^2 n})^3M_0}{\sqrt{\pi} n^2}.
\end{split}
\end{equation}

Finally, we have that by taking $N_2$ larger we can ensure using (\ref{LI2})  and (\ref{S2E2v2}) that 
\begin{equation}\label{L2E6}
\mathbb{P}(X \geq n^{2/3}) \leq \frac{c' e^{-b' \lfloor n^{2/3} \rfloor^2/n}}{1 - e^{-b' \lfloor n^{2/3} \rfloor / n}}  \leq Ce^{-cn^{1/3}} \mbox{, } \int_{n^{2/3} - 1}^\infty \hspace{-4mm}f(x) dx = 1 - \Phi\left( \frac{ n^{2/3} - 1}{\sigma^2 \sqrt{n}}  \right) \leq C e^{-c n^{1/3}}.
\end{equation}
Combining (\ref{L2E1}), (\ref{L2E2}), (\ref{L2E3}), (\ref{L2E4}), (\ref{L2E5}) and (\ref{L2E6}) we conclude for $|x| \leq \sqrt{3 \tilde{c}} n^{1/2}$ and $n$ large
$$\left| G(x) - F(x)\right| = \left| \bar{G}(x) - \bar{F}(x) \right| \leq \frac{C}{\sqrt{n}} ,$$
which implies (\ref{S2E3}) in view of Lemma \ref{LemmaI2}.\\

Next we consider the case $n^{5/8} \geq x \geq  \sqrt{3 \tilde{c} } \cdot n^{1/2}$. In this case we have
\begin{equation}\label{L2E1v2}
 \bar{G}(x) =  \sum_{j > x} f(j) + \sum_{j > x} \left[ \mathbb{P}(X = j) - f(j) \right] = \bar{F}(x) +\sum_{j > x} \left[ \mathbb{P}(X = j) - f(j) \right] +  O\left( \frac{e^{-\frac{x^2}{2 n \sigma^2} }}{\sqrt{2\pi \sigma^2 n}} \right),
\end{equation}
where in the last equality we used that $f(y)$ is decreasing on $[ x, \infty)$ and its integral over any unit interval is at most $\frac{e^{-\frac{x^2}{2 n \sigma^2} }}{\sqrt{2\pi \sigma^2 n}}$. Notice that for $x+1 \leq j \leq n^{2/3}$ we have $|\delta(j)| \leq C |j|^3/ n^2 \leq C$, where $C = M_0 \cdot [1 + (3 \tilde{c})^{-3/2}]$. This means that $\left|e^{\delta(j)} - 1\right| \leq e^C |\delta(j)| \leq C|j|^3/n^2$ and so
\begin{equation}\label{L2E3v2}
\sum_{j = x+1}^{\lfloor n^{2/3} \rfloor} \frac{e^{-\frac{j^2}{2 n \sigma^2}}}{\sqrt{2\pi \sigma^2 n}} |e^{\delta(j)} - 1| \leq C \sum_{j = x+1}^{\lfloor n^{2/3} \rfloor} \frac{e^{-\frac{j^2}{2 n \sigma^2} }}{\sqrt{2\pi \sigma^2 n}}\left[  \frac{j^3}{n^2} \right] \leq \frac{C}{n^2}\int_x^\infty   \frac{u^3e^{-\frac{u^2}{2 n \sigma^2} }}{\sqrt{2\pi \sigma^2 n}}du \leq \frac{Cx^2}{n^{3/2}} e^{-\frac{x^2}{2n \sigma^2}}.
\end{equation}

From (\ref{L2E6}) we know that by possibly making $N_2$ larger we can ensure
\begin{equation}\label{L2E6v2}
\mathbb{P}(X \geq n^{2/3}) \leq Ce^{-cn^{1/3}} \leq \frac{1}{\sqrt{n}} \cdot e^{-\frac{x^2}{2n \sigma^2}} \mbox{ and } \int_{n^{2/3} - 1}^\infty f(x) dx \leq C e^{-c n^{1/3}} \leq \frac{1}{\sqrt{n}} \cdot e^{-\frac{x^2}{2n \sigma^2}}.
\end{equation}
Combining (\ref{L2E1v2}), (\ref{L2E2}), (\ref{L2E3v2}) and (\ref{L2E6v2}) we conclude for  $n^{5/8} \geq x \geq  \sqrt{3 \tilde{c} } \cdot n^{1/2}$ and all large $n$
$$\left| G(x) - F(x)\right| = \left| \bar{G}(x) - \bar{F}(x) \right| \leq C\left[ 1 + \frac{x^2}{n^{3/2}} \right]\cdot  \frac{e^{-\frac{x^2}{2 n \sigma^2} }}{\sqrt{2\pi \sigma^2 n}} \leq C \cdot \frac{x^3}{n^2} \cdot \bar{F}(x),$$
where in the last inequality we used (\ref{LI2}). The above inequality implies that for all large $n$
\begin{equation}\label{L2E7v2}
\begin{split}
&\bar{G}(x) \leq \left [1 + C\frac{x^3}{n^2}  \right] \bar{F}(x) \leq e^{Cx^3/n^2}  \bar{F}(x) \leq \bar{F} \left( x - C \left[1 + \frac{x^2}{n} \right] \right)\\
&\bar{G}(x) \geq \left [1 - C\frac{x^3}{n^2}  \right] \bar{F}(x) \geq e^{-Cx^3/n^2}  \bar{F}(x) \geq \bar{F} \left( x + C \left[1 + \frac{x^2}{n} \right] \right),
\end{split}
\end{equation}
where the right most inequalities used Lemma \ref{LemmaI3}. From (\ref{L2E7v2}) we conclude (\ref{S2E3}) for some large $c_0$ and all $n^{5/8} \geq x \geq  \sqrt{3 \tilde{c} } \cdot n^{1/2}$ provided $n$ is large enough.

We finally consider the case $n \epsilon_2 \geq x \geq n^{5/8}$, where $\epsilon_2$ is to be chosen sufficiently small as follows. Consider the functions $h_{\pm}(z) = - \frac{z^2}{2\sigma^2} \pm 2M_0\frac{z^3}{\sqrt{n}}$. Then 
$$h_{\pm}'(z) = - \frac{z}{\sigma^2} \pm 6M_0 \frac{z^2}{\sqrt{n}} \leq - \tilde{c} z \pm 6M_0 \frac{z^2}{\sqrt{n}} \leq z \left[ \pm 6M_0 \frac{z}{\sqrt{n}} - \tilde{c} \right]$$
 and we can choose $\epsilon_1 \leq \min (\epsilon_0,1)$ sufficiently small (depending on $M_0$  and $\tilde{c}$) such that the functions $h_{\pm}(z)$ are decreasing and moreover $-\frac{3z^2}{2  \sigma^2} \leq h_-(z) \leq h_+(z) \leq -  \frac{z^2}{4 \sigma^2}$  for $0 < z \leq \epsilon_1 \sqrt{n}$. We next pick $\epsilon_2> 0$ (depending on $\tilde{c}$, $M_0$, $b'$ and $c'$) so that $\epsilon_2 \leq \epsilon_1/2$ and for all $n \geq \epsilon_2^{-6}$
\begin{equation}\label{L2E6v3}
\begin{split}
\mathbb{P}(X \geq n \epsilon_1) \leq \frac{c' e^{-b' \lfloor n \epsilon_1 \rfloor^2/n}}{1 - e^{-b' \lfloor n  \epsilon_1 \rfloor / n}}  \leq  \frac{e^{h_+(\sqrt{n}\epsilon_2)}}{\sqrt{2\pi \sigma^2 n}} \leq  \frac{e^{h_+(x/\sqrt{n})}}{\sqrt{2\pi \sigma^2 n}}.
\end{split}
\end{equation}
Using the inequality $e^{\delta(j)} \leq \exp \left(2M_0 \frac{j^3}{n^2}\right)$ for $x+ 1 \leq j \leq \epsilon_1 n $ and the fact that $h_+(z)$ is decreasing on $0 < z \leq \epsilon_1 \sqrt{n}$ by our choice of $\epsilon_1$ we see that
\begin{equation}\label{L2E1v3}
\begin{split}
 \bar{G}(x) =  \sum_{j = x+1}^{\lfloor n \epsilon_1 \rfloor} f(j) e^{\delta(j)}  + \mathbb{P}(X \geq n \epsilon_1) \leq \int_{x/ \sqrt{n}}^{\sqrt{n}\epsilon_1} \frac{e^{h_+(u)}du}{\sqrt{2\pi \sigma^2}} + \mathbb{P}( X \geq n \epsilon_1) \leq \int_{x-1/ \sqrt{n}}^{\sqrt{n}\epsilon_1} \frac{e^{h_+(u)}du}{\sqrt{2\pi \sigma^2}},
\end{split}
\end{equation}
where in the last inequality we used (\ref{L2E6v3}). Using that $2x \leq 2\epsilon_2 n \leq \epsilon_1 n$ we have 
$$ \int_{x-1/\sqrt{n}}^{\sqrt{n} \epsilon_1} \frac{e^{h_+(u)}du}{\sqrt{2\pi \sigma^2}}  = \int_{(x - 2)/\sqrt{n}}^{2x/\sqrt{n}}  \frac{e^{h_+(u)}du}{\sqrt{2\pi \sigma^2}} + \int_{2x/ \sqrt{n}}^{\sqrt{n}\epsilon_1}   \frac{e^{h_+(u)}du}{\sqrt{2\pi \sigma^2}}   - \int_{(x - 2)/\sqrt{n}}^{x-1/\sqrt{n}}  \frac{e^{h_+(u)}du}{\sqrt{2\pi \sigma^2}} \leq $$
$$\int_{(x - 2)/\sqrt{n}}^{2x/\sqrt{n}}  \frac{e^{h_+(u)}du}{\sqrt{2\pi \sigma^2}} + \int_{2x/ \sqrt{n}}^{\sqrt{n}\epsilon_1}   \frac{e^{-u^2/4\sigma^2 }du}{\sqrt{2\pi \sigma^2}} - \int_{(x - 2)/\sqrt{n}}^{x-1/\sqrt{n}}  \frac{e^{ - u^2/ 2\sigma^2}du}{\sqrt{2\pi \sigma^2}} \leq \int_{(x - 2)/\sqrt{n}}^{2x/\sqrt{n}}  \frac{e^{h_+(u)}du}{\sqrt{2\pi \sigma^2}},$$
where the last inequality holds provided $n$ is sufficiently large in view of (\ref{LI2}). Combining the above with Lemma \ref{LemmaI3} we see that by possibly making $\epsilon_2$ smaller and $N_2$ larger we can ensure that
 \begin{equation}\label{L2E2v3}
 \bar{G}(x)  \leq \int_{(x - 2)/\sqrt{n}}^{2x/\sqrt{n}}  \frac{e^{h_+(u)}du}{\sqrt{2\pi \sigma^2}} \leq e^{16 M_0 x^3/n^2} \cdot \bar{F}(x - 2) \leq  \bar{F} \left( x - C \left[1 + \frac{x^2}{n} \right] \right).
\end{equation}

To get the lower bound notice that $2x \leq 2\epsilon_2 n \leq \epsilon_1 n$ and so
\begin{equation}\label{L2E3v3}
\begin{split}
 \bar{G}(x) \geq  \sum_{j = x+1}^{\lfloor n \epsilon_1 \rfloor} f(j) e^{\delta(j)} \geq \int_{\frac{x+1}{\sqrt{n}}}^{\frac{2x}{ \sqrt{n}}} \frac{e^{h_-(u)}du}{\sqrt{2\pi \sigma^2}} \geq  e^{-16 M_0 x^3/n^2} \cdot \bar{F}(x +1)  \geq  \bar{F} \left( x + C \left[1 + \frac{x^2}{n} \right] \right).
\end{split}
\end{equation}
  From (\ref{L2E2v3}) and (\ref{L2E3v3}) we conclude (\ref{S2E3}) for some large $c_2$ and all $\epsilon_2 n \geq x \geq n^{5/8}$ provided $n \geq N_2$ with $N_2$ large enough and $\epsilon_2$ small enough. This suffices for the proof.
\end{proof}

As an immediate corollary to the above lemma we have the following statement.
\begin{corollary}\label{C2}
Let $M_0 > 0$, $\epsilon_0 > 0$, $\tilde{c}  \in (0,1)$, $b' > 0$ and $c' > 0$ be given. Then we can find constants $c_2, \epsilon_2 > 0$, $N_2 \in \mathbb{N}$ such that the following holds for every positive integer $n \geq N_2$ and every $\sigma^2 \in[ \tilde{c}, \tilde{c}^{-1}]$. Suppose that $X$ is a continuous random variable with density $g$ and for all $x \in \{y: y \in \mathbb{R}, |y| \leq n \epsilon_0\}$
\begin{equation}\label{S3E1}
g(x) = \frac{1}{\sqrt{2\pi \sigma^2 n}} \exp \left( - \frac{x^2}{2n \sigma^2} + \delta(x) \right),
\end{equation}
where 
\begin{equation}\label{S3E2}
|\delta(x)| \leq  M_0 \left[ \frac{1}{\sqrt{n}} + \frac{|x|^3}{n^2} \right].
\end{equation}
Assume additionally that for any $x \in \mathbb{R}$ 
\begin{equation}\label{S3E2v2}
g(x) \leq c' e^{-b' x^2/ n}.
\end{equation}
Then for any $|x| \leq \epsilon_2 n$ we have
\begin{equation}\label{S3E3}
F \left(x - c_2 \left(1 + \frac{x^2}{n} \right) \right) \leq \mathbb{P}(X \leq x) \leq F \left(x + c_2 \left(1 + \frac{x^2}{n} \right) \right),
\end{equation} 
where $F(x)$ is the cumulative distribution function of a $N(0, \sigma^2 n)$ random variable.
\end{corollary}
\begin{proof} By our assumptions we know that $W = \lfloor X \rfloor$ is an integer valued random variable that satisfies the conditions of Lemma \ref{L2}. The result now follows from (\ref{S2E3}) and the fact that $\mathbb{P}(W \leq x - 1) \leq \mathbb{P}(X \leq x) \leq \mathbb{P}(W \leq x+1).$
\end{proof}
\section{Strong coupling} \label{Section6}
 We formulate quantified refinements of Theorems \ref{S2ContKMT} and \ref{S2DiscKMT} as Theorems \ref{ContKMTA} and \ref{KMTA}, respectively, below and present their proof. As usual we split our discussion depending on whether our random walk bridge has continuous or discrete jumps.

%-------------------------------------------------------------------------------------------------------------------------------------------------------------------------------------------------
% Section 6.1
%
%-------------------------------------------------------------------------------------------------------------------------------------------------------------------------------------------------
\subsection{Continuous case}\label{Section6.1} 

We use the same notation as in Sections \ref{Section2.1} and \ref{Section3}.  

\begin{lemma}\label{S6S1L1}
Suppose that $f_X$ satisfies Assumptions C1-C5 and fix $p \in (\alpha ,\beta)$. Let $s = p - \epsilon'$ and $t = p + \epsilon'$, where $\epsilon' > 0$ is sufficiently small so that $\alpha < s < t < \beta$. Then there exists $ \epsilon_3 \in (0, \epsilon')$ and $N_3 \in \mathbb{N}$  such that for every $b_1 > 0$ there exist constants $0 < c_1, a_1 < \infty$ such that the following holds. Suppose that $m, n$ are integers such that $m,n \geq N_3$ with $|m-n|\leq 1$, set $N = m+ n$. We can define a probability space on which are defined a standard normal random variable $\xi$ and a collection of random variables $W = W^{(m,n,z)}$ for all $z \in \{ x \in L_N: |x - pN| \leq \epsilon_3 N \}$ such that the law of $W^{(m,n,z)}$  is the same as that of $S_n^{(N,z)}$ and such that we have almost surely
\begin{equation}\label{S6S1E1}
\mathbb{E} \left[e^{a_1 |Z - W|} \Big{|} W \right] \leq c_1 \cdot  \exp \left(b_1 \frac{(W - pn)^2 + (z - pN)^2}{N}\right),
\end{equation}
where 
$$Z = Z^{(m,n,z)} = \frac{z}{2} + \frac{\sqrt{N} \sigma_p}{2} \cdot \xi, \mbox { so that } Z \sim N\left( \frac{z}{2}, \frac{\sigma^2_p N}{4}  \right).$$
The constants $\epsilon_3$ and $N_3$ depend on the values $p$, $s,t$ and the function $f_X(\cdot)$, where the dependence on the latter is through the constants in Definition \ref{DefParC}.
\end{lemma}
\begin{proof}
Notice that we only need to prove the lemma for $N$ sufficiently large. In order to simplify the notation we will assume that $n =m = N/2$ (the other cases can be handled similarly). \\

 We apply Propositions \ref{S3S1P2} and \ref{S3S1P3} for the variables $s$ and $t$. This implies that provided $N_3 \geq \max (N_0, N_1)$ as in the statements of those propositions and $n \geq N_3$ we have that the random variable $ S^{(N,z)}_n - z/2 $ satisfies the conditions of Corollary \ref{C2} for $M_0 = M$ as in Proposition \ref{S3S1P2}, $\epsilon_0 = \epsilon'$ as in the statement of this proposition, $\tilde{c} = (1/2) \cdot \min ( m_{\hat{s},\hat{t}}, M_{\hat{s},\hat{t}}^{-1})$ as in Definition \ref{DefDelta} for the variables $\hat{s}, \hat{t}$ as in Definition \ref{DefN2}, $b' = a$ and $c' = A$ as in the statement of Proposition \ref{S3S1P3}. We consequently, let $c_2, N_2, \epsilon_2$ be as in the statement of that corollary for the above constants.

In what follows we fix $\epsilon_3 \leq 4^{-1}\min (\epsilon_2, \epsilon' )$ sufficiently small so that $\epsilon_3 M \leq 1/M_{\hat{s},\hat{t}}$ where $M$ is as in the statement of Proposition \ref{S3S1P2} and $M_{\hat{s},\hat{t}}$ is as in Definition \ref{DefDelta}  for the variables $\hat{s}, \hat{t}$ as in Definition \ref{DefN2}. Observe that the choice of $\epsilon_3$ implies that $\epsilon_3 M \leq 1/ \sigma_{z/N}^2$ for all $|z -pN| \leq N \epsilon_3$. We also set $N_3 = \max (N_0, N_1, N_2)$.

We denote by $\Phi$ the cumulative distribution function of a normal random variable with mean $0$ and variance $1$. Let $G_{n,m, z}$ denote the cumulative distribution function of $S_n^{(N,z)}$. In addition, let $G^{\epsilon_3, +}_{n,m, z}$ and $G^{\epsilon_3, -}_{n,m, z}$ denote the cumulative distribution function of $S_n^{(N,z)}$ conditioned on $\{ S_n^{(N,z)} > z/2 + 2\epsilon_3 n\}$ and $\{S_n^{(N,z)} < z/2 - 2\epsilon_3 n\}$ respectively. For convenience we let $A < B$ be the unique real numbers such that  
$$1 - \Phi(B) = \mathbb{P}( S_n^{(N,z)} > z/2 + 2\epsilon_3 n), \hspace{3mm} \Phi(A) = \mathbb{P}( S_n^{(N,z)} < z/2 + 2\epsilon_3 n).$$

We now turn to defining our probability space. We let $U_1, U_2, U_3$ be three independent uniform $(0,1)$ random variables and set $\xi = \Phi^{-1}(U_1)$. In addition, we set $W_+ = \left(G^{\epsilon_3, +}_{n,m, z} \right)^{-1}(U_2)$ and $W_- = \left(G^{\epsilon_3, -}_{n,m, z}\right)^{-1} (U_3)$. Given a realization of $\xi, W_-$ and $W_+$ we define a random variable $W$ as follows
\begin{itemize}
\item if $A \leq \xi \leq B$ we set $W = \left( G_{n,m,z} \right)^{-1}( U_1)$;
\item if $\xi > B$ we set $W = W_+$;
\item if $\xi < A$ we set $W = W_-$.
\end{itemize}
It is easy to see that as defined $W$ indeed has the same distribution as $ S^{(N,z)}_n$. In words, $W$ is {\em quantile coupled} to $\xi$ near $0$ and independent from it for large values. 

We denote
$$Z = Z_{n,z} = z/2 +  \frac{\sigma_p \sqrt{N}}{2} \cdot \xi,  \hspace{5mm} \hat Z = \hat Z_{n,z} = z/2 +   \frac{\sigma_{z/N} \sqrt{N}}{2}\cdot \xi.$$
 and write $F = F_{n,z}$ for the cumulative distribution function of $\hat Z$. It is easy to check that our construction satisfies the following property. If $y \in [z/2 - 2n \epsilon_3, z/2 + 2n \epsilon_3]$ and $x > 0$ is fixed and 
$$F(y-x) \leq  G_{n,m,z}(y)  \leq F(y + x),$$
then
\begin{equation}\label{S6S1E2}
|\hat Z -   W  |  \leq x  \hspace{5mm} \mbox{ on the event $A \leq \xi \leq B$ } .
\end{equation}
By our choice of $\epsilon_3, N_3$ and $c_2$ and Corollary \ref{C2} applied to $ S_n^{(N,z)} - z/2 $ we have that for all $y \in [z/2 - 2n \epsilon_3, z/2 + 2n \epsilon_3]$
\begin{equation}\label{S6S1E3}
F\left( y- c_2\left[1 + \frac{(y- z/2)^2}{n}  \right]\right) \leq  G_{n,m,z}(y) \leq F\left( y + c_2\left[1 + \frac{(y- z/2)^2}{n}  \right]\right).
\end{equation}
Combining (\ref{S6S1E2}) and (\ref{S6S1E3}) we get 
\begin{equation}\label{S6S1E4}
 |\hat Z -   W   | \leq c_2\left[1 + \frac{(  W   - z/2)^2}{n}\right] \mbox{ almost surely on the event $A \leq \xi \leq B$},
\end{equation}
for all $n \geq N_3$, provided that $|z - pN| \leq \epsilon_3 N$, $|W - z/2| \leq 2\epsilon_3 n$. \\

We next claim that $|A| = O(\sqrt{N})$ and $|B| = O(\sqrt{N})$. To see the latter notice that 
$$\mathbb{P}( \xi  \geq B) = \mathbb{P} (W \geq z/2 + 2n \epsilon_3) = 1 - \mathbb{P}( W - z/2 \leq 2n \epsilon_3) \geq 1 - \mathbb{P}\left( \hat{Z} - \frac{z}{2} \leq 2n \epsilon_3 + c_2 \left(1 + \frac{4 n^2 (\epsilon_3)^2}{n} \right)  \right)$$
$$ = \mathbb{P} \left( \frac{\sigma_{z/N} \sqrt{N}}{2} \cdot \xi \geq 2n \epsilon_3 + c_2[1 + 4 n (\epsilon_3)^2] \right) \geq \mathbb{P}( \xi \geq \tilde{C} \sqrt{N}),$$
for some positive constant $\tilde{C}$. The inequality in the first line follows from Corollary \ref{C2} applied to $W - z/2$. The above implies that $B \leq \tilde{C}\sqrt{N}$ and an anologous argument shows that $A \geq - \tilde{C}\sqrt{N}$ for some possibly larger $\tilde{C}$. We conclude that there is a constant $\tilde{C} > 0$ such that $|\xi| \leq \tilde{C} \sqrt{N}$ on the event $A \leq \xi \leq B$. 

The latter implies that almost surely on the event $A \leq \xi \leq B$ we have
$$\mathbb{E}\left[ e^{|Z - \hat Z|} \Big{|} W\right] \leq \mathbb{E}\left[ e^{|\xi| | \sigma_p -  \sigma_{z/N}|}   \Big{|} W \right]\leq  \mathbb{E}\left[ e^{ \tilde{C} \sqrt{N} | \sigma_p -  \sigma_{z/N}|}   \Big{|} W \right].$$

From Lemma \ref{S2S1L1} we know that we can find a constant $c_p > 0$, that depends on $m_{\hat{s}, \hat{t}}$ and $M_{\hat{s},\hat{t}}$ as in Definition \ref{DefDelta} as well as $M^{(3)}_{\hat{s}, \hat{t}}$ as in Definition \ref{DefM34} for the variables $\hat{s}, \hat{t}$ as in Definition \ref{DefN2}, such that $|\sigma_p - \sigma_{z/N}|^2  \leq  c_p | p - z/N|^2$ for all $|z - pN| \leq \epsilon_3 N$. Combining the latter with the Cauchy-Schwarz inequality, (\ref{S6S1E4}) and the triangle inequality we conclude that there are constants $C, c > 0$ such that if $|W - z/2| \leq  2\epsilon_3 n$ and $|z - pN| \leq \epsilon_3 N$ then 
\begin{equation*}
\mathbb{E} \left[e^{|W - Z|} \Big{|}W \right] \leq \mathbb{E} \left[e^{|W - \hat Z| + |Z - Z|} \Big{|}W  \right] \leq C\exp\left(\frac{c_p(z-pN)^2}{N} + \frac{  c(W - z/2)^2}{n}\right).
\end{equation*}
Applying Jensen's inequality to the above we have for any $v \in \mathbb{N}$ that
\begin{equation*}
 \mathbb{E} \left[e^{(1/v)|W - Z|} \Big{|}W  \right] \leq \mathbb{E} \left[e^{|W - Z|} \Big{|}W \right]^{1/v} \leq  C^{1/v} \exp \left( \frac{c_p(z-pN)^2}{Nv} +  \frac{c(W - z/2)^2}{nv}\right),
\end{equation*}
and if we further use that $(x +y )^2 \leq 2x^2 + 2y^2$ above we see that
\begin{equation}\label{S6S1E6}
 \mathbb{E} \left[e^{(1/v)|W - Z|} \Big{|}W \right] \leq  C^{1/v} \cdot \exp \left( \frac{[c_p + c](z-pN)^2}{Nv} + \frac{4c(W - pn)^2}{Nv}\right),
\end{equation}
provided $ n \geq N_3$, $|w - z/2| \leq 2\epsilon_3 n$ and $|z - pN| \leq \epsilon_3 N$.\\

Suppose now that $b_1$ is given, and let $v$ be sufficiently large so that 
$$\frac{c_p + c}{v} \leq b_1 \mbox{ and } \frac{4c}{v} \leq b_1.$$
If $a_1 \leq 1/v$ we see from (\ref{S6S1E6}) that
\begin{equation}\label{S6S1E7}
 \mathbb{E} \left[e^{a_1|W - Z|} \Big{|}W \right] \leq  C \cdot \exp \left( \frac{b_1(z-pN)^2}{N} + \frac{b_1(w - pn)^2}{N}\right),
\end{equation}
provided $ n \geq N_3$, $|w - z/2| \leq 2\epsilon_3 n$ and $|z - pN| \leq \epsilon_3 N$.

Suppose now that $|W - z/2| > 2\epsilon_3 n$ and suppose for concreteness that $W - z/2 \geq 2\epsilon_3 n$. On the event $\{W > z/2 + 2\epsilon_3 n\}$ we have that $W$ and $Z$ are independent with $Z$ having the distribution of a normal random variable with mean $z/2$ and variance $\frac{\sigma_p^2 N}{4}$ conditioned on being larger than $s := z/2 + \frac{\sigma_p \sqrt{N}}{2} \cdot B$. It follows that almost surely on $\{W > z/2 + 2\epsilon_3 n\}$
$$ \mathbb{E} \left[e^{|W - Z|} \Big{|}W  \right]  \leq e^{ |W - z/2|} \cdot \int_{B}^\infty \frac{e^{ \frac{\sigma_p \sqrt{N}}{2} |y|} e^{-y^2/2}}{\sqrt{2\pi}} \cdot \left(1 - \Phi(B) \right)^{-1}.$$
From our earlier work we know that $B \leq \tilde{C} \sqrt{N}$ for some $\tilde{C} > 0$. This implies that
$$ 1 - \Phi(B)  \geq e^{-c N\epsilon_3^2},$$
for some sufficiently large $ c > 0$. Combining the last two inequalities gives for some new $c > 0$
$$ \mathbb{E} \left[e^{|W - Z|} \Big{|}W \right]  \leq \exp \left(c N +  |W - z/2| \right) \leq \exp \left( (c + 5/4)N + \frac{(z-pN)^2}{N} + \frac{(W - pn)^2}{N}\right),$$ 
where the last inequality uses the triangle inequality and the fact that $\sqrt{ab} \leq a + b$ for $a,b \geq 0$. 
Applying Jensen's inequality to the above we have for any $v \in \mathbb{N}$ that
\begin{equation*}
 \mathbb{E} \left[e^{(1/v)|W - Z|} \Big{|}W \right] \leq \mathbb{E} \left[e^{|W - Z|} \Big{|}W  \right]^{1/v}\hspace{-3mm} \leq  \exp \left( \frac{(c + 5/4)N}{v} + \frac{(z-pN)^2}{vN} + \frac{(W - pn)^2}{vN}\right).
\end{equation*}
In particular, suppose that $v$ is sufficiently large so that 
$$ \frac{1}{v} \leq \frac{b_1}{2} \mbox{ and } \frac{c+5/4}{v} \leq \frac{b_1 \epsilon_3^2}{8}$$
and $a_1 \leq 1/v$. We then have from the above inequality that
$$ \mathbb{E} \left[e^{a_1|W - Z|} \Big{|}W \right] \leq \exp \left(  \frac{b_1 \epsilon_3^2}{16} N + \frac{b_1(z-pN)^2}{2N} + \frac{b_1(W - pn)^2}{2N}\right) \leq \exp \left( \frac{b_1(z-pN)^2}{N} + \frac{b_1(W - pn)^2}{N}\right),$$
where in the last inequality we used that $|W - z/2| \geq 2\epsilon_3 n$ and $|z/2 - pn| \leq \epsilon_3 n$. We conclude that (\ref{S6S1E7}) holds even when $W - z/2 > 2\epsilon_3 n$ . An analogous argument shows that (\ref{S6S1E7}) also holds when $W - z/2 < - 2\epsilon_3 n$, and so almost surely for all $W$. This suffices for the proof.

\end{proof} 

We also isolate for future use the following statement.
\begin{lemma}\label{S6S1L2}
Assume the same notation as in Lemma \ref{S6S1L1}. There exist positive constants $b_2, c_2, N_4$ such that for every integers $m,n \geq N_4$, $N = m+n$ such that $|m - n| \leq 1$, every $z$ such that $| z - pN|\leq \epsilon' N$ and $w \in \mathbb{R}$,
$$f_{m,n}(w|z) \leq c_2  N^{-1/2} \exp \left( -b_2 \frac{(w - (z/2))^2}{N}\right).$$
\end{lemma}
The constants $b_2, c_2, N_4$ depend on $s, t, p$ and the constants in Definition \ref{DefParC}.
\begin{proof}
This is an immediate corollary of Propositions \ref{S3S1P2} and \ref{S3S1P3}.
\end{proof}

We now turn to the main theorem of this section.
\begin{theorem}\label{ContKMTA}
Suppose that $f_X$ satisfies Assumptions C1-C6 and fix $p \in (\alpha ,\beta)$. Let $s = p - \epsilon'$ and $t = p + \epsilon'$, where $\epsilon' > 0$ is sufficiently small so that $\alpha < s < t < \beta$. For every $b > 0$, there exist constants $0 < C, a, \alpha' < \infty$ such that for every positive integer $n$, there is a probability space on which are defined a Brownian bridge $B^\sigma$ with variance $\sigma^2 = \sigma^2_p$ and the family of processes $S^{(n,z)} $ for $z \in L_n$ such that
\begin{equation}\label{ContKMTeqA}
\mathbb{E}\left[ e^{a \Delta(n,z)} \right] \leq C e^{\alpha' (\log n)} e^{b|z- p n|^2/n},
\end{equation}
where $\Delta(n,z) = \Delta(n,z,B^{\sigma}, S^{(n,z)})=  \sup_{0 \leq t \leq n} \left| \sqrt{n} B^\sigma_{t/n} + \frac{t}{n}z - S^{(n,z)}_t \right|.$ The constants $C, a, \alpha'$ depend on $b$ as well as $s,t, p$ and $f_X$ through the constants in Definition \ref{DefParC} and the functions in Assumption C6.
\end{theorem}
\begin{proof}It suffices to prove the theorem when $b$ is sufficiently small. For the remainder we fix $b > 0$ such that $b < b_2/ 37$, where $b_2$ is the constant from Lemma \ref{S6S1L2}. Let $\epsilon_3$ and $N_3$ be as in Lemma \ref{S6S1L1} and $N_4$ as in Lemma \ref{S6S1L2} for our choice of $s,t$ and put $N_5 = \max (N_3, N_4)$.

 In this proof, by an {\em $n$-coupling} we will mean a probability space on which are defined a Brownian bridge $B^\sigma$ and the family of processes $\{ S^{(n,z)}: z \in L_n\}$. Notice that for any $n$-coupling if $z \in L_n$, $S_t = S^{(n,z)}_t$ then 
$$\Delta(n,z) =  \sup_{0 \leq t \leq n} \left| \sqrt{n} B^\sigma_{t/n} + \frac{t}{n}z - S^{(n,z)}_t \right| \leq |z| +  \max_{0 \leq k \leq n} |S^{(n,z)}_k| + \sup_{0 \leq t \leq n} |\sqrt{n} B^\sigma_{t/n}|$$ 
which implies
$$\mathbb{E}\left[ e^{a \Delta(n,z)} \right]  \leq \mathbb{E}\left[\exp \left(3a \sup_{0 \leq t \leq 1 }  \sqrt{n} |B^\sigma_t| \right) \right] + \exp( 3a|z|) + 
\mathbb{E} \left[ \exp \left(3a \max_{1 \leq k \leq n} |S_k| \right) \Big{|} S_n = z\right].$$
Note that if $|z - pn| \geq \epsilon_3 n$ we have 
$$b |z - pn|^2/n \geq  \frac{b \epsilon_3^2n}{2}  + \frac{b |z - pn|^2}{2n} \geq \frac{b \epsilon_3^2n}{2} + b \kappa \frac{z^2}{n},$$ 
where $\kappa$ is sufficiently small so that 
$$\kappa < 1/2, \hspace{5mm} \frac{p}{1 - 2\kappa} \in [p - \epsilon_3, p+\epsilon_3], \mbox{ and } \epsilon_3/2 - \kappa (\pm \epsilon_3 + p)^2 > 0 .$$ 
In view of the above and Assumption C6 there exists $\hat{a}$  small enough and $\hat{C}$ large enough depending on $b$ such that if $a < \hat{a}$ we can ensure that
$$\exp( 3a|z|) + 
\mathbb{E} \left[ \exp \left(3a \max_{1 \leq k \leq n} |S_k| \right) \Big{|} S_n = z\right] \leq \hat{C} e^{b|z- p n|^2/n},$$
provided that $|z - pn| \geq \epsilon_3n$. 

Further we know that there exist positive constants $\tilde c$ and $u$ such that $\mathbb{E}\left[\exp \left( \sup_{0 \leq t \leq 1 } y|B^\sigma_t| \right) \right] \leq \tilde c e^{uy^2}$ for any $y > 0$ (see e.g. (6.5) in \cite{LF}). Clearly, there exists $\hat{a}_2$ (depending on $b$) such that if $0 < a < \hat{a}_2$ then $18 ua^2 \leq b \epsilon_3^2$. This implies that if $a < a_0 := \min(\hat{a}, \hat{a}_2)$ then 
$$\mathbb{E}\left[\exp \left(3a \sup_{0 \leq t \leq 1 }  \sqrt{n} |B^\sigma_t| \right) \right] + \exp( 3a|z|) + 
\mathbb{E} \left[ \exp \left(3a \max_{1 \leq k \leq n} |S_k| \right) \Big{|} S_n = z\right] \leq  [\hat{C} + \tilde{c}] e^{b|z- p n|^2/n},$$
provided that $|z - pn| \geq \epsilon_3n$. 

The latter has the following implication. Firstly, (\ref{ContKMTeqA}) will hold for any $n$-coupling with $C = \hat{C}_1 := \tilde c + \hat{C}$, $\alpha' = 0$ and $a \in (0,a_0)$ if $z \in L_n$ satisfies $|z - pn| \geq \epsilon_3n$. Moreover, we can find a constant $\hat{C}_2 > 1$ such that if $a < a_0$, $|z - pn| \leq \epsilon_3n$ and $n \leq 4N_5$ then
$$\mathbb{E}\left[\exp \left(3a \sup_{0 \leq t \leq 1 }  \sqrt{n} |B^\sigma_t| \right) \right] + \exp( 3a|z|) + 
\mathbb{E} \left[ \exp \left(3a \max_{1 \leq k \leq n} |S_k| \right) \Big{|} S_n = z\right] \leq  \hat{C}_2.$$

For the remainder of the proof we take $b_1 = b/20$ and let $a_1, c_1$ be as in Lemma \ref{S6S1L1} for this value of $b_1$. We will take $a = (1/2) \cdot \min(a_0, a_1)$ and $C = \max(\hat{C}_1, \hat{C}_2)$ as above and show how to construct the $n$-coupling so that (\ref{ContKMTeqA}) holds for some $\alpha'$. \\

We will show that for every positive integer $s$, there exist $n$-couplings for all $n \leq 2^s$ such that
\begin{equation}\label{S42E7}
\mathbb{E}\left[ e^{a \Delta(n,z)} \right] e^{-b|z- p n|^2/n} \leq A^{s-1} \cdot C, \hspace{5mm} \forall{z \in L_n},
\end{equation}
where $A =  1 + 2c_1 ( 1 + 8c_2 b^{-1/2}) $. The theorem clearly follows from this claim.\\

We proceed by induction on $s$ with base case $s = 1$ being true by our choice of $C$ above. We suppose our claim is true for $s$ and let $2^s < n \leq 2^{s+1}$. We will show how to construct a probability space on which we have a Brownian bridge and a family of processes $\{S^{(n,z)}: |z -p n| \leq \epsilon_3 n \}$, which satisfy (\ref{S42E7}). Afterwards we can adjoin (after possibly enlarging the probability space) the processes for $|z| > n \epsilon_3$. Since $C \geq \hat{C}_1$ and $a < a_0$ we know that (\ref{S42E7}) will continue to hold for these processes as well. Hence, we assume that $|z - pn | \leq \epsilon_3 n$. 

If $2^{s+1} \leq 4N_5$ then by our choice of $C \geq \hat{C}_2$ and the fact that $A > 1$ we will have that (\ref{S42E7}) holds for any coupling provided $|z - pn | \leq \epsilon_3 n$. We may thus assume that $2^s > 2N_5$. For simplicity we assume that $n = 2k$, where $k \geq N_5$ is an integer such that $2^{s-1} < k \leq 2^s$ (if $n$ is odd we write $n = k + (k+1)$ and do a similar argument).

We define the $n$-coupling as follows:
\begin{itemize}
\item Choose two independent $k$-couplings 
$$\left( \{S^{1 (k,z))}\}_{z \in L_k}, B^1 \right), \hspace{5mm} \left( \{S^{2 (k,z))}\}_{z \in L_k}, B^2 \right), \mbox{ satisfying (\ref{S42E7})}.$$
Such a choice is possible by the induction hypothesis.
\item We let $W^z$ and $\xi$ be as in the statement of Lemma \ref{S6S1L1}, and set $Z^z = \frac{z}{2} + \frac{\sqrt{n}  \sigma_p}{2} \cdot \xi$. Assume, as we may, that all of these random variables are independent of the two $k$-couplings chosen above. Observe that by our choice of $a$ and $k \geq N_5$ we have that
\begin{equation}\label{S42Q4}
\mathbb{E} \left[e^{a |Z^z - W^z|} \Big{|} W^z \right] \leq c_1  \cdot  \exp \left(\frac{b}{20} \cdot \frac{(W^z - kp)^2 + (z - np)^2}{n}\right).
\end{equation}
\item Let
\begin{equation}\label{S42Q5}
B_t = \begin{cases} 2^{-1/2} B_{2t}^1 + t\sigma \xi & 0 \leq t \leq 1/2,\\2^{-1/2} B_{2(t-1/2)}^2 + (1-t) \sigma\xi & 1/2 \leq t \leq 1. \end{cases}
\end{equation}
By Lemma 6.5 in \cite{LF}, $B_t$ is a Brownian bridge with variance $\sigma^2$.
\item Let $S^{(n,z)}_k = W^z$, and 
$$S_{m}^{(n,z)} = \begin{cases}S^{1 (k,W^z)}_m &0 \leq m \leq k, \\ W^z + S_{m-k}^{2 (k,z-W^z)}, &k \leq m \leq n.\end{cases}$$
What we have done is that we first chose the value of $S_k^{(n,z)}$ from the conditional distribution of $S_k$, given $S_n = z$. Conditioned on the midpoint $S^{(n,z)}_k = W^z$ the two halves of the random walk bridge are independent and upto a trivial shift we can use $S^{1 (k,W^z)}$ and $S^{2 (k,z -W^z)}$ to build them.
\end{itemize}
The above defines our coupling and what remains to be seen is that it satisfies (\ref{S42E7}) with $s + 1$.

Note that 
$$\Delta(n,z,S^{(n,z)}, B) \leq |Z^z - W^z| + \max \left(\Delta(k, W^z, S^{1 (k,W^z)}, B^1),\Delta(k, z- W^z, S^{2 (k,z - W^z)}, B^2)  \right)$$
and therefore almost surely
$$\mathbb{E}\left[ e^{a\Delta(n,z)} \Big{|} W^z \right] \leq \mathbb{E}\left[ e^{a|Z^z - W^z|} \Big{|} W^z \right] \times  C A^{s-1}\left( e^{b|W^z - kp|^2/k} + e^{b|z- W^z - kp|^2/k}\right).$$
In deriving the last expression we used that our two $k$-couplings satisfy (\ref{S42E7}) and the simple inequality $\mathbb{E}[e^{\max (Z_1,Z_2)}] \leq\mathbb{E}[e^{Z_1}]  + \mathbb{E}[e^{Z_2}]  $. Taking expectation on both sides above we see that
\begin{equation}\label{S42Q6}
\mathbb{E}\left[ e^{a\Delta(n,z)} \right] \leq C\cdot (2c_1) \cdot A^{s-1} \mathbb{E} \left[ \exp \left(\frac{9}{4} \cdot \frac{b\max(|W^z - kp|^2,|z-W^z - kp|^2) }{n} \right) \right] .
\end{equation}
In deriving the last expression we used (\ref{S42Q4}) and the simple inequality $x^2 + y^2 \leq 5 \max (x^2, (x-y)^2)$ as well as that $k = n/2$.\\

We finally estimate the expectation in (\ref{S42Q6}) by splitting it over $W^z$ such that $|W^z - z/2| > |z-pn|/6$ and $|W^z - z/2| \leq |z-pn|/6$; we call the latter events $E_1$ and $E_2$ respectively. Notice that if $|W^z  - z/2| \leq |z-pn|/6$ we have $\max( |W^z - pk|^2,|z-W^z - pk|^2) \leq (2|z - pn|/3)^2$; hence
\begin{equation}\label{S42Q7}
\mathbb{E} \left[ \exp \left(\frac{9}{4} \cdot \frac{\max( |W^z - kp|^2,|z-W^z - kp|^2) }{n} \right) \cdot {\bf 1} \{ E_2 \}\right] \leq \exp \left(\frac{|z-pn|^2}{n} \right).
\end{equation}

To handle the case $|W^z - z/2| > |z-pn|/6$  we use Lemma \ref{S6S1L2}, from which we know that 
$$f_{m,n}(W_z|z)  \leq c_2 n^{-1/2} \exp \left(- b_2 \frac{(W^z - (z/2))^2}{n} \right).$$
Using the latter together with the fact that for $|W^z - z/2| > |z-pn|/6$ we have that $(W^z- z/2)^2 > \frac{1}{16} \max \left( (W^z - kp)^2 , |z-W^z - kp|^2 \right)$ we see that
\begin{equation}\label{S42Q8} 
\begin{split}
&\mathbb{E} \left[ \exp \left(\frac{9}{4} \cdot \frac{b \max( |W^z - kp|^2,|z-W^z - kp|^2) }{n} \right) \cdot {\bf 1} \{E_1 \} \right] \leq \\
& c_2 n^{-1/2} \int_{\mathbb{R}}  \exp \left(- \frac{b}{16} \cdot \frac{(y - kp)^2}{n} \right) dy=   c_2 n^{-1/2} 4  \frac{\pi^{1/2} n^{1/2}}{b^{1/2}} \leq 8c_2  b^{-1/2}.
\end{split}
\end{equation}
Combining the above estimates we see that
$$\mathbb{E}\left[ e^{a\Delta(n,z)} \right] \leq C\cdot (2c_1) \cdot A^{s-1} \left[ \exp \left(\frac{|z-pn|^2}{n} \right) + 8c_2 b^{-1/2} \right] \leq C \cdot A^{s} \exp \left(\frac{|z-pn|^2}{n} \right) .$$
The above concludes the proof.

\end{proof}

%%%%%%%%%%%%%%%%%%%%%%%%%%%%%%%%%%%%%%%%%%%%%%%%%%%%%%%%%%%%%%%
%
% Section 6.2
%
%%%%%%%%%%%%%%%%%%%%%%%%%%%%%%%%%%%%%%%%%%%%%%%%%%%%%%%%%%%%%%%

\subsection{Discrete case}
We use the same notation as in Sections \ref{Section2.2} and \ref{Section4}.

\begin{lemma}\label{S6S2L1}
Suppose that $p_X$ satisfies Assumptions D1-D4 and fix $p \in (\alpha ,\beta)$. Let $s = p -\epsilon'$ and $t = p + \epsilon'$, where $\epsilon' > 0$ is sufficiently small so that $\alpha < s < t < \beta$. Then there exists $\epsilon_3 \in (0, \epsilon')$ and $N_3 \in \mathbb{N}$ such that for every $b_1 > 0$ there exist constants $0 < c_1, a_1 < \infty$ such that the following holds. Suppose that $m,n$ are integers such that $m,n \geq N_3$ with $|m-n| \leq 1$, set $N = m + n$. We can define a probability space on which are defined a standard normal random variable $\xi$ and a collection of random variables $W = W^{(m,n,z)}$ for all $z \in \{ x \in L_N: |x - pN| \leq \epsilon_3N \}$ such that the law of $W^{(m,n,z)}$ is given by $p_{n,m}(\cdot|z)$ and such that we have almost surely 
\begin{equation}\label{S6S2E1}
\mathbb{E} \left[e^{a_1 |Z - W|} \Big{|} W \right] \leq c_1 \cdot  \exp \left(b_1 \frac{(W - pn)^2 + (z - pN)^2}{N}\right),
\end{equation}
where 
$$Z = Z^{(m,n,z)} = \frac{z}{2} + \frac{\sqrt{N} \sigma_p}{2} \cdot \xi, \mbox { so that } Z \sim N\left( \frac{z}{2}, \frac{\sigma^2_p N}{4}  \right).$$
The constants $\epsilon_3$ and $N_3$ depend on the values $p$, $s,t$ and the function $p_X(\cdot)$, where the dependence on the latter is through the constants in Definition \ref{DefParDisc}.
\end{lemma}
\begin{proof}
Notice that we only need to prove the lemma for $N$ sufficiently large. In order to simplify the notation we will assume that $n =m = N/2$ (the other cases can be handled similarly). \\

 We apply Propositions \ref{S4S1P2} and \ref{S4S1P3} for the variables $s$ and $t$. This implies that provided $N_3 \geq \max (N_0, N_1)$ as in the statements of those propositions and $n \geq N_3$ we have that the random variable $ S^{(N,z)}_n - z/2 $ satisfies the conditions of Lemma \ref{L2} for $M_0 = M$ as in Proposition \ref{S4S1P2}, $\epsilon_0 = \epsilon'$ as in the statement of this proposition, $\tilde{c} = (1/2) \cdot \min ( m_{\hat{s},\hat{t}}, M_{\hat{s},\hat{t}}^{-1})$ as in Definition \ref{DefDeltaDisc} for the variables $\hat{s}, \hat{t}$ as in Definition \ref{DefN2Disc}, $b' = a$ and $c' = A$ as in the statement of Proposition \ref{S4S1P3}. We consequently, let $c_2, N_2, \epsilon_2$ be as in the statement of that corollary for the above constants.

In what follows we fix $\epsilon_3 \leq 4^{-1}\min (\epsilon_2, \epsilon' )$ sufficiently small so that $\epsilon_3 M \leq 1/M_{\hat{s},\hat{t}}$ where $M$ is as in the statement of Proposition \ref{S4S1P2} and $M_{\hat{s},\hat{t}}$ is as in Definition \ref{DefDeltaDisc}  for the variables $\hat{s}, \hat{t}$ as in Definition \ref{DefN2Disc}. Observe that the choice of $\epsilon_3$ implies that $\epsilon_3 M \leq 1/ \sigma_{z/N}^2$ for all $|z -pN| \leq N \epsilon_3$. We also set $N_3 = \max (N_0, N_1, N_2)$.

Let $\hat{A} = \{ x \in \mathbb{Z}: x \in[z/2 - 2\epsilon_3 n , z/2 + 2\epsilon_3 n] \}$ and let $\hat{a}_1, \dots, \hat{a}_k$ be an enumeration of the elements in $\hat{A}$ in increasing order. Let $G = G_{n,z}$ denote the cumulative distribution function of $S^{(N,z)}_n$. In addition, we let $\Phi$ denote the cumulative distribution function of a standard normal random variable. Since $\Phi$ is strictly increasing and $p_{n,m}(\hat{a}|z) > 0$ for all $\hat{a} \in \hat{A}$ we can define the unique real numbers $r_{j-}$ and $r_{j}$ for $j = 1, \dots, k$ that satisfy
$$\Phi(r_{j-}) = G(\hat{a}_j -), \hspace{5mm} \Phi(r_j) = G(\hat{a}_j).$$

Suppose that we have a probability space that supports three independent variables $W_{-},W_{+}$ and $\xi$, where $\xi$ is a standard normal random variable, $W_{-}$ has the distribution of $S^{(N,z)}_n$ conditioned on being less than $\hat{a}_1$ and $W_{+}$ has the distribution of $S^{(N,z)}_n$ conditioned on being larger than $\hat{a}_k$. 
Set
$$Z = Z_{n,z} = z/2 +  \frac{\sigma_p \sqrt{N}}{2} \cdot \xi,  \hspace{5mm} \hat Z = \hat Z_{n,z} = z/2 +   \frac{\sigma_{z/N} \sqrt{N}}{2}\cdot \xi.$$
Given a realization of $\xi$, $W_{-}$ and $W_{+}$ we define a random variable $W$ as follows.
\begin{itemize}
\item if $ r_{j-} < \xi \leq r_j$ we set $W = \hat{a}_j$;
\item if $ \xi \leq r_{1-}$ we set $W = W_-$;
\item if $\xi \geq r_k$ we set $W = W_+$. 
\end{itemize}
It is easy to see that as defined $W$ indeed has the same distribution as $S^{(N,z)}_n$. In words, $W$ is {\em quantile coupled} to $\xi$ near $0$ and independent from it for large values. 

We denote
$$Z = Z_{n,z} = z/2 +  \frac{\sigma_p \sqrt{N}}{2} \cdot \xi,  \hspace{5mm} \hat Z = \hat Z_{n,z} = z/2 +   \frac{\sigma_{z/N} \sqrt{N}}{2}\cdot \xi.$$
 and write $F = F_{n,z}$ for the distribution function of $\hat Z$. It is easy to check that our construction satisfies the following property. If $j = 1, \dots, k$ and 
$$F(\hat{a}_j-x) \leq G(\hat{a}_j-) < G(\hat{a}_j) \leq F(\hat{a}_j + x),$$
then
\begin{equation}\label{S6S2L1E1}
|\hat Z - W| = |\hat Z - \hat{a}_j| \leq x \hspace{5mm} \mbox{ on the event } \{W = \hat{a}_j \} \mbox{ for $j = 1, \dots, k$}.
\end{equation}
By our choice of $\epsilon_3, N_3$ and $c_2$ and Lemma \ref{L2} we have that for all $j = 1, \dots, k$ and $n \geq N_3$
\begin{equation}\label{S6S2L1E2}
F\left( \hat{a}_j - c_2\left[1 + \frac{(\hat{a}_j- z/2)^2}{n}  \right]\right) \leq G(\hat{a}_j) \leq F\left(\hat{a}_j + c_2\left[1 + \frac{(\hat{a}_j- z/2)^2}{n}  \right]\right).
\end{equation}
Combining (\ref{S6S2L1E1}) and (\ref{S6S2L1E2}) we get 
\begin{equation}\label{S6S2L1E3}
 |\hat Z - W| \leq c_2\left[1 + \frac{(W- z/2)^2}{n}\right] \mbox{ on the event $W \in \hat{A}$},
\end{equation}
for all $n \geq N_3$, provided that $|z - pN| \leq \epsilon_3 N$, $|W - z/2| \leq 2\epsilon_3 n$. 

We next claim that $|r_{1-}| = O(\sqrt{N})$ and $|r_k| = O(\sqrt{N})$. To see the latter notice that 
$$\mathbb{P}( \xi \geq r_k) = \mathbb{P} (W \geq z/2 + 2n \epsilon_3) = 1 - \mathbb{P}( W - z/2 \leq 2n \epsilon_3) \geq 1 - \mathbb{P}\left( \hat{Z} - \frac{z}{2} \leq 2n \epsilon_3 + c_2 \left(1 + \frac{4 n^2 (\epsilon_3)^2}{n} \right)  \right)$$
$$ = \mathbb{P} \left( \frac{\sigma_{z/N} \sqrt{N}}{2} \cdot \xi \geq 2n \epsilon_3 + c_2[1 + 4 n (\epsilon_3)^2] \right) \geq \mathbb{P}( \xi \geq \tilde{C} \sqrt{N}),$$
for some positive constant $\tilde{C}$. The inequality in the first line follows from Lemma \ref{L2} applied to $W - z/2$. The above implies that $r_k \leq \tilde{C}\sqrt{N}$ and an anologous argument shows that $r_{1-} \geq - \tilde{C}\sqrt{N}$ for some possibly larger $\tilde{C}$. We conclude that there is a constant $\tilde{C} > 0$ such that $|\xi| \leq \tilde{C} \sqrt{N}$ on the event $W \in \hat{A}$. 

The latter implies that almost surely on the event $W \in \hat{A}$ we have
$$\mathbb{E}\left[ e^{|Z - \hat Z|} \Big{|} W\right] \leq \mathbb{E}\left[ e^{|\xi| | \sigma_p -  \sigma_{z/N}|}   \Big{|} W \right]\leq  \mathbb{E}\left[ e^{ \tilde{C} \sqrt{N} | \sigma_p -  \sigma_{z/N}|}   \Big{|} W \right].$$

From Lemma \ref{S2S2L1} we know that we can find a constant $c_p > 0$, that depends on $m_{\hat{s}, \hat{t}}$ and $M_{\hat{s},\hat{t}}$ as in Definition \ref{DefDeltaDisc} as well as $M^{(3)}_{\hat{s}, \hat{t}}$ as in Definition \ref{DefM34Disc} for the variables $\hat{s}, \hat{t}$ as in Definition \ref{DefN2Disc}, such that $|\sigma_p - \sigma_{z/N}|^2  \leq  c_p | p - z/N|^2$ for all $|z - pN| \leq \epsilon_3 N$. Combining the latter with the Cauchy-Schwarz inequality, (\ref{S6S2L1E3}) and the triangle inequality we conclude that there are constants $C, c > 0$ such that if $|W - z/2| \leq  2\epsilon_3 n$ and $|z - pN| \leq \epsilon_3 N$ then 
\begin{equation*}
\mathbb{E} \left[e^{|W - Z|} \Big{|}W \right] \leq \mathbb{E} \left[e^{|W - \hat Z| + |Z - Z|} \Big{|}W  \right] \leq C\exp\left(\frac{c_p(z-pN)^2}{N} + \frac{  c(W - z/2)^2}{n}\right).
\end{equation*}
Applying Jensen's inequality to the above we have for any $v \in \mathbb{N}$ that
\begin{equation*}
 \mathbb{E} \left[e^{(1/v)|W - Z|} \Big{|}W  \right] \leq \mathbb{E} \left[e^{|W - Z|} \Big{|}W \right]^{1/v} \leq  C^{1/v} \exp \left( \frac{c_p(z-pN)^2}{Nv} +  \frac{c(W - z/2)^2}{nv}\right),
\end{equation*}
and if we further use that $(x +y )^2 \leq 2x^2 + 2y^2$ above we see that
\begin{equation}\label{S6S2L1E6}
 \mathbb{E} \left[e^{(1/v)|W - Z|} \Big{|}W \right] \leq  C^{1/v} \cdot \exp \left( \frac{[c_p + c](z-pN)^2}{Nv} + \frac{4c(W - pn)^2}{Nv}\right),
\end{equation}
provided $ n \geq N_3$, $|w - z/2| \leq 2\epsilon_3 n$ and $|z - pN| \leq \epsilon_3 N$.\\

Suppose now that $b_1$ is given, and let $v$ be sufficiently large so that 
$$\frac{c_p + c}{v} \leq b_1 \mbox{ and } \frac{4c}{v} \leq b_1.$$
If $a_1 \leq 1/v$ we see from (\ref{S6S2L1E6}) that
\begin{equation}\label{S6S2L1E7}
 \mathbb{E} \left[e^{a_1|W - Z|} \Big{|}W \right] \leq  C \cdot \exp \left( \frac{b_1(z-pN)^2}{N} + \frac{b_1(w - pn)^2}{N}\right),
\end{equation}
provided $ n \geq N_3$, $|w - z/2| \leq 2\epsilon_3 n$ and $|z - pN| \leq \epsilon_3 N$.

Suppose now that $|W - z/2| > 2\epsilon_3 n$ and suppose for concreteness that $W - z/2 \geq 2\epsilon_3 n$. On the event $\{W > z/2 + 2\epsilon_3 n\}$ we have that $W$ and $Z$ are independent with $Z$ having the distribution of a normal random variable with mean $z/2$ and variance $\frac{\sigma_p^2 N}{4}$ conditioned on being larger than $s := z/2 + \frac{\sigma_p \sqrt{N}}{2} \cdot r_k$. It follows that almost surely on $\{W > z/2 + 2\epsilon_3 n\}$
$$ \mathbb{E} \left[e^{|W - Z|} \Big{|}W  \right]  \leq e^{ |W - z/2|} \cdot \int_{r_k}^\infty \frac{e^{ \frac{\sigma_p \sqrt{N}}{2} |y|} e^{-y^2/2}}{\sqrt{2\pi}} \cdot \left(1 - \Phi(r_k) \right)^{-1}.$$
From our earlier work we know that $r_k  \leq \tilde{C} \sqrt{N}$ for some $\tilde{C} > 0$. This implies that
$$ 1 - \Phi(r_k)  \geq e^{-c N\epsilon_3^2},$$
for some sufficiently large $ c > 0$. Combining the last two inequalities gives for some new $c > 0$
$$ \mathbb{E} \left[e^{|W - Z|} \Big{|}W \right]  \leq \exp \left(c N +  |W - z/2| \right) \leq \exp \left( (c + 5/4)N + \frac{(z-pN)^2}{N} + \frac{(W - pn)^2}{N}\right),$$ 
where the last inequality uses the triangle inequality and the fact that $\sqrt{ab} \leq a + b$ for $a,b \geq 0$. 
Applying Jensen's inequality to the above we have for any $v \in \mathbb{N}$ that
\begin{equation*}
 \mathbb{E} \left[e^{(1/v)|W - Z|} \Big{|}W \right] \leq \mathbb{E} \left[e^{|W - Z|} \Big{|}W  \right]^{1/v}\hspace{-3mm} \leq  \exp \left( \frac{(c + 5/4)N}{v} + \frac{(z-pN)^2}{vN} + \frac{(W - pn)^2}{vN}\right).
\end{equation*}
In particular, suppose that $v$ is sufficiently large so that 
$$ \frac{1}{v} \leq \frac{b_1}{2} \mbox{ and } \frac{c+5/4}{v} \leq \frac{b_1 \epsilon_3^2}{8}$$
and $a_1 \leq 1/v$. We then have from the above inequality that
$$ \mathbb{E} \left[e^{a_1|W - Z|} \Big{|}W \right] \leq \exp \left(  \frac{b_1 \epsilon_3^2}{16} N + \frac{b_1(z-pN)^2}{2N} + \frac{b_1(W - pn)^2}{2N}\right) \leq \exp \left( \frac{b_1(z-pN)^2}{N} + \frac{b_1(W - pn)^2}{N}\right),$$
where in the last inequality we used that $|W - z/2| \geq 2\epsilon_3 n$ and $|z/2 - pn| \leq \epsilon_3 n$. We conclude that (\ref{S6S2L1E7}) holds even when $W - z/2 > 2\epsilon_3 n$ . An analogous argument shows that (\ref{S6S2L1E7}) also holds when $W - z/2 < - 2\epsilon_3 n$, and so almost surely for all $W$. This suffices for the proof.

\end{proof} 

We also isolate for future use the following statement.
\begin{lemma}\label{S6S2L2}
Assume the same notation as in Lemma \ref{S6S2L1}. There exist positive constants $b_2, c_2, N_4$ such that for every integers $m,n \geq N_4$, $N = m+n$ such that $|m - n| \leq 1$, every $z \in \{ x \in L_N: |x - pN| \leq \epsilon_3N \}$ and $w \in \mathbb{Z}$,
$$p_{m,n}(w|z) \leq c_2  N^{-1/2} \exp \left( -b_2 \frac{(w - (z/2))^2}{N}\right).$$
\end{lemma}
The constants $b_2, c_2, N_4$ depend on $s, t, p$ and the constants in Definition \ref{DefParDisc}.
\begin{proof}
This is an immediate corollary of Propositions \ref{S4S1P2} and \ref{S4S1P3}.
\end{proof}

We now turn to the main theorem of this section.
\begin{theorem}\label{KMTA}
Suppose that $p_X$ satisfies Assumptions D1-D5 and fix $p \in (\alpha ,\beta)$. Let $s = p - \epsilon'$ and $t = p + \epsilon'$, where $\epsilon' > 0$ is sufficiently small so that $\alpha < s < t < \beta$. For every $b > 0$, there exist constants $0 < C, a, \alpha' < \infty$ such that for every positive integer $n$, there is a probability space on which are defined a Brownian bridge $B^\sigma$ with variance $\sigma^2 = \sigma^2_p$ and the family of processes $S^{(n,z)} $ for $z \in L_n$ such that
\begin{equation}\label{DiscKMTeqA}
\mathbb{E}\left[ e^{a \Delta(n,z)} \right] \leq C e^{\alpha' (\log n)} e^{b|z- p n|^2/n},
\end{equation}
where $\Delta(n,z) = \Delta(n,z,B^{\sigma}, S^{(n,z)})=  \sup_{0 \leq t \leq n} \left| \sqrt{n} B^\sigma_{t/n} + \frac{t}{n}z - S^{(n,z)}_t \right|.$ The constants $C, a, \alpha'$ depend on $b$ as well as $s,t, p$ and $p_X$ through the constants in Definition \ref{DefParDisc} and the functions in Assumption D5.
\end{theorem}
\begin{proof}It suffices to prove the theorem when $b$ is sufficiently small. For the remainder we fix $b > 0$ such that $b < b_2/ 37$, where $b_2$ is the constant from Lemma \ref{S6S2L2}. Let $\epsilon_3$ and $N_3$ be as in Lemma \ref{S6S2L1} and $N_4$ as in Lemma \ref{S6S2L2} for our choice of $s,t$ and put $N_5 = \max (N_3, N_4)$.

 In this proof, by an {\em $n$-coupling} we will mean a probability space on which are defined a Brownian bridge $B^\sigma$ and the family of processes $\{ S^{(n,z)}: z \in L_n\}$. Notice that for any $n$-coupling if $z \in L_n$, $S_t = S^{(n,z)}_t$ then 
$$\Delta(n,z) =  \sup_{0 \leq t \leq n} \left| \sqrt{n} B^\sigma_{t/n} + \frac{t}{n}z - S^{(n,z)}_t \right| \leq |z| +  \max_{0 \leq k \leq n} |S^{(n,z)}_k| + \sup_{0 \leq t \leq n} |\sqrt{n} B^\sigma_{t/n}|$$ 
which implies
$$\mathbb{E}\left[ e^{a \Delta(n,z)} \right]  \leq \mathbb{E}\left[\exp \left(3a \sup_{0 \leq t \leq 1 }  \sqrt{n} |B^\sigma_t| \right) \right] + \exp( 3a|z|) + 
\mathbb{E} \left[ \exp \left(3a \max_{1 \leq k \leq n} |S_k| \right) \Big{|} S_n = z\right].$$
Note that if $|z - pn| \geq \epsilon_3 n$ we have 
$$b |z - pn|^2/n \geq  \frac{b \epsilon_3^2n}{2}  + \frac{b |z - pn|^2}{2n} \geq \frac{b \epsilon_3^2n}{2} + b \kappa \frac{z^2}{n},$$ 
where $\kappa$ is sufficiently small so that 
$$\kappa < 1/2, \hspace{5mm} \frac{p}{1 - 2\kappa} \in [p - \epsilon_3, p+\epsilon_3], \mbox{ and } \epsilon_3/2 - \kappa (\pm \epsilon_3 + p)^2 > 0 .$$ 
In view of the above and Assumption D5 there exists $\hat{a}$  small enough and $\hat{C}$ large enough depending on $b$ such that if $a < \hat{a}$ we can ensure that
$$\exp( 3a|z|) + 
\mathbb{E} \left[ \exp \left(3a \max_{1 \leq k \leq n} |S_k| \right) \Big{|} S_n = z\right] \leq \hat{C} e^{b|z- p n|^2/n},$$
provided that $|z - pn| \geq \epsilon_3n$. 

Further we know that there exist positive constants $\tilde c$ and $u$ such that $\mathbb{E}\left[\exp \left( \sup_{0 \leq t \leq 1 } y|B^\sigma_t| \right) \right] \leq \tilde c e^{uy^2}$ for any $y > 0$ (see e.g. (6.5) in \cite{LF}). Clearly, there exists $\hat{a}_2$ (depending on $b$) such that if $0 < a < \hat{a}_2$ then $18 ua^2 \leq b \epsilon_3^2$. This implies that if $a < a_0 := \min(\hat{a}, \hat{a}_2)$ then 
$$\mathbb{E}\left[\exp \left(3a \sup_{0 \leq t \leq 1 }  \sqrt{n} |B^\sigma_t| \right) \right] + \exp( 3a|z|) + 
\mathbb{E} \left[ \exp \left(3a \max_{1 \leq k \leq n} |S_k| \right) \Big{|} S_n = z\right] \leq  [\hat{C} + \tilde{c}] e^{b|z- p n|^2/n},$$
provided that $|z - pn| \geq \epsilon_3n$. 

The latter has the following implication. Firstly, (\ref{DiscKMTeqA}) will hold for any $n$-coupling with $C = \hat{C}_1 := \tilde c + \hat{C}$, $\alpha' = 0$ and $a \in (0,a_0)$ if $z \in L_n$ satisfies $|z - pn| \geq \epsilon_3n$. Moreover, we can find a constant $\hat{C}_2 > 1$ such that if $a < a_0$, $|z - pn| \leq \epsilon_3n$ and $n \leq 4N_5$ then
$$\mathbb{E}\left[\exp \left(3a \sup_{0 \leq t \leq 1 }  \sqrt{n} |B^\sigma_t| \right) \right] + \exp( 3a|z|) + 
\mathbb{E} \left[ \exp \left(3a \max_{1 \leq k \leq n} |S_k| \right) \Big{|} S_n = z\right] \leq  \hat{C}_2.$$

For the remainder of the proof we take $b_1 = b/20$ and let $a_1, c_1$ be as in Lemma \ref{S6S2L1} for this value of $b_1$. We will take $a = (1/2) \cdot \min(a_0, a_1)$ and $C = \max(\hat{C}_1, \hat{C}_2)$ as above and show how to construct the $n$-coupling so that (\ref{DiscKMTeqA}) holds for some $\alpha'$. \\

We will show that for every positive integer $s$, there exist $n$-couplings for all $n \leq 2^s$ such that
\begin{equation}\label{S6S2T1E7}
\mathbb{E}\left[ e^{a \Delta(n,z)} \right] e^{-b|z- p n|^2/n} \leq A^{s-1} \cdot C, \hspace{5mm} \forall{z \in L_n},
\end{equation}
where $A =  1 + 2c_1 ( 1 + c_2(8 b^{-1/2} + 2)) $. The theorem clearly follows from this claim.\\

We proceed by induction on $s$ with base case $s = 1$ being true by our choice of $C$ above. We suppose our claim is true for $s$ and let $2^s < n \leq 2^{s+1}$. We will show how to construct a probability space on which we have a Brownian bridge and a family of processes $\{S^{(n,z)}: |z -p n| \leq \epsilon_3 n \}$, which satisfy (\ref{S6S2T1E7}). Afterwards we can adjoin (after possibly enlarging the probability space) the processes for $|z| > n \epsilon_3$. Since $C \geq \hat{C}_1$ and $a < a_0$ we know that (\ref{S6S2T1E7}) will continue to hold for these processes as well. Hence, we assume that $|z - pn | \leq \epsilon_3 n$. 

If $2^{s+1} \leq 4N_5$ then by our choice of $C \geq \hat{C}_2$ and the fact that $A > 1$ we will have that (\ref{S6S2T1E7}) holds for any coupling provided $|z - pn | \leq \epsilon_3 n$. We may thus assume that $2^s > 2N_5$. For simplicity we assume that $n = 2k$, where $k \geq N_5$ is an integer such that $2^{s-1} < k \leq 2^s$ (if $n$ is odd we write $n = k + (k+1)$ and do a similar argument).

We define the $n$-coupling as follows:
\begin{itemize}
\item Choose two independent $k$-couplings 
$$\left( \{S^{1 (k,z))}\}_{z \in L_k}, B^1 \right), \hspace{5mm} \left( \{S^{2 (k,z))}\}_{z \in L_k}, B^2 \right), \mbox{ satisfying (\ref{S42E7})}.$$
Such a choice is possible by the induction hypothesis.
\item We let $W^z$ and $\xi$ be as in the statement of Lemma \ref{S6S2L1}, and set $Z^z = \frac{z}{2} + \frac{\sqrt{n}  \sigma_p}{2} \cdot \xi$. Assume, as we may, that all of these random variables are independent of the two $k$-couplings chosen above. Observe that by our choice of $a$ and $k \geq N_5$ we have that
\begin{equation}\label{S6S2T1E8}
\mathbb{E} \left[e^{a |Z^z - W^z|} \Big{|} W^z \right] \leq c_1  \cdot  \exp \left(\frac{b}{20} \cdot \frac{(W^z - kp)^2 + (z - np)^2}{n}\right).
\end{equation}
\item Let
\begin{equation}\label{S6S2T1E9}
B_t = \begin{cases} 2^{-1/2} B_{2t}^1 + t\sigma \xi & 0 \leq t \leq 1/2,\\2^{-1/2} B_{2(t-1/2)}^2 + (1-t)\sigma\xi & 1/2 \leq t \leq 1. \end{cases}
\end{equation}
By Lemma 6.5 in \cite{LF}, $B_t$ is a Brownian bridge with variance $\sigma^2$.
\item Let $S^{(n,z)}_k = W^z$, and 
$$S_{m}^{(n,z)} = \begin{cases}S^{1 (k,W^z)}_m &0 \leq m \leq k, \\ W^z + S_{m-k}^{2 (k,z-W^z)}, &k \leq m \leq n.\end{cases}$$
What we have done is that we first chose the value of $S_k^{(n,z)}$ from the conditional distribution of $S_k$, given $S_n = z$. Conditioned on the midpoint $S^{(n,z)}_k = W^z$ the two halves of the random walk bridge are independent and upto a trivial shift we can use $S^{1 (k,W^z)}$ and $S^{2 (k,z -W^z)}$ to build them.
\end{itemize}
The above defines our coupling and what remains to be seen is that it satisfies (\ref{S6S2T1E7}) with $s + 1$.

Note that 
$$\Delta(n,z,S^{(n,z)}, B) \leq |Z^z - W^z| + \max \left(\Delta(k, W^z, S^{1 (k,W^z)}, B^1),\Delta(k, z- W^z, S^{2 (k,z - W^z)}, B^2)  \right)$$
and therefore almost surely
$$\mathbb{E}\left[ e^{a\Delta(n,z)} \Big{|} W^z \right] \leq \mathbb{E}\left[ e^{a|Z^z - W^z|} \Big{|} W^z \right] \times  C A^{s-1}\left( e^{b|W^z - kp|^2/k} + e^{b|z- W^z - kp|^2/k}\right).$$
In deriving the last expression we used that our two $k$-couplings satisfy (\ref{S6S2T1E7}) and the simple inequality $\mathbb{E}[e^{\max (Z_1,Z_2)}] \leq\mathbb{E}[e^{Z_1}]  + \mathbb{E}[e^{Z_2}]  $. Taking expectation on both sides above we see that
\begin{equation}\label{S6S2T1E10}
\mathbb{E}\left[ e^{a\Delta(n,z)} \right] \leq C\cdot (2c_1) \cdot A^{s-1} \mathbb{E} \left[ \exp \left(\frac{9}{4} \cdot \frac{b\max(|W^z - kp|^2,|z-W^z - kp|^2) }{n} \right) \right] .
\end{equation}
In deriving the last expression we used (\ref{S6S2T1E8}) and the simple inequality $x^2 + y^2 \leq 5 \max (x^2, (x-y)^2)$ as well as that $k = n/2$.\\

We finally estimate the expectation in (\ref{S6S2T1E10}) by splitting it over $W^z$ such that $|W^z - z/2| > |z-pn|/6$ and $|W^z - z/2| \leq |z-pn|/6$; we call the latter events $E_1$ and $E_2$ respectively. Notice that if $|W^z  - z/2| \leq |z-pn|/6$ we have $\max( |W^z - pk|^2,|z-W^z - pk|^2) \leq (2|z - pn|/3)^2$; hence
\begin{equation}\label{S6S2T1E11}
\mathbb{E} \left[ \exp \left(\frac{9}{4} \cdot \frac{\max( |W^z - kp|^2,|z-W^z - kp|^2) }{n} \right) \cdot {\bf 1} \{ E_2 \}\right] \leq \exp \left(\frac{|z-pn|^2}{n} \right).
\end{equation}

To handle the case $|W^z - z/2| > |z-pn|/6$  we use Lemma \ref{S6S2L2}, from which we know that 
$$p_{m,n}(W^z|z)  \leq c_2 n^{-1/2} \exp \left(- b_2 \frac{(W^z - (z/2))^2}{n} \right).$$
Using the latter together with the fact that for $|W^z - z/2| > |z-pn|/6$ we have that $(W^z- z/2)^2 > \frac{1}{16} \max \left( (W^z - kp)^2 , |z-W^z - kp|^2 \right)$ we see that
\begin{equation}\label{S6S2T1E12} 
\begin{split}
&\mathbb{E} \left[ \exp \left(\frac{9}{4} \cdot \frac{b \max( |W^z - kp|^2,|z-W^z - kp|^2) }{n} \right) \cdot {\bf 1} \{E_1 \} \right] \leq \\
& c_2 n^{-1/2} \sum_{y \in \mathbb{Z}}  \exp \left(- \frac{b}{16} \cdot \frac{(y - kp)^2}{n} \right) \leq    c_2 n^{-1/2} \left[2 + 4  \frac{\pi^{1/2} n^{1/2}}{b^{1/2}} \right] \leq c_2(8  b^{-1/2} + 2).
\end{split}
\end{equation}
Combining the above estimates we see that
$$\mathbb{E}\left[ e^{a\Delta(n,z)} \right] \leq C\cdot (2c_1) \cdot A^{s-1} \left[ \exp \left(\frac{|z-pn|^2}{n} \right) + c_2(8  b^{-1/2} + 2) \right] \leq C \cdot A^{s} \exp \left(\frac{|z-pn|^2}{n} \right) .$$
The above concludes the proof.

\end{proof}

\section{Assumptions D5 and C6}\label{Section7}

%-------------------------------------------------------------------------------------------------------------------------------------------------------------------------------------------------
% Section 7.1
%
%-------------------------------------------------------------------------------------------------------------------------------------------------------------------------------------------------
\subsection{Strongly unimodal distributions}\label{Section7.1}
In this section we give sufficient conditions for the technical Assumptions D5 and C6 to hold.

\subsubsection{Continuous case} The goal of this section is to give general conditions under which a distribution satisfying Assumptions C1-C5 will also satisfy Assumption C6. We use the same notation as in Sections \ref{Section2.1} and \ref{Section3}.

Let us introduce some useful notation. Let $f$ be a continuous probability density function on $\mathbb{R}$. We say that $f$ is {\em unimodal} if there exists at least one real number $M$ such that 
$$f(x) \leq f(y) \mbox{ for all $ x \leq y \leq M,$ and } f(x) \leq f(y) \mbox{ for all $ x \geq y \geq M$.}$$
We further say that $f(\cdot)$ is {\em strongly unimodal} if the convolution of $f(\cdot)$ with any unimodal distribution function $h(\cdot)$ on $\mathbb{R}$ is again unimodal. In \cite{Ibr56}, the author proved that $f(\cdot)$ is strongly unimodal if and only if it is log-concave, i.e.
$\log f$ is concave.

\begin{definition}\label{ExpBC}
Suppose that $f_X$ satisfies Assumptions C1-C5 and $\alpha = -\infty$, $\beta = \infty$. It follows from Assumption C2 that $X$ has all finite moments and we let $\mu = \mathbb{E}[X]$. In addition, we have $\Lambda'(0) = \frac{M'_X(0)}{M_X(0)} = \mu$ and so $u_\mu = (\Lambda')^{-1}( \mu)= 0$ and $G_\mu(u_\mu) = \Lambda(u_\mu) - u_\mu \cdot \mu = 0$. The latter and  Proposition \ref{S3S1P1} imply that there is a constant $\Delta > 0$ such that for all $n \geq 1$ we have
$$ \inf_{x \in [-1,1]} f_n(n\mu +  x) \geq n^{-1/2}\Delta .$$
Indeed, the latter is obvious from (\ref{S3S1E3}) for all large $n$ and for small $n$ we can deduce it from the continuity and positivity of $ f_n(n\mu +  x)$ on the interval $[-1, 1]$ from Assumption C1. The above implies that we can find a constant $R > 0$ such that $R > |\mu| + 1 + \Delta^{-1}$. 

In view of Proposition \ref{S3S1P1} applied to $s = -2R$ and $t = 2R$ we also deduce that there are positive constants $C_R$ and $c_R$ such that for all $n \geq 1$ and $z \in [-2R, 2R]$
$$f_n(nz) \geq C_Rn^{-1/2} e^{-c_R n}.$$
As before the above follows from Proposition \ref{S3S1P1} provided $n$ is sufficiently large, while for small $n$ it follows from the continuity and positivity of $f_n(nz)$ on $[-2R, 2R]$.

Finally, given the above constants, $\lambda$ as in Assumption C2 and $L$ as in Assumption C5, we can find constants $\hat{C}_R$ and $\hat{c}_R$ such that for all $n \geq 1$ we have
$$ \mathbb{E}[ e^{\lambda |X|}]^n \left[ \frac{4 n^{3/2}}{\Delta}  + LC^{-1}_R \sqrt{n} e^{c_R n}  \right] \leq \hat{C}_R \cdot e^{\hat{c}_R n}.$$
\end{definition}

The main result of the section is as follows.

\begin{lemma}\label{Section7L1}
Suppose that $f_X$ satisfies Assumptions C1-C5. Then it will also satisfy Assumption C6 if any of the following hold
\begin{itemize}
\item $\alpha > -\infty$;
\item $\beta < \infty$;
\item $\alpha = -\infty$, $\beta = \infty$ and the density function $f(x)$ of $X$ is a strongly unimodal function.
\end{itemize} 
Moreover, if $\alpha > -\infty$ then we can take $\hat{a}(\hat{b}) =  \frac{\hat{b}}{1 + \hat{b} + |\alpha|}$ and $\hat{C}(\hat{b}) = 1$; if $\beta < \infty$ then we can take $\hat{a}(\hat{b}) = \frac{\hat{b}}{1 + \hat{b} + |\beta|}$ and $\hat{C}(\hat{b}) = 1$. If $\alpha = -\infty$ and $\beta = \infty$ then we can choose $\hat{a}(\hat{b}) = \lambda v^{-1}$ and $\hat{C}(\hat{b}) = \hat{C}_R^{1/v}$, where $v$ is a large enough integer such that $c_R v^{-1} \leq \hat{b}/2$ and $\lambda v^{-1} \leq \hat{b}/2$ with $c_R, \hat{C}_R$ as in Definition \ref{ExpBC} and  $\lambda$ as in Assumption C2.
\end{lemma}
\begin{proof}
Assume first that $\alpha > -\infty$. Then we have for any $k \in \{1, \dots, n\}$ and $z \in L_n$ that 
$$S_k \geq -k \alpha \mbox{ and } S_n - S_k \geq - (n-k) \alpha \mbox{ almost surely.} $$
The latter implies that 
$$|z| + n |\alpha| \geq |S_k|,$$
which means using the ineqiality $|xy| \leq x^2 + y^2$ that
$$\mathbb{E} \left[ \exp \left(\hat{a} \max_{1 \leq k \leq n} |S_k| \right) \Big{|} S_n = z\right] \leq \exp \left( \hat{a}|z| + \hat{a} |\alpha| n\right) \leq \exp \left( \hat{a}|z|^2/n + \hat{a}n + \hat{a} |\alpha| n\right).$$
Thus if we choose $\hat{C} = 1$ and $\hat{a} = \frac{\hat{b}}{1 + \hat{b} + |\alpha|}$ we would obtain (\ref{S2S1E3}). An analogous argument establishes (\ref{S2S1E3}) when $\beta < \infty$. \\

In the remainder we focus on the last case. Notice that by assumption $f_m(x)$ are unimodal functions for any $m \geq 1$. For future use we call $\mu = \mathbb{E}[X]$ and for $|t| \leq \lambda$ as in Assumption C2 we set $M_{|X|}(t) = \mathbb{E} \left[e^{t |X|} \right]$. We also let $\Delta$, $R$, $C_R$, $c_R$, $\hat{C}_R$ and $\hat{c}_R$ be as in Definition \ref{ExpBC}.

By definition we have for $m \geq 1$ that
$$\inf_{x \in [-1, 1]} f_m\left( m \mu + x \right) \geq m^{-1/2} \cdot \Delta,$$
The latter implies that if $M_m$ is any real number such that $f_m(x) \leq f_m(y)$ for all $ x \leq y \leq M_m$ and $f_m(x) \leq f_m(y) $ for all $ x \geq y \geq M_m$, we then have $|M_m |\leq R m $. Indeed, if we suppose for example that $M_m > R m$ then this would mean that $f_m(t + m\mu) \geq f_m( m\mu)$ for all $t\in [0, (1 + \Delta^{-1}) m]$, so
$$\int_{0}^{(1 + \Delta^{-1}) m } f_m(t +m \mu  ) \geq \left((1 + \Delta^{-1}) m + 1\right) \cdot f_m( m \mu) \geq (1 + \Delta^{-1}) m^{1/2} \cdot \Delta > 1,$$
which is impossible. One rules out the case $M_m < - Rm$ in a similar fashion.\\

Let us now fix $n \geq 1$, $1 \leq m < n$, $|z| > 2Rn$ and $\lambda > 0$ as in Assumption C2. We then have that 
\begin{equation}
\begin{split}
& \mathbb{E} \left[ e^{\lambda|S_m|}\Big{|} S_n = z\right] = (I) + (II) + (III), \mbox{ where } (I) = \frac{\int_{|t| \leq |z| + Rn } f_m(t) f_{n-m}(z - t)  e^{\lambda |t|} dt }{\int_{\mathbb{R}} f_m(t) f_{n-m}(z - t) dt},\\
&(II) = \frac{\int_{t > |z| + Rn } f_m(t) p_{n-m}(z - t)  e^{\lambda |t|} dt }{\int_{\mathbb{R}} f_m(t) f_{n-m}(z - t)dt }, (III) =  \frac{\int_{t < -|z| - Rn } f_m(t) f_{n-m}(z - t)  e^{\lambda |t|} dt }{\int_{\mathbb{R}} f_m(t) f_{n-m}(z - t)dt }.
\end{split}
\end{equation}
Firstly, we have the trivial bound 
\begin{equation}\label{S73EQ1}
(I) \leq  e^{\lambda Rn + \lambda|z|}.
\end{equation}
In addition, if $z < -2Rn$ then by the unimodality of the density function $f_{n-m}(\cdot)$ we get
$$(II) \leq \frac{\int_{t >|z| + Rn } f_m(t) f_{n-m}(z - t)  e^{\lambda |t|} dt  }{\int_{m \mu}^{m \mu+1} f_m(t) f_{n-m}(z - t) dt} \leq  \frac{\sqrt{n}}{c} \cdot \int_{t >|z|+  Rn } f_m(t) e^{\lambda |t|} dt  \leq  \frac{\sqrt{n}}{\Delta} \mathbb{E} \left[ e^{\lambda |S_m|} \right] \leq  \frac{\sqrt{n}}{\Delta} M_{|X|}(\lambda)^n  .$$
On the other hand, if $z > 2Rn$ we have by the unimodality of $f_{m}(\cdot)$ that
$$(II) \leq \frac{\int_{t > Rn + |z| } f_m(t) f_{n-m}(z - t)  e^{\lambda|t|} dt }{ \int^{\mu (n-m) + 1}_{\mu (n-m)}f_{m}(z - t) f_{n-m}(t)dt} \leq  \frac{\sqrt{n}}{\Delta} \cdot \int_{t > Rn + |z| } f_{n-m}(z-t) e^{\lambda |t| dt}  \leq  \frac{\sqrt{n}}{\Delta} e^{\lambda z} M_{|X|}(\lambda)^n  .$$
Applying the same arguments to $(III)$ and combining the cases $z > 2Rn$ and $z < -2Rn$ we conclude that if $|z| > 2Rn$ we have
\begin{equation}\label{S73EQ2}
(II) + (III) \leq \frac{4 \sqrt{n}}{\Delta} e^{\lambda |z|} M_{|X|}(\lambda)^n 
\end{equation}
Combining (\ref{S73EQ1}) and (\ref{S73EQ2}) and the inequality
\begin{equation*}\label{S73EQ2.5}
 \mathbb{E} \left[ \exp \left(\lambda \max_{1 \leq k \leq n} |S_k| \right) \Big{|} S_n = z\right] \leq \sum_{m = 1}^n\mathbb{E} \left[ e^{\lambda|S_m|}\Big{|} S_n = z\right],
\end{equation*}
we conclude that if $|z| > 2Rn$ then
\begin{equation}\label{S73EQ3}
 \mathbb{E} \left[ \exp \left(\lambda \max_{1 \leq k \leq n} |S_k| \right) \Big{|} S_n = z\right]  \leq \frac{4 n^{3/2}}{\Delta} e^{\lambda |z|} M_{|X|} (\lambda)^n.
\end{equation}
Suppose now that $|z| \leq 2Rn$. Then by definition we have
\begin{equation*}
\begin{split}
&\mathbb{E} \left[ e^{\lambda|S_m|}\Big{|} S_n = z\right]  = \frac{\int_{\mathbb{R} } f_m(t) f_{n-m}(z - t)  e^{\lambda |t|} dt }{f_n(z)} \leq C^{-1}_R\sqrt{n} e^{c_R n} \int_{ \mathbb{R} } f_m(t) f_{n-m}(z - t)  e^{\lambda |t|} dt \\
 &\leq LC^{-1}_R \sqrt{n} e^{c_R n} \int_{\mathbb{R} } f_m(t)   e^{\lambda |t|}dt \leq LC^{-1}_R \sqrt{n} e^{c_R n}  M_{|X|}(\lambda)^n,
\end{split}
\end{equation*}
where $L$ is as in Assumption C5. 
 
Combining the latter with (\ref{S73EQ3}) we conclude that for any $z \in \mathbb{R}$ we have
 \begin{equation}\label{S73EQ4}
 \mathbb{E} \left[ \exp \left(\lambda \max_{1 \leq k \leq n} |S_k| \right) \Big{|} S_n = z\right]  \leq  \left[ \frac{4 n^{3/2}}{\Delta}  + LC^{-1}_R \sqrt{n} e^{c_R n}  \right] \cdot M_{|X|}(\lambda)^n \cdot e^{\lambda |z|} \leq \hat{C}_R \cdot e^{\hat{c}_R n + \lambda |z|}.
\end{equation}

From Jensen's inequality and (\ref{S73EQ4})  we know that for any $v \in \mathbb{N}$
\begin{equation}\label{S71EQ4}
 \mathbb{E} \left[ \exp \left(\lambda v^{-1} \max_{1 \leq k \leq n} |S_k| \right) \Big{|} S_n = z\right]  \leq\hat{C}^{1/v}_R \cdot e^{v^{-1}\hat{c}_R n + v^{-1}\lambda |z|}.
\end{equation}

Suppose now that $\hat{b} > 0$ is given. Then we can choose $v$ sufficiently large so that $\lambda/ v \leq  \hat{b}/2$ and $c_R/v \leq  \hat{b}/2$. Consequently, if we set $\hat{a}  = \lambda v^{-1}$ and $\hat{C}= \hat{C}^{1/v}_R$ we would have in view of (\ref{S71EQ4})
$$ \mathbb{E} \left[ \exp \left(\hat{a}\max_{1 \leq k \leq n} |S_k| \right) \Big{|} S_n = z\right]  \leq \hat{C} \cdot e^{(\hat{b}/2) (|z| + n)}  \leq \hat{C} \cdot e^{\hat{b} (n + z^2/n) },$$
where we used that $|z|/2 \leq z^2/n + n/2$ as follows by the Cauchy-Schwarz inequality.
\end{proof}

\subsubsection{Discrete case} In this section we give general conditions under which a distribution satisfying Assumptions D1-D4 will also satisfy Assumption D5. We use the same notation as in Sections \ref{Section2.2} and \ref{Section4}.

We first introduce some useful notation. Let $p(n)$ be a probability mass function on $\mathbb{Z}$. We say that $p$ is {\em unimodal} if there exists at least one integer $M$ such that 
$$p(n) \geq p(n-1) \mbox{ for all $n \leq M,$ and } p(n+1) \leq p(n) \mbox{ for all $n \geq M$.}$$
We further say that $p(\cdot)$ is {\em strongly unimodal} if the convolution of $p(\cdot)$ with any unimodal distribution function $h(\cdot)$ on $\mathbb{Z}$ is again unimodal. In \cite[Theorem 3]{KG71}, inspired by the classical work of \cite{Ibr56}, the authors proved that $p(\cdot)$ is strongly unimodal if and only if 
\begin{equation}\label{S71E1}
p(n)^2 \geq p(n-1) p(n+1) \mbox{ for all $n \in \mathbb{Z}$}.
\end{equation}

\begin{definition}\label{ExpBD}
Suppose that $p_X$ satisfies Assumptions D1-D4 and $\alpha = -\infty$, $\beta = \infty$. It follows from Assumption D2 that $X$ has all finite moments and we let $\mu = \mathbb{E}[X]$. In addition, we have $\Lambda'(0) = \frac{M'_X(0)}{M_X(0)} = \mu$ and so $u_\mu = (\Lambda')^{-1}( \mu)= 0$ and $G_\mu(u_\mu) = \Lambda(u_\mu) - u_\mu \cdot \mu = 0$. The latter and  Proposition \ref{S4S1P1} imply that there is a constant $\Delta > 0$ such that for all $n \geq 1$ we have
$$  p_n( \lfloor \mu m \rfloor ) \geq n^{-1/2}\Delta .$$
Indeed, the latter is obvious from (\ref{S4S1E3}) for all large $n$ and for small $n$ we can deduce it from the positivity of $ p_n(\lfloor \mu n \rfloor )$ from Assumption C1. The above implies that we can find a constant $R > 0$ such that $R > |\mu| + 1 + \Delta^{-1}$. 

In view of Proposition \ref{S4S1P1} applied to $s = -2R$ and $t = 2R$ we also deduce that there are positive constants $C_R$ and $c_R$ such that for all $n \geq 1$ and $z \in [-2R, 2R] \cap L_n$
$$p_n(z) \geq C_Rn^{-1/2} e^{-c_R n}.$$
As before the above follows from Proposition \ref{S4S1P1} provided $n$ is sufficiently large, while for small $n$ it follows from the positivity of $p_n(z)$ on $[-2R, 2R] \cap L_n$.

Finally, given the above constants, we can find constants $\hat{C}_R$ and $\hat{c}_R$ such that for all $n \geq 1$ we have
$$ \mathbb{E}[ e^{\lambda |X|}]^n \left[ \frac{4 n^{3/2}}{\Delta}  + LC^{-1}_R \sqrt{n} e^{c_R n}  \right] \leq \hat{C}_R \cdot e^{\hat{c}_R n}.$$
\end{definition}

The main result of the section is as follows.

\begin{lemma}\label{Section7L2}
Suppose that $p_X$ satisfies Assumptions D1-D4. Then it will also satisfy Assumption D5 if any of the following hold
\begin{itemize}
\item $\alpha > -\infty$;
\item $\beta < \infty$;
\item $\alpha = -\infty$, $\beta = \infty$ and $p_X(n)$ is a strongly unimodal function.
\end{itemize} 
Moreover, if $\alpha > -\infty$ then we can take $\hat{a}(\hat{b}) =  \frac{\hat{b}}{1 + \hat{b} + |\alpha|}$ and $\hat{C}(\hat{b})  = 1$; if $\beta < \infty$ then we can take $\hat{a}(\hat{b})  = \frac{\hat{b}}{1 + \hat{b} + |\beta|}$ and $\hat{C}(\hat{b})  = 1$. If $\alpha = -\infty$ and $\beta = \infty$ then we can choose $\hat{a}(\hat{b})  = \lambda v^{-1}$ and $\hat{C}(\hat{b})  = \hat{C}_R^{1/v}$, where $v$ is a large enough integer such that $c_R v^{-1} \leq \hat{b}/2$ and $\lambda v^{-1} \leq \hat{b}/2$ with $c_R, \hat{C}_R$ as in Definition \ref{ExpBD} and  $\lambda$ as in Assumption D2.
\end{lemma}
\begin{proof}
The cases $\alpha > -\infty$ and $\beta < \infty$ can be handled exactly the same as in the proof of Lemma \ref{Section7L1}. We focus on the case $\alpha = -\infty$ and $\beta = \infty$ in the remainder.

Notice that by assumption $p_m(n)$ are unimodal functions for any $m \geq 1$. For future use we call $\mu = \mathbb{E}[X]$ and for $|t| \leq \lambda$ as in Assumption D2 we set $M_{|X|}(t) = \mathbb{E} \left[e^{t |X|} \right]$.

By definition we have for $m \geq 1$ that
$$p_m\left( \lfloor m \mu \rfloor \right) \geq m^{-1/2} \cdot \Delta,$$
The latter implies that if $M_m$ is any integer such that $p_m(x) \leq p_m(y)$ for all $ x \leq y \leq M_m$ and $p_m(x) \leq p_m(y) $ for all $ x \geq y \geq M_m$, we then have $|M_m |\leq R m $.  Indeed, if we suppose for example that $M_m > R m$ then $p_m(n + \lfloor m\mu \rfloor) \geq p_m(\lfloor m\mu \rfloor)$ for all $n = 0, \dots, \lfloor (1 + \Delta^{-1}) m \rfloor$ and so
$$\sum_{n = 0}^{\lfloor (1 + \Delta^{-1}) m \rfloor} p_m(n + \lfloor m \mu \rfloor ) \geq (\lfloor (1 + \Delta^{-1}) m \rfloor + 1) \cdot p_m(\lfloor m \mu \rfloor) \geq (1 + \Delta^{-1}) m \frac{\Delta}{\sqrt{m}} > 1,$$
which is impossible. One rules out the case $M_m < - Rm$ in a similar fashion.\\

Let us now fix $n \geq 1$, $1 \leq m < n$, $|z| > 2Rn$ and $\lambda > 0$ as in Assumption D2. We then have that 
\begin{equation}
\begin{split}
& \mathbb{E} \left[ e^{\lambda|S_m|}\Big{|} S_n = z\right] = (I) + (II) + (III), \mbox{ where } (I) = \frac{\sum_{|k| \leq |z| + Rn } p_m(k) p_{n-m}(z - k)  e^{\lambda |k|}  }{\sum_{k \in \mathbb{Z}} p_m(k) p_{n-m}(z - k) },\\
&(II) = \frac{\sum_{k > |z| + Rn } p_m(k) p_{n-m}(z - k)  e^{\lambda |k|}  }{\sum_{k \in \mathbb{Z}} p_m(k) p_{n-m}(z - k) }, (III) =  \frac{\sum_{k < -|z| - Rn } p_m(k) p_{n-m}(z - k)  e^{\lambda |k|}  }{\sum_{k \in \mathbb{Z}} p_m(k) p_{n-m}(z - k) }.
\end{split}
\end{equation}
Firstly, we have the trivial bound 
\begin{equation}\label{S72EQ1}
(I) \leq e^{\lambda Rn + \lambda|z|}.
\end{equation}
In addition, we have that if $z < -2Rn$ then by the unimodality of the sequence $p_{n-m}(\cdot)$ we get
$$(II) \leq \frac{\sum_{k > Rn + |z| } p_m(k) p_{n-m}(z - k)  e^{\lambda |k|}  }{ p_m(M_m) p_{n-m}(z - M_m)} \leq  \frac{\sqrt{n}}{c} \cdot \sum_{k > Rn + |z| } p_m(k) e^{\lambda |k|}  \leq  \frac{\sqrt{n}}{c} \mathbb{E} \left[ e^{\lambda |S_m|} \right] \leq  \frac{\sqrt{n}}{c} M_{|X|}(\lambda)^n  .$$
On the other hand, if $z > 2Rn$ we have by the unimodality of $p_{m}(\cdot)$ that
$$(II) \leq \frac{\sum_{k > Rn + |z| } p_m(k) p_{n-m}(z - k)  e^{\lambda|k|}  }{ p_{m}(z - M_{n-m}) p_{n-m}(M_{n-m})} \leq  \frac{\sqrt{n}}{c} \cdot \sum_{k > Rn + |z| } p_{n-m}(z-k) e^{\lambda |k|}  \leq  \frac{\sqrt{n}}{c} e^{\lambda z} M_{|X|}(\lambda)^n  .$$
Applying the same arguments to $(III)$ and combining the cases $z > 2Rn$ and $z < -2Rn$ we conclude that if $|z| > 2Rn$ we have
\begin{equation}\label{S72EQ2}
(II) + (III) \leq \frac{4 \sqrt{n}}{c} e^{\lambda |z|} M_{|X|}(\lambda)^n 
\end{equation}
Combining (\ref{S72EQ1}) and (\ref{S72EQ2}) and the inequality
\begin{equation}\label{S72EQ2.5}
 \mathbb{E} \left[ \exp \left(\lambda \max_{1 \leq k \leq n} |S_k| \right) \Big{|} S_n = z\right] \leq \sum_{m = 1}^n\mathbb{E} \left[ e^{\lambda|S_m|}\Big{|} S_n = z\right],
\end{equation}
we conclude that if $|z| > 2Rn$ then
\begin{equation}\label{S72EQ3}
 \mathbb{E} \left[ \exp \left(\lambda \max_{1 \leq k \leq n} |S_k| \right) \Big{|} S_n = z\right]  \leq \frac{4 n^{3/2}}{c} e^{\lambda |z|} M_{|X|} (\lambda)^n \cdot e^{\lambda Rn},
\end{equation}
where we apply inequality $e^x + e^y \leq e^{x+y}$ for $x,y \geq 1$.

Suppose now that $|z| \leq 2Rn$. Then by definition we have
\begin{equation*}
\begin{split}
&\mathbb{E} \left[ e^{\lambda|S_m|}\Big{|} S_n = z\right]  = \frac{ \sum_{k \in \mathbb{Z}} p_m(k) p_{n-m}(z - k)  e^{\lambda |k|}  }{p_n(z)} \leq C^{-1}_R\sqrt{n} e^{c_R n} \sum_{k \in \mathbb{Z}} p_m(k) p_{n-m}(z - k)  e^{\lambda |k|} \\
 &\leq C^{-1}_R \sqrt{n} e^{c_R n} \sum_{k \in \mathbb{Z}} p_m(k)   e^{\lambda |k|} = C^{-1}_R \sqrt{n} e^{c_R n}  M_{|X|}(\lambda)^n.
\end{split}
\end{equation*}
 
Combining the latter with (\ref{S72EQ3}) we conclude that for any $z \in \mathbb{R}$ we have
 \begin{equation}\label{S72EQ4}
 \mathbb{E} \left[ \exp \left(\lambda \max_{1 \leq k \leq n} |S_k| \right) \Big{|} S_n = z\right]  \leq  \left[ \frac{4 n^{3/2}}{\Delta}  +C^{-1}_R \sqrt{n} e^{c_R n}  \right] \cdot M_{|X|} (\lambda)^n \cdot e^{\lambda |z|} \leq \hat{C}_R \cdot e^{\hat{c}_R n + \lambda |z|}.
\end{equation}
From here the proof proceeds as that of Lemma \ref{Section7L1}.
\end{proof}

%-------------------------------------------------------------------------------------------------------------------------------------------------------------------------------------------------
% Section 7.2
%
%-------------------------------------------------------------------------------------------------------------------------------------------------------------------------------------------------
\subsection{Insufficiency of Assumptions D1-D4}\label{Section7.2}
In this section we construct a probability distribution $p_X$, which satisfies Assumptions D1-D4, but for which the statement of Theorem \ref{S2DiscKMT} does not hold. The example illustrates that in general one needs further assumptions on $p_X$ in order to ensure the strong coupling of random walk bridges with step distribution $p_X$ and Brownian bridges of fixed variance. 

We will use the same notation as in Section \ref{Section2.1}. Suppose that $A = \{ x \in \mathbb{Z}: x = 3^n + n \mbox{ for some $n \in \mathbb{N}$} \}$ and $B =  \{ x \in \mathbb{Z}: x = -3^n  \mbox{ for some $n \in \mathbb{N}$} \}$. For convenience we denote $a_n = 3^n + n$ and $b_n = - 3^n$ for $n \geq 1$ and note that these are distinct integers. We define a weight function $w$ as follows
\begin{equation}\label{S7S2E1}
w(x) = \begin{cases} \exp( - x^2 ) \mbox{ if $x \in A\cup B$} , \\ \exp( - g(x)) \mbox{ if $\not \in A \cup B$, where $g(x) = 10^{10^{|x|}}$} \end{cases}
\end{equation}
Observe that $w(x) > 0$ for all $x \in \mathbb{Z}$ and $w(x) \leq e^{-x^2}$ for all $x \in \mathbb{Z}$. This means that $Z:= \sum_{x \in \mathbb{Z}} w(x) < \infty$ and the function
\begin{equation}\label{S7S2E2}
p_X(x) := w(x) \cdot Z^{-1}
\end{equation}
defines a probability mass function on $\mathbb{Z}$. We note that $p_X$ satisfies Assumption D1, with $\alpha = -\infty$ and $\beta = \infty$; Assumption D2 with any $\lambda > 0$, in particular we have $\mathcal{D}_{\Lambda} = \mathbb{R}$ and so by Lemma \ref{S2S2L1} we know that $\Lambda_X$ is continuous on $\mathbb{R}$ so that Assumption D3 is also satisfied. Finally, by definition $p_X(x) \leq Z^{-1} e^{-x^2}$ and so Assumption D4 is satisfied with $D = Z^{-1}$ and $d = 1$. Overall, we see that $p_X$ satisfies Assumptions D1-D4. \\

Suppose now that $S^{(n,z)}$ is a random walk bridge whose steps size is $p_X$. We want to show that for any $a,c ,C > 0$ and $\sigma > 0$ and any coupling of $S^{(2,z)}$ with a Brownian bridge $B^{\sigma}$ of variance $\sigma^2$  there exists a $z \in \mathbb{Z}$ such that
\begin{equation}\label{S7DiscKMTBreak}
\mathbb{E}\left[ e^{a \Delta(2,z)} \right] \geq C e^{c|z|^2},
\end{equation}
where $\Delta(n,z) = \Delta(n,z,B^{\sigma}, S^{(n,z)})=  \sup_{0 \leq t \leq n} \left| \sqrt{n} B^\sigma_{t/n} + \frac{t}{n}z - S^{(n,z)}_t \right|.$ 
The latter statement implies that we cannot couple the bridge of size two to any fixed variance Brownian bridge uniformly in the endpoint $z$, which means that Theorem \ref{S2DiscKMT} fails to hold for this bridge.
 
\begin{remark}\label{S7S2R1}
Let us heuristically explain why the above example breaks the coupling. The distribution in (\ref{S7S2E2}) satisfies the condition that it has spikes at the points in $A$ and $B$ and is extremely small away from those sets. The latter means that for certain large enough $z$, we will have that conditional on $X_1 + X_2 = z$, with overwhelming probability $X_1 = 3^z + z$ and $X_2 = -3^z$ or $X_1 = -3^z$ and $X_2 = 3^z + z$. The latter implies that the midpoint of the bridge is essentially a Bernoulli variable that takes the values $3^z + z$ and $-3^z$ with equal probability. This makes its variance increase as we increase $z$, which makes a close coupling to a Brownian bridge of fixed variance impossible. 

The main take-away point is that while $p_X$ may be an extremely well-behaved distribution, the conditional distribution of the midpoint of a Bridge with step size $p_X$ can become quite singular in the presence of spikes in $p_X$. This means that one needs better control of the conditional distribution, and one way to achieve this is to assume $p_X$ has no spikes. This is one reason behind our introduction of the strongly log-concave distributions in Section \ref{Section7.1} above.     
\end{remark}

In the remainder we prove (\ref{S7DiscKMTBreak}). We will prove that there are large enough $z$ such that
\begin{equation*}
\mathbb{E}\left[ e^{a |S_1^{(2,z)} - \sqrt{2} B^{\sigma}_{1/2} - z/2| } \right] \geq C e^{c|z|^2},
\end{equation*}
which certainly implies (\ref{S7DiscKMTBreak}). Using that $e^{a |x - y|} \geq e^{a|x| - a|y|} \geq e^{(a/2) |x| } - e^{a|y|}$ we see that
$$\mathbb{E}\left[ e^{a |S_1^{(2,z)} - \sqrt{2} B^{\sigma}_{1/2} - z/2| } \right] \geq \mathbb{E}\left[ e^{(a/2) |S_1^{(2,z)}|} \right] + \mathbb{E} \left[ e^{a|\sqrt{2} B^{\sigma}_{1/2}| + |az/2| } \right] =\mathbb{E}\left[ e^{(a/2) |S_1^{(2,z)}|} \right] - \mathbb{E} \left[ e^{a|\sqrt{2} B^{\sigma}_{1/2}| + |az/2| } \right].$$
Furthermore we have
$$\mathbb{E} \left[ e^{a|\sqrt{2} B^{\sigma}_{1/2}| + |az/2| } \right] \leq e^{|az|/2} \cdot \left[ \mathbb{E} \left[ \exp \left( a\sqrt{2}B^{\sigma}_{1/2}  \right) \right] + \mathbb{E} \left[ \exp \left( -a\sqrt{2}B^{\sigma}_{1/2}  \right) \right] \right] = 2e^{|az|/2 +  a^2 \sigma^2/4} .$$
Combining the above statements we see that to prove (\ref{S7DiscKMTBreak}) it is enough to show that for any fixed $a, c,C > 0$ we can find large enough $z$ so that
\begin{equation}\label{S7S2E3}
\mathbb{E}\left[ e^{a |S_1^{(2,z)}| } \right] \geq C e^{c|z|^2}.
\end{equation}
This is the statement we will establish.

We claim that if $z = 2 \cdot 3^m$ with $m$ sufficiently large we have
\begin{equation}\label{S7S2E4}
3 p_1(a_z) p_1(b_z)  \geq p_2(z).
\end{equation}
If true the above would imply
\begin{equation*}
\mathbb{E}\left[ e^{a |S_1^{(2,z)}| } \right]  = \sum_{k \in \mathbb{Z}} \frac{p_1(k) p_1(z- k) e^{ak}}{p_2(z)} \geq \frac{p_1(a_z) p_1(b_z)e^{a a_z}}{p_2(z)} \geq \frac{1}{3} \cdot e^{a(3^z+z)}, 
\end{equation*}
which certainly implies (\ref{S7S2E3}). We thus focus on (\ref{S7S2E4}).

We have for all $m \geq 2$ that
\begin{equation}\label{S7S2E5}
\begin{split}
&p_2(z) \leq (I) + (II), \mbox{ where } (I) = 2\sum_{r = 1}^\infty {\bf 1} \{k \not \in A, z - k \not \in A\}  p_1(k) p_1(z-k), \\
& (II) = 2\sum_{ r  = 1 }^\infty  p_1(a_r) p_1(z-a_r).
\end{split}
\end{equation} 

If $r \leq m$ then we have $4 \leq  a_r \leq 3^{m} + m $ and so $3^{m+1} + m \geq  2 \cdot 3^m - 4 \geq z - a_r \geq  2 \cdot 3^m - 3^{m} - m \geq 3^{m-1} + m.$ This means that $z- a_r \not \in A\cup B$ and so
\begin{equation}\label{S7S2E6}
\sum^{m}_{r = 1} p_1(a_r) p_1(z-a_r) \leq \sum^{m}_{r = 1} p_1(z-a_r)  \leq  Z^{-1} \cdot m \cdot \exp \left( - g(3^{m-1} + m)  \right). 
\end{equation}
If $ m < r < 2 \cdot 3^m $ then we have $z - a_r = 2 \cdot 3^m - 3^r - r$ and so
$$ -3^{r-1} > z - a_r > -3^{r} .$$
This means that $z - a_r \not \in A \cup B$ and so
\begin{equation}\label{S7S2E7}
\sum^{z - 1}_{r = m+1} p_1(a_r) p_1(z-a_r) \leq \sum^{z-1}_{r = m+1} p_1(z-a_r)  \leq  Z^{-1} \cdot (z - m) \cdot \exp \left( - g(3^m) \right). 
\end{equation}
If $2 \cdot 3^m < r$ then 
$$ - 3^{r} > z - a_r = 2 \cdot 3^m - 3^r - r > - 3^{r+1}.$$
This means that $z - a_r \not \in A \cup B$ and so
\begin{equation}\label{S7S2E8}
\sum^{\infty}_{r = z+1} p_1(a_r) p_1(z-a_r) \leq \sum^{\infty}_{r = z+1} p_1(z-a_r)  \leq  Z^{-1} \cdot  \sum^{\infty}_{r = z+1}  \exp \left( - g(3^r)  \right). 
\end{equation}
Combining (\ref{S7S2E6}), (\ref{S7S2E7}) and (\ref{S7S2E8}) we have
\begin{equation}\label{S7S2E9}
(II) - 2  p_1(a_z) \cdot p_1(z - a_z) \leq Z^{-1} \cdot  \sum^{\infty}_{r = z+1}  \exp \left( - g(3^r)  \right) + Z^{-1} \cdot z \cdot \exp( - g(z/6)) \leq e^{ - g(z/10)},
\end{equation}
where the last inequality holds provided $m$ (and hence $z$) is sufficiently large. On the other hand,
$$p_1(a_z) \cdot p_1(z - a_z) = \exp( - a_z^2) \cdot \exp ( -b_z^2) = \exp \left( - (3^z + z)^2 - 3^{2z} \right) \geq 10 \cdot  e^{ - g(z/10)},$$
for all large enough $m$ and so we conclude that for all large $m$ and $z = 2 \cdot 3^m$ we have
\begin{equation}\label{S7S2E10}
(II) \leq  (2.2) \cdot  p_1(a_z) \cdot p_1(z - a_z).
\end{equation}
We next focus on $(I)$. Notice that if $k \leq 3^m$ then $z - k \geq 3^m$ and so
$$\sum_{r = 1}^{z/2} {\bf 1} \{k \not \in A, z - k \not \in A\}  p_1(k) p_1(z-k) \leq \sum_{r = 1}^{z/2} {\bf 1} \{k \not \in A, z - k \not \in A\}  p_1(z-k) \leq (z/2) \cdot \exp(-g(z/2)). $$
In addition, we have
$$\sum_{r = z/2 + 1}^{\infty} {\bf 1} \{k \not \in A, z - k \not \in A\}  p_1(k) p_1(z-k) \leq $$
$$\sum_{r = z/2 + 1}^{\infty}  {\bf 1} \{k \not \in A, z - k \not \in A\}  p_1(k) \leq \sum_{r = z/2 + 1}^\infty \exp ( - g(r)) \leq \exp(- g(z/3)), $$
 where the last inequality holds for all large enough $m$. Combining the latter we get for all large $m$
\begin{equation}\label{S7S2E11}
(I) \leq   z \cdot \exp(-g(z/2)) + 2\cdot  \exp(- g(z/3)) \leq \exp(-g(z/10)) \leq (0.1) \cdot  p_1(a_z) \cdot p_1(z - a_z).
\end{equation}
Combining (\ref{S7S2E10}) and (\ref{S7S2E11}) we conclude (\ref{S7S2E4}), which concludes our proof.

\section{Examples}\label{Section8}
In this section we present several examples of distributions that satisfy Assumptions C1-C6 in Section \ref{Section8.1} and Assumptions D1-D5 in Section \ref{Section8.2}. The goal is to illustrate how to verify that a given distribution satisfies the assumptions and in particular prove Theorems \ref{S1ContKMT} and \ref{S1DiscKMT}. In Section \ref{Section8.3} we discuss an example with the log-gamma distribution with parameter $\gamma > 0$. The log-gamma distribution is of interest to us due to connections to integrable probability and the example we consider is the principal one that motivated our quantified Theorem \ref{ContKMTA}. This example will benefit the future work \cite{CNW}.

%-------------------------------------------------------------------------------------------------------------------------------------------------------------------------------------------------
% Section 8.1
%
%-------------------------------------------------------------------------------------------------------------------------------------------------------------------------------------------------
\subsection{Examples: continuous jumps}\label{Section8.1} 

We continue with the notation from Section \ref{Section2.1}.\\

{\bf \raggedleft Example 1.} We consider the distributions in Theorem \ref{S1ContKMT}. By assumption we know that $X$ is a continuous random variable with density $f_X(\cdot)$, which has a compact interval of support $[\alpha, \beta]$ and which is continuously differentiable and positive on $(\alpha, \beta)$ with a bounded derivative. Since the derivative of $f_X$ is bounded and continuous on $(\alpha, \beta)$ we conclude that $f_X$ can be continuously extended to $[\alpha, \beta]$ and so Assumption C1 is satisfied. In addition, since $X$ is uniformly bounded, we see that Assumption C2 is satisfied for any $\lambda >0$ and so $\mathcal{D}_\Lambda = \mathbb{R}$. The latter and Lemma \ref{S2S1L1} imply that $\Lambda(\cdot)$ is continuous on $\mathbb{R}$ and so Assumption C3 holds.

We next observe using integration by parts that if $z \in \mathbb{C}$ and $z \neq 0 $ we have
$$\int_\alpha^{\beta} f_X(x) e^{x z} dx = f_X(\beta) \cdot \frac{e^{\beta z}}{z}   -  f_X(\alpha) \cdot \frac{e^{\alpha z}}{z} - \int_{\alpha}^\beta f'_X(x) \cdot \frac{e^{x z}}{z}dx.$$
Let us fix $s,t \in \mathbb{R}$ with $\alpha <  s < t < \beta$ and suppose that $z = u + i v$ with $u \in [s,t]$. Then the boundedness of $f_X(\cdot)$ and $f'_X(\cdot)$ and the above equation imply that 
$$ \left|\int_\alpha^{\beta} f_X(x) e^{x z} dx  \right| \leq \frac{K_1(s,t)}{1 + |v|},$$
for some sufficiently large constant $K_1(s,t)$ and so Assumption C4 holds with $p(s,t) = 1$.

As $f_X(\cdot)$ has compact support and is bounded, Assumption C5 holds as well. In view of Lemma \ref{Section7L1} Assumption C6 also holds. Overall, we conclude that $f_X$ satisfies Assumptions C1-C6 and so by Theorem \ref{S2ContKMT} we conclude Theorem \ref{S1ContKMT}.\\

The above example illustrates that our strong coupling result holds for essentially any compactly supported density with a bounded continuous derivative. We next illustrate a case when the support is not compact using the usual exponential distribution.\\

{\bf \raggedleft Example 2.} Suppose that $X$ has exponential distribution with parameter $\mu > 0$, i.e. $f_{X}(x) = {\bf 1}\{x > 0 \} \cdot \mu e^{-\mu x}$. In this case Assumption C1 holds trivially with $\alpha = 0$ and $\beta = \infty$. In addition, we have $M_X(t) = \frac{\mu}{\mu - t}$ and so Assumption C2 holds with any $0 < \lambda < \mu$. Next, we have that $\Lambda(x) =   \log (\mu) - \log (\mu -t)$ is lower semi-continuous on $\mathcal{D}_\lambda = (-\infty, \mu)$ and Assumption C3 holds. 

Let us fix $s,t \in \mathbb{R}$ with $0 <  s < t < \infty$ and suppose that $z = u + i v$ with $u \in [s,t]$. Then we have
$$\left| M_X(z)  \right| = \left|\frac{\mu}{\mu - z} \right| \leq  \frac{K_1(s,t)}{1 + |v|}$$
for some sufficiently large constant $K_1(s,t)$ and so Assumption C4 holds with $p(s,t) = 1$. Assumption C5 holds trivially as $f_X(x) = 0$ for $ x \leq 0 $ and Assumption C6 is satisfied in view of Lemma  \ref{Section7L1}. Overall, we conclude that $f_X$ satisfies Assumptions C1-C6 and so Theorem \ref{S2ContKMT} holds for random walk bridges with exponential jumps.

%-------------------------------------------------------------------------------------------------------------------------------------------------------------------------------------------------
% Section 8.2
%
%-------------------------------------------------------------------------------------------------------------------------------------------------------------------------------------------------
\subsection{Examples: discrete jumps}\label{Section8.2}

We continue with the notation from Section \ref{Section2.2}.\\

{\bf \raggedleft Example 1.} We consider the distributions in Theorem \ref{S1DiscKMT}. By assumption we know that $X$ is an integer valued random variable with probability mass function $p_X(\cdot)$ such that $p_X(x) > 0$ for all $x \in \mathbb{Z} \cap [\alpha, \beta]$ and $\mathbb{P}(X \in [\alpha, \beta ]) = 1$. The latter iplies that $p_X$ satisfies Assumption D1. In addition, since $X$ is uniformly bounded, we see that Assumption D2 is satisfied for any $\lambda >0$ and so $\mathcal{D}_\Lambda = \mathbb{R}$. The latter and Lemma \ref{S2S2L1} imply that $\Lambda(\cdot)$ is continuous on $\mathbb{R}$ and so Assumption D3 holds. As $p_X(\cdot)$ is compactly supported and bounded, Assumption D4 holds as well. In view of Lemma \ref{Section7L2} Assumption D5 also holds. Overall,  we conclude that $p_X$ satisfies Assumptions D1-D5 and so by Theorem \ref{S2DiscKMT} we conclude Theorem \ref{S1DiscKMT}.\\

The above example illustrates that our strong coupling result holds for essentially any integer valued variable with a single compact (integer) interval of support. We next illustrate a case when the support is not compact using the usual geometric distribution.\\

{\bf \raggedleft Example 2.} Suppose that $X$ has geometric distribution with parameter $q \in (0,1)$, i.e. $p_{X}(n) = q \cdot (1-q)^n$ for $n \geq 0$. In this case Assumption D1 holds trivially with $\alpha = 0$ and $\beta = \infty$. In addition, we have $M_X(t) = \frac{q}{1 - (1-q)e^t}$ and so Assumption D2 holds with any $0 < \lambda < - \log (1- q)$. Next, we have that $\Lambda(x) =   \log (q) - \log (1 - (1-q)e^t)$ is lower semi-continuous on $\mathcal{D}_\lambda = (-\infty, - \log (1- q) )$ and Assumption D3 holds. 

Assumption D4 holds trivially as $p_X(x) = 0$ for $ x < 0 $ and Assumption D5 is satisfied in view of Lemma  \ref{Section7L2}. Overall, we conclude that $p_X$ satisfies Assumptions D1-D5 and so Theorem \ref{S2DiscKMT} holds for random walk bridges with geometric jumps.

%-------------------------------------------------------------------------------------------------------------------------------------------------------------------------------------------------
% Section 8.3
%
%-------------------------------------------------------------------------------------------------------------------------------------------------------------------------------------------------
\subsection{Example: log-gamma distribution}\label{Section8.3}

The log-gamma density function with parameter $\gamma > 0$ is given by
\begin{equation}\label{S83E1}
f_\gamma(x)=\frac{1}{\Gamma(\gamma)}\exp\left( \gamma x-e^{x} \right) \mbox{ for $x \in \mathbb{R}$}.
\end{equation}
If $\xi$ is a random variable with density $f_\gamma$ one readily observes that
\begin{equation}\label{S83E2}
M_\xi(t) = \frac{\Gamma(\gamma + t)}{\Gamma(\gamma)}, \mbox{ and so } M_\xi(t) < \infty \mbox{ for $t > - \gamma$}.
\end{equation}
The above formula also implies that 
\begin{equation}\label{S83E3}
\mathbb{E}[\xi] = m_\gamma = \psi^{(0)}(\gamma) \mbox{ and } Var(\xi) = \sigma_\gamma^2 = \psi^{(1)}(\gamma), 
\end{equation}
where $\psi^{(k)}$ denote the polygamma functions given by
\begin{equation}\label{S83E4}
\psi^{(-1)}(z) = \log \Gamma (z) \mbox{ and } \psi^{(k)}(z) = \frac{d^{k+1}}{dz^{k+1}} \psi^{(-1)}(z), \mbox{ for $k \geq 0$}.
\end{equation}

We consider in this section random walk bridges as in the setup of Section \ref{Section2.1}, whose jump has distribution $X = \frac{\xi - m_{\gamma}}{\sigma_\gamma}$. To indicate the dependence of the bridges on $\gamma$ we write $S^{(n,z)}_\gamma$ to denote a process whose law is given by that of a random walk bridge with step distribution $X$ and which is condititioned to end at $z$ after $n$ steps. The main result we wish to establish is the following.
\begin{corollary}\label{CorLogGamma}
 For any $b > 0 $ and $\gamma_0 > 0$ there exist constants $0<C,a,\alpha'<\infty$ such that for every positive integer $n$ and $\gamma \geq \gamma_0$, there is a probability space on which are defined a Brownian bridge $B^{\sigma}$ with $\sigma = 1$ and a family of processes $S^{(n,z)}_\gamma$ for $z\in\mathbb{R}$ such that
\begin{equation}\label{S83E5}
\mathbb{E}[e^{a \Delta(n,z)}]\leq Ce^{\alpha'(\log n)} e^{bz^2/n},
\end{equation}
where $\Delta(n,z)=\sup_{0\leq t\leq n} |\sqrt{n}B_{t/n}+\frac{t}{n}z-{S}^{(n,z)}_{\gamma,t}|.$
\end{corollary}

In the remainder of this section we provide the proof of Corollary \ref{CorLogGamma}. The goal is to show that the density
\begin{equation}\label{S83E6}
f_X(x) = \frac{\sigma_\gamma}{\Gamma(\gamma)}\exp\left( \gamma(\sigma_\gamma x+m_\gamma )-e^{\sigma_\gamma x+m_\gamma} \right)
\end{equation}
satisfies Assumptions C1-C6 and that the constants in Definition \ref{DefParC} and the functions in Assumption 6 can be chosen uniformly in $\gamma \geq \gamma_0$. If true then Corollary \ref{CorLogGamma} will follow from Theorem \ref{ContKMTA} applied to $p = 0$ and $\epsilon' = 1$. For clarity we split the proof into several steps and use the same notation as in Section \ref{Section2.1}.
.\\

{\bf \raggedleft Step 1.} In this step we summarize several statements that we will need throughout the proof. 

From (\ref{S83E2}) we have 
\begin{equation}\label{S83MGF}
M_X(t) = e^{- m_\gamma t/ \sigma_\gamma}\frac{\Gamma(\gamma + t/ \sigma_\gamma)}{\Gamma(\gamma)} \mbox{ and } \Lambda(t) = \log [ M_X(t)] = \psi^{(-1)}\left(\gamma+\frac{t}{\sigma_\gamma}\right)-\psi^{(-1)}(\gamma)- \frac{m_\gamma t}{\sigma_\gamma},
\end{equation}
Using (\ref{S83E6}) we have
\begin{equation}\label{S83L2E1}
\frac{d}{dx}\log f_X(x) =\sigma_\gamma \left( \gamma-e^{ \sigma_\gamma x+m_\gamma} \right) \mbox{ and } \frac{d^2}{dx^2}\log f_X(x)=-\sigma_\gamma^2e^{m_\gamma}\cdot e^{\sigma_\gamma x}.
\end{equation}

From \cite[Lemma 3]{GQ} we have for $x > 0$
\begin{equation}\label{S83EIneq}
\begin{split}
\log(x) - \frac{1}{x} \leq &\psi^{(0)}(x) \leq \log(x) - \frac{1}{2x} \\
 \frac{(k-1)!}{x^k}+ \frac{k!}{2x^{k+1}} \leq &\psi^{(k)}(x) \leq \frac{(k-1)!}{x^k}+ \frac{k!}{x^{k+1}}  \mbox{ for $k \in \mathbb{N}$}.
\end{split}
\end{equation}
Using (\ref{S83EIneq}) and \cite[(6.3.18)]{AS} we know that 
\begin{equation}\label{S83EAsympt}
\begin{split}
\sigma_\gamma = \gamma^{-1/2} + O(\gamma^{-1}) \mbox{ and } m_\gamma = \log \gamma - \frac{1}{2\gamma} + O(\gamma^{-2}) \mbox{ as $\gamma \rightarrow \infty$}.
\end{split}
\end{equation}
We have the following series representation for $\psi^{(0)}(z)$ for $z \neq 0, -1, -2, \dots$, see e.g. \cite[6.3.16]{AS}, 
\begin{equation}\label{S83L1E1}
\psi^{(0)}(z) = - \gamma_E + \sum_{n = 0}^\infty \left[ \frac{1}{n+1} - \frac{1}{n+z} \right],
\end{equation}
where $\gamma_E$ is the Euler constant. \\

{\bf \raggedleft Step 2.} In this step we demonstrate that $f_X(\cdot)$ satisfies Assumptions C1-C5. 

From (\ref{S83E6}) we know that Assumption C1 holds with $\alpha = -\infty$ and $\beta = \infty$. In addition, from (\ref{S83MGF}) we know that Assumption C2 holds for any $0 < \lambda < \sigma_\gamma \cdot \gamma$, in particular it holds when $\lambda = 2^{-1} \cdot  \sigma_\gamma \cdot \gamma$. We have $\mathcal{D}_\Lambda = (- \gamma \sigma_\gamma, \infty)$ and $\Lambda(\cdot)$ is lower semi-continuous on $\mathbb{R}$. This verifies Assumption C3.

We isolate the verification of Assumption C4 in the following lemma.
\begin{lemma}\label{log_gamma_C4}
For any $\gamma >0$  and $- \sigma_\gamma \cdot \gamma <S<T<\infty$ there is a $K_1(S,T,\gamma) > 0$ such that 
\begin{equation}\label{S83E7}
|M_X (z)|\leq \frac{K_1}{1+|v|}, \mbox{ where $z = u + i v$ with $s \leq u \leq t$. }
\end{equation}
\end{lemma}
\begin{proof}
From (\ref{S83L1E1}) we have
$$|M_X (z)| = |M_X(u)|  \cdot \left|\frac{M_X(z)}{M_X(u)}\right|= |M_X(u)|  \exp \left( \int_0^v \sum_{n = 0}^\infty Re \left[   \frac{i}{n+ 1 } -  \frac{i}{n+ \gamma + (u+iy)/\sigma_\gamma } \right]dy\right).$$
We observe that
$$Re \left[   \frac{i}{n+ 1 } -  \frac{i}{n+ \gamma + (u+iy)/\sigma_\gamma } \right] = \frac{-y}{(n+ \gamma + u/\sigma_\gamma)^2 + y^2}.$$
Combining the last two statements we see
\begin{equation}\label{S83L1E2}
|M_X (z)|  \leq  |M_X(u)|\cdot \exp \left( \int_0^v \frac{- y\cdot dy}{ [ a^2 + y^2]}\right) = \frac{|M_X(u)|}{\sqrt{v^2 + a^2}},
\end{equation}
where $a = \gamma + u/\sigma_\gamma$. The last line proves (\ref{S83E7}).
\end{proof}

In view of (\ref{S83E7}) we conclude that $f_X$ satisfies Assumption C4. We next verify Assumption C5.
\begin{lemma}\label{S83L2}
For any $\gamma_0 > 0$ there exist constants $L, D,d > 0$ such that  
\begin{equation}\label{S83E8}
f_X(x)\leq L \mbox{ for all $x \in \mathbb{R}$ and } f_X(x)\leq De^{-d x^2} \mbox{ for all $x \geq 0$}.
\end{equation}
\end{lemma}

\begin{proof}
From (\ref{S83L2E1}) we know that $f_X$ is log-concave and has a unique maximum when $x = x_c = \sigma_\gamma^{-1} \cdot [\log (\gamma) - m_\gamma ]$. In particular, this implies that 
$$f_X(x) \leq f_X(x_c) =  \frac{\sigma_\gamma}{\Gamma(\gamma)}\exp\left( \gamma \log(\gamma) -\gamma \right).$$
The right side above is uniformly bounded on $[\gamma_0, M]$ for any finite $M$, and as $\gamma \rightarrow \infty$ we have by Stirling's approximation formula (see e.g. \cite[6.1.37]{AS}) and (\ref{S83EAsympt}) that
$$\frac{\sigma_\gamma}{\Gamma(\gamma)}\exp\left( \gamma \log(\gamma) -\gamma \right) \sim \frac{1}{\sqrt{2\pi}} \mbox{ as $\gamma \rightarrow \infty$}.$$
Overall we conclude that we can find $L$ sufficiently large depending on $\gamma_0$ alone so that the left inequality in (\ref{S83E8}) holds.

We next fix $x \geq 0$. We have
$$\frac{f_X(x)}{f_X(0)} = \exp \left( \gamma \sigma_\gamma x - e^{\sigma_\gamma x+ m_\gamma} + e^{m_\gamma} \right) \leq \exp \left(- \frac{e^{m_\gamma} \sigma_\gamma^2}{2} x^2\right), $$
where in the last inequality we used that $e^a \geq 1 + a + \frac{a^2}{2}$ for $a \geq 0$. We observe by (\ref{S83EIneq}) that
$$\frac{e^{m_\gamma} \sigma_\gamma^2}{2} \geq \frac{1}{2} e^{-1/\gamma} ,$$
and so we conclude that 
$$f_X(x) \leq f_X(0) \cdot \exp\left({-e^{-1/\gamma_0} \cdot x^2/2}\right).$$
This proves the right inequality in (\ref{S83E8}) with $D = L$ and $d = e^{-1/\gamma_0}/2$.
\end{proof}

{\bf \raggedleft Step 3.} In what follows we fix $-\infty < s < t < \infty$ and set $S_\gamma = u_s = (\Lambda')^{-1}(s)$ and $T_\gamma = u_t = (\Lambda')^{-1}(t)$. We write below $C(\gamma_0, s,t)$ to mean a generic positive constant that depend on $s,t$ and $\gamma_0$, whose value may change from line to line. The goal of this step is to show 
\begin{equation}\label{S83E11}
\begin{split}
\gamma + S_\gamma \sigma_\gamma^{-1} \geq C(s,t,\gamma_0) \cdot \gamma \mbox{ and } \gamma + T_\gamma \sigma_\gamma^{-1} \leq C(s,t,\gamma_0) \cdot \gamma.
\end{split}
\end{equation}

From (\ref{S83MGF}) we know that 
\begin{equation}\label{S83E9}
\Lambda'(S_\gamma) = \frac{\psi^{(0)}( \gamma + S_\gamma \sigma_\gamma^{-1}) - \psi^{(0)}(\gamma)}{\sigma_\gamma} = s \mbox{ and }\Lambda'(T_\gamma) = \frac{\psi^{(0)}( \gamma + T_\gamma \sigma_\gamma^{-1}) - \psi^{(0)}(\gamma)}{\sigma_\gamma} =t.
\end{equation}
Combining (\ref{S83E9}) and (\ref{S83EIneq}) we conclude that
\begin{equation}\label{S83E10}
\begin{split}
&\log \left[ \gamma + S_\gamma \sigma_\gamma^{-1}  \right] - \log [\gamma] - \frac{1}{2(\gamma + S_\gamma \sigma_\gamma^{-1})} + \frac{1}{\gamma} \geq \sigma_\gamma \cdot s\\
& \log\left[ \gamma + T_\gamma \sigma_\gamma^{-1} \right] - \log [\gamma] - \frac{1}{\gamma + T_\gamma \sigma_\gamma^{-1}} + \frac{1}{2\gamma} \leq \sigma_\gamma \cdot t.
\end{split}
\end{equation}
From the first line in (\ref{S83E10}) we see that
$$\log \left[ \gamma + S_\gamma \sigma_\gamma^{-1}  \right]  \geq\log [\gamma] + \sigma_\gamma \cdot s - \frac{1}{\gamma_0}.$$
Exponentiating both sides above and using (\ref{S83EAsympt}) we get the left part of (\ref{S83E11}).

On the other hand, from the second line in (\ref{S83E10}) we have
$$\log\left[ \gamma + T_\gamma \sigma_\gamma^{-1} \right]  \leq \log [\gamma] + \sigma_\gamma \cdot t + \frac{1}{\gamma + T_\gamma \sigma_\gamma^{-1}}.$$
Using the left part of (\ref{S83E11}) we have $\gamma + T_\gamma \sigma_\gamma^{-1} \geq \gamma + S_\gamma \sigma_\gamma^{-1} \geq C(s,t,\gamma_0) \cdot \gamma$ and so if we exponentiate  both sides of the above equation we conclude the right side of (\ref{S83E11}).\\

{\bf \raggedleft Step 4.} In this step we show that we can find $\infty > M_{s,t} > m_{s,t} > 0$ that depend on $s, t$ and $\gamma_0$ alone such that if $\gamma \geq \gamma_0$ and $x \in [S_\gamma, T_\gamma]$ we have
\begin{equation}\label{S83E13}
M_{s,t} \geq \Lambda''(x) \geq m_{s,t}.
\end{equation}

From (\ref{S83EIneq}) we have that for $x \in [S_\gamma, T_\gamma]$ 
\begin{equation*}
\begin{split}
& \frac{1}{\sigma_\gamma^2}  \cdot \left[\frac{1}{\gamma + S_\gamma\sigma_\gamma^{-1}} + \frac{1}{ (\gamma + S_\gamma\sigma_\gamma^{-1})^2}\right] \geq  \Lambda''(S_\gamma) \geq \Lambda''(x) = \frac{1}{\sigma_\gamma^2} \cdot \psi^{(1)} \left( \gamma + x \sigma_\gamma^{-1} \right)\geq \\
&\Lambda''(T_\gamma) \geq  \frac{1}{\sigma_\gamma^2}  \cdot \left[ \frac{1}{\gamma + T_\gamma\sigma_\gamma^{-1}} + \frac{1}{2 (\gamma + T_\gamma\sigma_\gamma^{-1})^2} \right].
\end{split}
\end{equation*}
The above inequalities together with (\ref{S83E11}) and (\ref{S83EAsympt}) imply (\ref{S83E13}).\\

{\bf \raggedleft Step 5.} We have from (\ref{S83EAsympt}) and (\ref{S83E11}) that there is $\delta^1_{s,t} \in (0,1)$ sufficiently small depending on $s,t$ and $\gamma_0$ such that
\begin{equation}\label{S83E13.5}
\gamma + \min (S_\gamma, 0) \cdot \sigma_\gamma^{-1} \geq 2 \delta^1_{s,t} \cdot \sigma_{\gamma}^{-1}. 
\end{equation}
We fix such a $\delta_{s,t}^1$ and denote $S_\gamma' = S_\gamma - \delta_{s,t}^1$, and $T_\gamma' = T_\gamma + \delta_{s,t}^1$. Notice that if $D_{\delta_{s,t}^1}(\min (0, S_\gamma), \max(T_\gamma, 0))$ is as in Definition \ref{DefDelta} then $\overline{D}_{\delta^1_{s,t}} \subset  \{z \in \mathbb{C}: - \gamma \cdot \sigma_\gamma < Re(z) < \infty \}$. In this step we show that we can find $\hat{M}(s,t, \gamma_0) > 0$, depending on $s,t$ and $\gamma_0$, such that
\begin{equation}\label{S83E19}
\left|\Lambda(z)  \right| \leq \hat{M}_0(s,t, \gamma_0) \mbox{ for all } z\in \overline{D}_{\delta_{s,t}^1}(\min (0, S_\gamma), \max(T_\gamma, 0)).
\end{equation}

From (\ref{S83MGF}) and (\ref{S83L1E1}) we have for $x \in (-\gamma \cdot \sigma_\gamma, \infty)$ that
$$\Lambda'(x) = \frac{1}{\sigma_\gamma} \cdot \left[ \psi^{(0)}(\gamma + x \sigma_\gamma^{-1}) - \psi^{(0)}(\gamma) \right] \mbox{ and } \Lambda''(x) = \frac{1}{\sigma_\gamma^2} \cdot \sum_{n = 0}^\infty \frac{1}{(n + \gamma + x \sigma_\gamma^{-1})^2} > 0, $$
which implies that $x = 0$ is the unique minimizer of $\Lambda(x)$ and the maximum of this function on $[S_\gamma', T_\gamma']$ is obtained either when $x = S_\gamma'$ or $x = T_\gamma'$. Furthermore, it follows from (\ref{S83EAsympt}), (\ref{S83E11}) and (\ref{S83E10}) that there is a sufficiently large positive constant $\hat{C}(s,t,\gamma_0) > 0$ such that 
\begin{equation}\label{S83E14}
\hat{C}(s,t, \gamma_0)  \geq T_{\gamma}' > S_\gamma'  \geq - \hat{C}(s,t, \gamma_0).
\end{equation}
Combining (\ref{S83E14}) with (\ref{S83E13.5}) and (\ref{S83EIneq}) we conclude that there is a sufficiently large positive constant $\hat{M}_1(s,t, \gamma_0) > 0$ such that for $x \in [\min (0, S_\gamma'), \max(T_\gamma', 0)]$ we have
\begin{equation}\label{S83E15}
\left|\Lambda'(x)  \right| \leq \hat{M}_1(s,t, \gamma_0).
\end{equation}
Combining (\ref{S83E14}) and (\ref{S83E15}) with the fact that $\Lambda(0) = 0$ we conclude that there is a sufficiently large constant $\hat{M}_0(s,t, \gamma_0) > 0$ such that for $x \in [\min (0, S_\gamma'), \max(T_\gamma', 0)]$ we have
\begin{equation}\label{S83E16}
\left|\Lambda(x)  \right| \leq \hat{M}_0(s,t, \gamma_0).
\end{equation}

Now we suppose that $x \in [\min (0, S_\gamma'), \max(T_\gamma', 0)]$ and note that
\begin{equation} \label{S83E17}
 \Lambda'(x + i y) = \frac{1}{\sigma^2_\gamma} \cdot \sum_{n = 0}^\infty \frac{\sigma_\gamma^{-1} y^2 + x (n + \gamma + x \sigma_\gamma^{-1})}{(n+\gamma) [(n+\gamma + x\sigma_\gamma^{-1})^2 + \sigma_\gamma^{-2} y^2]} + \frac{i}{\sigma^2_\gamma} \sum_{n = 0}^\infty \frac{y}{(n+\gamma + x \sigma_\gamma^{-1})^2 + \sigma_\gamma^{-2}y^2}.
\end{equation}
where we used (\ref{S83L1E1}). In particular, we see that 
\begin{equation*}
\begin{split}
&\frac{1}{\sigma^2_\gamma} \cdot \sum_{n = 0}^\infty \frac{\sigma_\gamma^{-1} y^2 + |x  | (n + \gamma + x\sigma_\gamma^{-1})}{(n+\gamma) [(n+\gamma + x\sigma_\gamma^{-1})^2 + \sigma_\gamma^{-2} y^2]} \leq \frac{1}{\sigma^3_\gamma \cdot \gamma } \cdot \sum_{n = 0}^\infty \frac{ y^2 }{(n+\gamma + \sigma_\gamma^{-1} x)^2} + \\
&+ \frac{1}{\sigma^2_\gamma} \cdot \sum_{n = 0}^\infty \frac{|x| }{(n+\gamma + x\sigma_\gamma^{-1})^2} + \frac{x^2}{\sigma_\gamma^3 \cdot \gamma}  \cdot \sum_{n = 0}^\infty \frac{1 }{(n+\gamma + \sigma_\gamma^{-1} x)^2} 
\end{split}
\end{equation*}
and also 
\begin{equation*}
\begin{split}
\frac{1}{\sigma^2_\gamma} \sum_{n = 0}^\infty \frac{|y|}{(n+\gamma + x \sigma_\gamma^{-1})^2 + y^2} \leq \frac{1}{\sigma^2_\gamma} \sum_{n = 0}^\infty \frac{|y|}{(n+\gamma + x \sigma_\gamma^{-1})^2} .
\end{split}
\end{equation*}
We use that 
$$\sum_{n = 0}^\infty\frac{1}{(n+\gamma + x \sigma_\gamma^{-1})^2} \leq \frac{1}{(\gamma + x\sigma_\gamma^{-1})^2} + \int_{0}^\infty \frac{1}{(\gamma + x \sigma_\gamma^{-1} + u)^2}du = \frac{1}{(\gamma + x\sigma_\gamma^{-1})^2} +  \frac{1}{\gamma + x\sigma^{-1}_\gamma}.$$
Substituting the above inequalities into (\ref{S83E17}) we get for $x \in [\min (0, S_\gamma'), \max(T_\gamma', 0)]$
$$ | \Lambda'(x + i y) | \leq  \left[\frac{1}{(\gamma + x\sigma_\gamma^{-1})^2} +  \frac{1}{\gamma + x\sigma^{-1}_\gamma} \right] \cdot \left[\frac{y^2}{\sigma^3_\gamma \cdot \gamma} + \frac{|x|}{\sigma_\gamma^2} + \frac{x^2}{\sigma_\gamma^3 \cdot \gamma} + \frac{|y|}{\sigma^2_\gamma} \right].$$
From (\ref{S83E11}) we have $\gamma +   S_\gamma' \sigma_\gamma^{-1} \geq C(s,t,\gamma_0) \cdot \gamma$ and so the above inequality implies
$$ | \Lambda'(x + i y) | \leq \frac{C(s,t,\gamma_0)}{\gamma} \cdot \left[\frac{y^2}{\sigma^3_\gamma \cdot \gamma} + \frac{|x|}{\sigma_\gamma^2} + \frac{x^2}{\sigma_\gamma^3 \cdot \gamma} + \frac{|y|}{\sigma^2_\gamma} \right].$$
If we finally combine the latter with (\ref{S83E14}) and (\ref{S83EIneq}) we see that 
\begin{equation}\label{S83E18}
| \Lambda'(x + i y) | \leq C(s,t,\gamma_0) \cdot [ 1 + y^2].
\end{equation}
In view of  (\ref{S83E16}) and (\ref{S83E18}) we know that by possibly making $\hat{M}_0(s,t, \gamma_0)$ larger we can ensure that (\ref{S83E19}) holds.\\

{\bf \raggedleft Step 6.} In this step we show that we can choose the constants in Definitions \ref{DefDelta} and \ref{DefK} uniformly in $\gamma \geq \gamma_0$. We fix $m_{s,t}$ and $M_{s,t}$ as in (\ref{S83E13}) above. From (\ref{S83E19}) and the fact that $x = 0$ is the unique minimizer of $\Lambda(x)$ on $[\min (0, S_\gamma'), \max(T_\gamma', 0)]$ we get 
\begin{equation}\label{S83E20}
e^{\hat{M}_0(s,t, \gamma_0)} \geq  M_X(x) \geq 1.
\end{equation}
Also we have
\begin{equation*}
\begin{split}
\left| M_X(x) - M_X(x+iy)\right| = M_X(x) \cdot \left |1 - \exp\left( \int_0^y i \Lambda'(x+iu) du \right) \right| \leq C(s,t,\gamma_0) \cdot |y|.
\end{split}
\end{equation*}
The latter implies that we can pick $0 < \delta_{s,t} \leq \delta_{s,t}^1$ sufficiently small depending on $s,t$ and $\gamma_0$ so that
\begin{equation}\label{S83E21}
8 \delta_{s,t} \cdot \hat{M}_0(s,t, \gamma_0) < m_{s,t} \mbox{ and } \left| M_X(x) - M_X(x+iy)\right| < 1/2.
\end{equation}
 In particular, the latter together with (\ref{S83E19}) and (\ref{S83E20}) imply that for $z \in \overline{D}_{\delta_{s,t}}(S_\gamma, T_\gamma)$ we have $Re[M_X(z)] \geq 1/2$ and $8 \delta_{s,t} \cdot \left|\Lambda(z)  \right| < m_{s,t}$. Thus $\delta_{s,t}$ satisfies the conditions in Definition \ref{DefDelta}.\\

 Note that by (\ref{S83L1E2}) we have
\begin{equation}
|M_X (x+ i y)|  \leq  |M_X(x)|\cdot \exp \left( \int_0^v \frac{- u\cdot du}{ [ a^2 + u^2]}\right) = \frac{|M_X(x)|}{\sqrt{y^2 + a^2}},
\end{equation}
where $a = \gamma + x \cdot \sigma_\gamma^{-1}$. Combining the latter with (\ref{S83E11}) we conclude that there is $K_{s,t}$ depending on $s,t$ and $\gamma_0$ such that for all $x \in [\min (0, S_\gamma'), \max(T_\gamma', 0)]$ we have
$$ \left| M(x + iy) \cdot e^{-\Lambda'(x) \cdot (x+iy)} e^{- \Lambda(x)+ x \Lambda'(x)} \right| \leq \frac{1}{\sqrt{y^2 + (\gamma + \min (0, S_\gamma') \cdot \sigma_\gamma^{-1})^2}}  \leq \frac{K_{s,t}}{1 + |y|}.$$
This fixes $K_{s,t}$ in Definition \ref{DefK} and $p_{s,t} = 1$. \\

{\bf \raggedleft Step 7.} In this step we show that we can choose $q_{s,t}$ in Definition \ref{DefQ} uniformly in $\gamma \geq \gamma_0$. 

Let $\epsilon_{s,t}$ and $R_{s,t}$ be as in the statement of Definition \ref{DefQ} for the constants $\delta_{s,t}$ and $K_{s,t}$ in Step 6. In view of (\ref{S83E17}) we have for any $x \in [S_\gamma, T_\gamma]$ that
\begin{equation}\label{S83E22} 
\frac{d}{dy}Re [\Lambda(x+iy)] =  \frac{1}{\sigma^2_\gamma} \sum_{n = 0}^\infty \frac{-y}{(n+\gamma + x \sigma_\gamma^{-1})^2 + \sigma_\gamma^{-2} y^2},
\end{equation}
which implies that $Re[ \Lambda(x+iy)]$ is decreasing in $y$ on $[0, \infty)$ and increasing in $y$ on $(-\infty, 0)$. Let us first consider the case $y\geq \epsilon_{s,t}$. The above inequality implies  that 
\begin{equation*}
\begin{split}
&Re[\Lambda(x+ i y) ] - \Lambda(x) \leq Re[\Lambda(x+ i \epsilon_{s,t})] - \Lambda(x) \leq  \int_0^{\epsilon_{s,t}} \sum_{n = 0}^\infty \frac{-u \sigma_{\gamma}^{-2} du}{(\gamma + S_\gamma \sigma_\gamma^{-1} + n)^2  } =  \\
& =  - \frac{\epsilon_{s,t}^2}{2\sigma_\gamma^2}\sum_{n = 0}^\infty \frac{1}{(\gamma + S_\gamma \sigma_\gamma^{-1} + n)^2  }\leq  - \frac{\epsilon_{s,t}^2}{2\sigma_\gamma^2} \cdot \int_1^\infty \frac{dv}{(\gamma + S_\gamma \sigma_\gamma^{-1} + v)^2  } = \frac{-\epsilon_{s,t}^2}{2\sigma_\gamma^2 (\gamma + S_\gamma \sigma_\gamma^{-1} + 1)}.
\end{split}
\end{equation*}
Combining the latter with (\ref{S83EAsympt}) and (\ref{S83E11}) we conclude that there is $q_{s,t} \in (0,1)$ that depends on $s,t$ and $\gamma_0$ such that 
$$Re[\Lambda(x+ i y) ] - \Lambda(x) \leq  \log[ q_{s,t}],$$
In particular, exponentiating both sides we see that for $x \in [S_\gamma, T_\gamma]$ and $y \geq \epsilon_{s,t}$ we have
\begin{equation}\label{S83E23}
 \left|\frac{M_X(x+ i y)}{M_X(x)} \right| \leq q_{s,t}.
\end{equation}
Since $|M_X(x+ i y)| = |M_X(x - iy)|$ we conclude that (\ref{S83E23}) holds for $|y| \geq \epsilon_{s,t}$, which verifies that $q_{s,t}$ satisfies the conditions in Definition \ref{DefQ}.\\

{\bf \raggedleft Step 8.} In this step we show that we can choose the constants in Definition \ref{DefM34} uniformly in $\gamma \geq \gamma_0$. We first show that we can find constants $\hat{M}_{k}(s,t,\gamma_0) > 0$ for $k = 0,1,2,3,4$ such that
\begin{equation}\label{S83E24}
| \Lambda^{(k)}(x)| \leq \hat{M}_{k}(s,t, \gamma_0) \mbox{ for all $x \in [S_\gamma, T_\gamma]$}.
\end{equation}
Indeed for $k = 0$, $k = 1$ and $k = 2$ this follows from (\ref{S83E16}), (\ref{S83E15}) and (\ref{S83E13}) respectively. Next we have for $k = 3,4$ that
$$| \Lambda^{(k)}(x)| = \frac{1}{\sigma_\gamma^{k}}\left| \psi^{(k-1)} ( \gamma + S_\gamma \sigma_\gamma^{-1}) \right| \leq \frac{1}{\sigma_\gamma^{k}} \cdot \left[ \frac{(k-2)!}{[\gamma + x\sigma_\gamma^{-1}]^{k-1}}+ \frac{(k-1)!} {[\gamma + S_\gamma \sigma_\gamma^{-1}]^{k}} \right],$$
where in the last inequalitly we used (\ref{S83EIneq}). Using (\ref{S83E11}) and (\ref{S83EAsympt}) we conclude (\ref{S83E24}) for $k = 3$ and $k = 4$ as well.

Next we recall that $F(z) = G_z(u_z) = \Lambda(u_z) - u_z \cdot z$. We claim that for $k = 0, 1,2,3,4$ we can find constants $M^{(k)}_{s,t}$ that depend on $s,t$ and $\gamma_0$ such that if $z \in [s,t]$ we have
\begin{equation}\label{S83E25}
| F^{(k)}(x)| \leq M^{(k)}_{s,t} \mbox{ for all $x \in [S_\gamma, T_\gamma]$}.
\end{equation}
If $z \in [s,t]$ then $u_z \in [S_\gamma, T_\gamma]$ and then in view of (\ref{S83E19}) and (\ref{S83E14}) we can find $M^{(0)}_{s,t}$ satisfying (\ref{S83E25}). We next use that $u_z = (\Lambda')^{-1}(z)$ to get
$$F'(z) = -u_z, \hspace{2mm} F''(z) = - \frac{1}{\Lambda''(u_z)} \hspace{2mm} F^{(3)}(z) = \frac{\Lambda^{(3)}(u_z)}{[\Lambda''(u_z)]^3} \hspace{2mm} F^{(4)}(z) = \frac{\Lambda^{(4)}(u_z) \cdot \Lambda''(u_z) - 3 \Lambda^{(3)}(u_z)}{[\Lambda''(u_z)]^5}.$$
The latter equalities together with (\ref{S83E24}) and (\ref{S83E13}) prove (\ref{S83E25}). The constants in (\ref{S83E25}) satisfy the conditions in Definition \ref{DefM34}.\\

{\bf \raggedleft Step 9.} In this step we show that we can choose the constants in Definitions \ref{DefN2} and \ref{DefParC} uniformly in $\gamma \geq \gamma_0$. Observe that by Steps 6. and 7. we can choose the constant $N_0$ in Proposition \ref{S3S1P1} depending on $s,t$ and $\gamma_0$ alone and the same is true for the constant $C_1$. Since $D,d$ and $L$ in Assumption C5 were chosen uniformly in Lemma \ref{S83L2} in Step 2. we conclude that we can pick $R_1$ in Definition \ref{DefN2} depending on $s,t$ and $\gamma_0$ alone. We now let $\hat{s} = -6R_1$ and $\hat{t} = 6R_1$. Then from Steps 6. and 7. we can pick all the remaining constants in Definition \ref{DefParC} uniformly in $\gamma \geq \gamma_0$.  \\

{\bf \raggedleft Step 10.}  In this step we show that for any $r > 0$ there is a constant $\Delta_0 > 0$ that depends on $r$ and $\gamma_0$ alone such that 
\begin{equation}\label{S83E26}
\inf_{x \in [-r,r]} f_X(x) \geq  \Delta_0. 
\end{equation}
We begin by proving a useful lemma.

\begin{lemma}\label{S83L3} The function $f_\gamma(x)$ converges uniformly over compact sets to $\phi(x) = \frac{e^{-x^2/2}}{\sqrt{2\pi}}$ as $\gamma \rightarrow \infty$.
\end{lemma}
\begin{proof}
Let us fix $R > 0$ and assume $x \in [-R,R]$. The functional equation $\Gamma(z+1) = z \Gamma(z)$ and \cite[Theorem 1.6]{Batir} give 
$$ \frac{F_\gamma(x) \cdot \gamma^{- \gamma +1/2} e^{\gamma} }{\sqrt{2\pi}}\cdot \frac{\sqrt{\gamma}}{\sqrt{\gamma + 1}}\leq f_\gamma(x) \leq \frac{F_\gamma(x) \cdot \gamma^{- \gamma + 1/2} e^{\gamma} }{\sqrt{2\pi}}, $$
where $F_\gamma(x) = \sigma_\gamma \cdot \exp( \gamma(\sigma_\gamma x + m_\gamma) - e^{\sigma_\gamma x + m_\gamma}).$ In addition, we have from (\ref{S83EIneq}) that
$$\gamma(\sigma_\gamma x + m_\gamma) - e^{\sigma_\gamma x + m_\gamma} = -\frac{x^2}{2} -e^{m_\gamma} + \gamma m_\gamma + O(\gamma^{-1/2}),$$
where the constant in the big $O$ notation depends on $R$. Combining the latter with (\ref{S83EAsympt}) we see that we can find a constant $C > 0$ depending on $R$ such that
$$ \frac{e^{-x^2/2 - C \gamma^{-1/2}} \cdot \sigma_\gamma \gamma^{1/2}  }{\sqrt{2\pi}}\cdot \frac{\sqrt{\gamma}}{\sqrt{\gamma + 1}}\leq f_\gamma(x) \leq \frac{e^{-x^2/2 + C \gamma^{-1/2}} \cdot \sigma_\gamma \gamma^{1/2}  }{\sqrt{2\pi}}, $$
from which we conclude the statement of the lemma after applying (\ref{S83EAsympt}).
\end{proof}

Let us fix $r > 0$. By Lemma \ref{S83L3} we know that there is $\gamma_1 \geq \gamma_0$, depending on $r$, such that if $\gamma \geq \gamma_1$ then 
$$
\inf_{x \in [-r,r]} f_X(x) \geq \frac{1}{2 \sqrt{2\pi}} \cdot e^{-r^2/2}.
$$
 Then since $f_X(x)$ is jointly continuous in $x$ and $\gamma$ and positive on $[-r,r] \times [\gamma_0, \gamma_1]$ there exists a positive constant $\Delta_1$ depending on $\gamma_0$ and $r$ such that 
\begin{equation}\label{S83E28}
\inf_{x \in [-r,r]} f^\gamma_1(x) \geq \Delta_1,
\end{equation}
for all $\gamma \in [\gamma_0, \gamma_1]$. In particular, we deduce that (\ref{S83E26}) holds with $\Delta_0 = \min \left(\Delta_1, \frac{1}{2 \sqrt{2\pi}} \cdot e^{-r^2/2} \right).$\\

{\bf \raggedleft Step 11.} Let us denote $f_n^\gamma(\cdot)$ the density of $S_n = X_1 + \cdots + X_n$ where $X_i$ are i.i.d. with distribution $f_X$. In this step we show that there is a positive constant $\Delta$ that depends on $\gamma_0$ such that
\begin{equation}\label{S83E29}
\inf_{x \in [-1,1]} f^\gamma_n(x) \geq  \Delta \cdot n^{-1/2}. 
\end{equation}

We apply Proposition \ref{S3S1P1} to the distribution $f_X$ and for the values $s = -1$ and $t = 1$. From our work in Steps 6. and 7. we know that we can find $N_0$ and $C_0 > 0$ depending on $\gamma_0$ such that for $N \geq N_0$ we have
$$f^\gamma_N(Nz) \geq \frac{C_0}{\sqrt{2\pi N   \Lambda''(u_z)} } \cdot \exp \left( N G_z(u_z)  \right).$$
In particular, using (\ref{S83E13}), the fact that $G_z(u_0) = 0$ and (\ref{S83E25}) we conclude that there is a constant $\Delta' > 0$ depending on $\gamma_0$ such that 
\begin{equation}\label{S83E30}
\inf_{x \in [-1,1]} f^\gamma_N(x) \geq  \Delta' \cdot N^{-1/2} \mbox{ for $\gamma \geq \gamma_0$ and $N \geq N_0$}.
\end{equation}

Next, we let $\Delta_0 \in (0,1)$ be sufficiently small so that (\ref{S83E26}) holds with $r = N_0$. Then we have for $1 \leq n \leq N_0$ and $x \in [-1,1]$ that 
\begin{equation*}
\begin{split}
&f_n^\gamma(x) = \int_{\mathbb{R}} \cdots \int_{\mathbb{R}}  f_X(x_1)  \cdots f_X(x_{n-1}) \cdot f_X(x - x_1 - \cdots - x_{n-1}) dx_1 \cdots dx_n  \geq \\
& \geq  \int_{0}^1 \cdots \int_{0}^1  f_X(x_1)  \cdots f_X(x_{n-1}) \cdot f_X(x - x_1 - \cdots - x_{n-1}) dx_1 \cdots dx_n \geq (\Delta_0)^n,
\end{split}
\end{equation*}
In particular, we conclude from the latter and (\ref{S83E30}) that (\ref{S83E29}) holds for all $n \geq 1$ with $\Delta = \min (\Delta_0^{N_0}, \Delta')$. \\

{\bf \raggedleft Step 12.} In this and the next step we show that we can choose the constants in Definition \ref{ExpBC} uniformly in $\gamma \geq \gamma_0$. From (\ref{S83E29}) we can choose $\Delta > 0$ depending on $\gamma_0$ alone so that it satisfies the conditions of that definition. We also set $R = 2 + \Delta^{-1}$ in that definition. We may now apply Proposition \ref{S3S1P1} to the distribution $f_X$ for the values $s = -2R$ and $t = 2R$. From our work in Steps 6. and 7. we know that we can find $N_0(R)$ and $C_0(R) > 0$ depending on $\gamma_0$ such that for $N \geq N_0(R)$ we have
$$f_N(Nz) \geq \frac{C_0(R)}{\sqrt{2\pi N   \Lambda''(u_z)} } \cdot \exp \left( N G_z(u_z)  \right).$$
In particular, using (\ref{S83E13}) and (\ref{S83E25}) we conclude there are positive constants $C_R$ and $c_R$ such that
\begin{equation}\label{S83E31}
 f^\gamma_N(Nz) \geq  C_R \cdot N^{-1/2} e^{-c_R N} \mbox{ for $\gamma \geq \gamma_0$, $z \in [-2R,2R]$ and $N \geq N_0(R)$}.
\end{equation}
Furthermore, we can apply (\ref{S83E26}) to $r= 2R + N_0(R)$ to obtain the existence of a positive constant $\Delta_0(R) \in (0,1)$ such that
$$\inf_{x \in [-r,r]} f^\gamma_1(x) \geq \Delta_0(R).$$
Consequently, we have for $z \in [-2R, 2R]$ and $1 \leq n \leq N_0(R)$ that
\begin{equation*}
\begin{split}
&f_n^\gamma(nz) = \int_{\mathbb{R}} \cdots \int_{\mathbb{R}}  f_X(x_1)  \cdots f_X(x_{n-1}) \cdot f_X(nx - x_1 - \cdots - x_{n-1}) dx_1 \cdots dx_n  \geq \\
& \geq  \int_{0}^1 \cdots \int_{0}^1  f_X(x_1)  \cdots f_X(x_{n-1}) \cdot f_X(nx - x_1 - \cdots - x_{n-1}) dx_1 \cdots dx_n \geq (\Delta_0(R))^n.
\end{split}
\end{equation*}
The latter implies that (\ref{S83E31}) continues to hold for $1 \leq N \leq N_0(R)$ as well provided we make $C_R$ small enough (and positive) depending on $\gamma_0$. This fixes the chooice of $\Delta, C_R$ and $c_R$.\\

{\bf \raggedleft Step 13.} As we mentioned in Step 2. Assumption C2 holds for any $\lambda \in (0, \gamma \sigma_\gamma^{-1})$. Consequently, by (\ref{S83EAsympt}) we can find $\lambda_0 > 0$ depending on $\gamma_0$ such that $f_X$  satisfies Assumption C2 for $ \lambda = \lambda_0$ and $\gamma \cdot \sigma_\gamma > 2\lambda_0$. We fix this choice for $\lambda$. Notice that by (\ref{S83MGF}) and (\ref{S83EIneq}) we have for $x \in [-\lambda, \lambda]$ that $|\Lambda'(x) | \leq C(\gamma_0)$
for some $C(\gamma_0) > 0$. The latter and $\Lambda(0) = 0$ imply that  for $x \in [-\lambda, \lambda]$ we have 
\begin{equation}\label{S83E32}
|\Lambda(x)| \leq C(\gamma_0)
\end{equation}
for some possibly different $C(\gamma_0) > 0$. 

 Finally, given $\lambda$ and $\Delta$, $c_R$, $C_R$ as in Step 12. and $L$ as in Lemma \ref{S83L2} we can find positive constants $\hat{C}_R$ and $\hat{c}_R$ that depend on $\gamma_0$ alone such that for all $n\geq 1$
$$ \mathbb{E}[ e^{\lambda |X|}]^n \left[ \frac{4 n^{3/2}}{\Delta}  + LC^{-1}_R \sqrt{n} e^{c_R n}  \right] \leq \hat{C}_R \cdot e^{\hat{c}_R n}.$$
In deriving the last expression we used (\ref{S83E32}) and the simple inequality $\mathbb{E}[ e^{\lambda |X|}] \leq e^{\Lambda(\lambda)} + e^{\Lambda(-\lambda)}$. 

From the proof of Lemma \ref{S83L2} we know that $f_X(x)$ is log-concave and so Lemma \ref{Section7L1} is applicable. From that lemma we conclude that we can find functions $\hat{a}$ and $\hat{C}$ that satisfy the conditions of Assumption C6. Moreover, from the fact that $\lambda, \hat{C}_R$ and $\hat{c}_R$ are all independent of $\gamma$ provided $\gamma \geq \gamma_0$, the lemma implies that the same is true for $\hat{a}$ and $\hat{C}$. \\

Summarizing all of our work in this section, we see that $f_X$ satisfies Assumptions C1-C6 and so we can apply Theorem \ref{ContKMTA} to it. Since the constants $C,a,\alpha'$ in that theorem depend only on the parameters in Definition \ref{DefParC} and the functions in Assumption 6, and the latter can be chosen uniformly in $\gamma \geq \gamma_0$ this implies that the same is true for $C,a,\alpha'$. We conclude that Theorem \ref{ContKMTA} implies Corollary \ref{CorLogGamma}. This suffices for the proof.

\bibliographystyle{amsplain}
\bibliography{PD}

\end{document}